\documentclass[twoside,12pt,openright]{report}  

\usepackage{amsmath}
\usepackage{amssymb}
\usepackage{amsfonts}
\usepackage{pgfplots}
\usepackage{caption}
\usepackage{float} 
\usepackage{amsthm}
\usepackage{algorithm}
\usepackage{algpseudocode}
\usepackage{float} 

\usepackage{subcaption}
\usepackage{array}
\usepackage{xcolor}

\usepackage{graphicx}
\usepackage{doc}
\usepackage{ifthen, color, fancyvrb}
\usepackage{nextpage}
\newcommand\myclearpage{\cleartooddpage
  [\thispagestyle{empty}]
  }

\usepackage[skip=4pt,font=footnotesize]{caption}

\DeclareGraphicsExtensions{ps,eps}
\usepackage{excludeonly}

\newtheorem{theorem}{Theorem}[chapter]

\DeclareMathOperator*{\proj}{\textbf{proj}}
\DeclareMathOperator*{\RR}{\mathbb{R}}
\DeclareMathOperator*{\KKT}{\text{KKT}_\omega}

\def\prefacesection#1{
\chapter*{#1}
\addcontentsline{toc}{chapter}{#1}
}

\begin{document}

\def\thefootnote{\fnsymbol{footnote}}

\thispagestyle{empty}

\def\shiftdowna{0.32in}  
\def\shiftdownb{0.22in}  


\begin{center}
\textbf{{\large Research in Industrial Projects for Students}}\\

\vspace \shiftdowna
\includegraphics[width=0.4\textwidth]{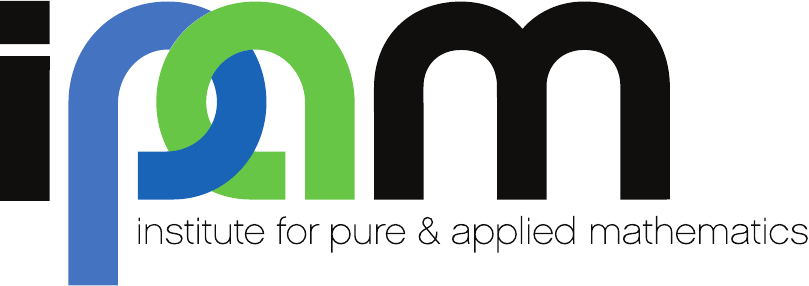}\\

\vspace \shiftdowna
\underline {Sponsor}\\ 
\vspace{5pt}
$\textbf{{\large Advanced Micro Devices, Inc.}}$ \\
\vspace \shiftdowna
\textbf{Final Report}

\vspace \shiftdowna
$\textbf{{\Large Accelerating Linear Programming (LP) Algorithm}}$\\
\vspace{3pt}
$\textbf{{\Large on AMD GPUs}}$


\vspace{0.35in}
\underline {Student Members}\\
\vspace{5pt}
$\text{Connor Phillips}$ (Project Manager), $ \text{\emph{James Madison University}}$,\\ 
\vspace{3pt}
$ \text{\texttt{phill3cm@dukes.jmu.edu}} $\\
\vspace{5pt}
$\text{Xiyan Hu}$, $\text{\emph{Colgate University}}$, $ \text{\texttt{xhu@colgate.edu}}$ \\
\vspace{3pt}
$\text{Titus Parker}$, $\text{\emph{Stanford University}}$, $ \text{\texttt{titus000@stanford.edu}}$ \\
\vspace{3pt}
$ \text{Yifa Yu}$, $\text{\emph{University of California, Davis}}$, $
\text{\texttt{yfayu@ucdavis.edu}}$\\

\vspace \shiftdownb
\underline {Academic Mentor} \\
\vspace{5pt}
$\text{Minxin Zhang}$, $ \text{\texttt{minxinzhang@math.ucla.edu}}$

\vspace \shiftdownb
\underline {Sponsoring Mentors}\\
\vspace{5pt}
$\text{Alireza Kaviani}$, $\text{\texttt{alireza.kaviani@amd.com}}$\\


\vspace \shiftdowna
$\text{Date: August 20, 2025}$ 

\end{center}

\vfill  
\footnoterule
\noindent \small{This project was jointly supported by $\text{Advanced Micro Devices, Inc.}$ and NSF Grant $\text{DMS}$ 1925919.}

\ifthenelse{\boolean{@twoside}}{\myclearpage}{}
\prefacesection{Abstract}

%
%


%

%
%
%
%
%
	%

Linear Programming (LP) is a foundational optimization technique with widespread applications in finance, energy trading, and supply chain logistics. However, traditional Central Processing Unit (CPU)-based LP solvers often struggle to meet the latency and scalability demands of dynamic, high-dimensional industrial environments, creating a significant computational challenge. This project addresses these limitations by accelerating linear programming on AMD Graphics Processing Units (GPUs), leveraging the ROCm open-source platform and PyTorch.

The core of this work is the development of a robust, high-performance, open-source implementation of the Primal-Dual Hybrid Gradient (PDHG) algorithm, engineered specifically for general LP problems on AMD hardware. Performance is evaluated against standard LP test sets and established CPU-based solvers, with a particular focus on challenging real-world instances including the Security-Constrained Economic Dispatch (SCED) to guide hyperparameter tuning. Our results show a significant improvement, with up to a 36x speedup on GPU over CPU for large-scale problems, highlighting the advantages of GPU acceleration in solving complex optimization tasks.

\vspace{24pt}


\ifthenelse{\boolean{@twoside}}{\myclearpage}{}
\prefacesection{Acknowledgments}

We would like to express our sincere gratitude to our academic mentor, Dr. Minxin Zhang, for her invaluable guidance and steadfast mentorship throughout this research project. We are also deeply appreciative of our industry mentor, Dr. Alireza Kaviani, for his crucial support and industry perspective. Furthermore, we wish to extend our special thanks to Dr. Amir Mousavi and Dr. Steven Diamond at Gridmatic. Their insightful feedback and their generosity in providing benchmark data for the real-world problems were instrumental to our work.

This research was also made possible by the support from several organizations. We thank the Institute for Pure and Applied Mathematics (IPAM) for their support, with special thanks to the Research in Industrial Projects for Students (RIPS) program director, Prof. Susana Serna. We are grateful to Advanced Micro Devices, Inc. (AMD) for providing resources through the AI \& HPC Research Cluster, with thanks to Dr. Karl W. Schulz, and for access to the AMD Developer Cloud, with thanks to Dr. Thomas Papatheodore. We also acknowledge the UCLA VAST lab for providing access to the UCLA Heterogeneous Accelerated Compute Clusters.

\ifthenelse{\boolean{@twoside}}{\myclearpage}{}
\tableofcontents

\ifthenelse{\boolean{@twoside}}{\myclearpage}{}
\listoffigures

\ifthenelse{\boolean{@twoside}}{\myclearpage}{}
\listoftables


\renewcommand{\thefootnote}{\arabic{footnote}}
\setcounter{footnote}{0}

\ifthenelse{\boolean{@twoside}}{\myclearpage}{}
\chapter{Introduction}\label{Ch:Introduction}

Advanced Micro Devices (AMD), founded in 1969 as a Silicon Valley start-up, leads high-performance and adaptive computing, powering products and services that help solve the world’s most important challenges. AMD technologies advance the future of data centers, embedded systems, gaming, and PC markets.

Among these challenges, the optimization of industrial systems, particularly within the energy sector, represents a critical application for high-performance computing. Recognizing an opportunity to establish a strong foothold in this domain, AMD is promoting its Graphics Processing Unit (GPU) accelerators and open-source ROCm software platform, shown in Figure~\ref{fig:rocm_stack}, as a powerful alternative to NVIDIA's CUDA ecosystem. A key part of this strategy involves demonstrating superior performance on foundational computational tasks central to these industries, such as linear programming, before customers become locked into a single software library. 

\begin{figure}[h!]
    \centering
    \includegraphics[width=0.8\textwidth]{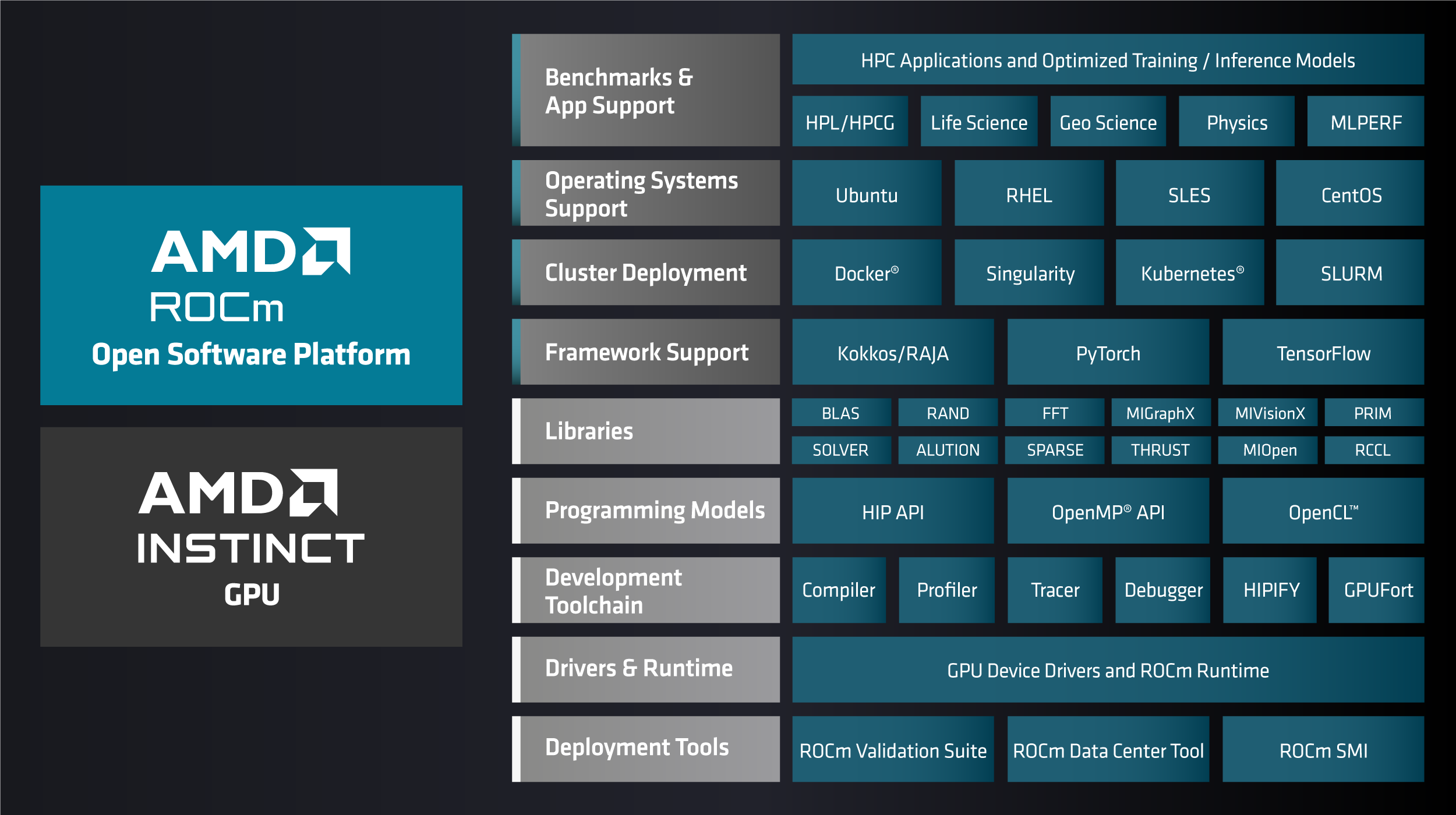}
    \caption{The layered architecture of the AMD ROCm open software platform}
    \label{fig:rocm_stack}
\end{figure}

Linear Programming (LP), the task of optimizing a linear objective function subject to linear inequality constraints, is a cornerstone of mathematical programming with applications spanning finance, energy trading, transportation, and supply chain logistics. Mature CPU-based solvers, such as the Simplex method~\cite{dantzig1948programming} and Interior Point Methods (IPMs)~\cite{dikin1967iterative}, are well established. However, their standard computer implementations rely heavily on matrix factorizations ~\cite{bartels1969simplex,doi:10.1137/S1052623496304712}, which form a computational bottleneck for modern, high-dimensional problems.

Factorization-based approaches present two key challenges for large-scale instances. First, the factorization of a sparse matrix often yields a much denser representation, substantially increasing memory usage and frequently causing out-of-memory errors, even when the original problem data fits comfortably in memory. Second, matrix factorization is inherently sequential, making it difficult to exploit massive parallelism and thus limiting performance gains on modern architectures such as GPUs or distributed systems.

First-order methods, particularly the Primal-Dual Hybrid Gradient (PDHG) algorithm \cite{chambolle2011first}, offer a compelling alternative. By replacing costly factorizations with operations dominated by matrix-vector multiplications, vector updates, and simple projections, PDHG aligns naturally with the highly parallel execution model of GPUs. Recent work by Lu and Yang~\cite{lu2023cupdlp} on cuPDLP.jl, a GPU implementation of restarted PDHG for LP in Julia, has shown that PDHG variants can outperform traditional CPU solvers on some large-scale instances. However, most existing GPU implementations are developed for the NVIDIA CUDA ecosystem, leaving an important gap for solutions optimized for the increasingly competitive AMD GPU hardware.

The objective of this project is to address this gap by developing a high-performance, open-source PDHG solver tailored for AMD GPUs, implemented in PyTorch with ROCm compatibility. The solver is evaluated on both standard LP benchmarks and challenging real-world applications such as Security-Constrained Economic Dispatch (SCED), with the aim of delivering substantial speedups for industrial-scale optimization on open hardware platforms.

\section{Related Work}
Building on the challenges identified in the previous section, our approach begins with a review of the algorithmic landscape for solving LP problems. Over the past decades, two families of algorithms, vertex-based methods such as the simplex algorithm, and interior-point methods, have dominated the field and form the basis of most industrial solvers. Understanding their computational structure, strengths, and limitations is essential for motivating the shift toward first-order methods that are more amenable to parallel hardware.  

\paragraph{Simplex Method.}  
Introduced by George Dantzig in 1947 ~\cite{dantzig1948programming}, the simplex method is widely recognized as the foundation of optimization as a formal scientific discipline. The algorithm proceeds by traversing the vertices of the feasible region, determined by the system of linear constraints, iteratively selecting adjacent vertices that improve the objective value. Despite its exponential worst-case complexity, the simplex algorithm has been extensively refined over the years. The dual simplex method~\cite{dantzig2016linear} addresses modified problem formulations, while the steepest-edge variant~\cite{forrest1992steepest} employs a pivoting rule that selects the direction yielding the greatest improvement per unit step, often reducing iteration counts by guiding the search more effectively. Together, these and other refinements have ensured that simplex variants remain core components of many modern LP solvers.

\paragraph{Interior-Point Methods (IPMs).}  
Interior-point methods arose as a significant alternative to vertex-following techniques such as the simplex method. The conceptual foundations can be traced to Fiacco and McCormick’s work on barrier methods for nonlinear programming~\cite{fiacco1964computational} and Dikin’s iterative interior-point scheme~\cite{dikin1967iterative}. Early formulations, however, gained limited traction due to issues of numerical instability and the considerable computational demands of handling large-scale problems.  
The field advanced with Karmarkar’s introduction of the projective scaling algorithm~\cite{karmarkar1984new}, a polynomial-time method that keeps iterates strictly inside the feasible region. This breakthrough stimulated the development of more effective variants, most notably primal–dual IPMs~\cite{wright1997primal}, which are widely valued for their robust convergence and ability to solve primal and dual problems in tandem.  With their combination of polynomial-time complexity and strong practical performance, IPMs have become a central component of modern optimization software.

\paragraph{Primal–Dual Hybrid Gradient (PDHG) Method.} The Primal–Dual Hybrid Gradient (PDHG) method, originally introduced by Chambolle and Pock in 2011~\cite{chambolle2011first}, is a first-order algorithm for solving saddle-point problems. The term `first order' means that it uses gradient information. The method alternates between gradient-based updates on the primal and dual variables. It applies simple proximal mappings, thereby avoiding costly matrix factorizations that dominate the runtime of simplex and interior-point methods. This structure makes PDHG inherently well-suited for large-scale problems and for execution on massively parallel hardware such as GPUs. In its original form, PDHG is designed for a broad class of convex optimization problems, including image reconstruction and variational imaging tasks. In this work, we adapt the algorithm for general linear programming problems by expressing the LP in a saddle-point form and designing update rules tailored to the polyhedral constraint structure. \\


Each of the three algorithmic paradigms presents a distinct profile of computational trade-offs. While higher-order methods converge in a relatively small number of steps, each iteration involves complex and expensive linear algebra. Conversely, first-order methods are defined by a very large number of computationally inexpensive iterations. These fundamental differences in iteration cost, memory footprint, and amenability to parallel execution are summarized in Table \ref{tab:methods}.
\newcommand{\good}{\textcolor{green}{\large$\bullet$}}
\newcommand{\ok}{\textcolor{gray}{\large$\bullet$}}
\newcommand{\bad}{\textcolor{red}{\large$\bullet$}}
\renewcommand{\arraystretch}{1.4}
\begin{table}[h]
\centering
\renewcommand{\arraystretch}{1.4}

\small
\good~\textsc{Favorable} \quad
\bad~\textsc{Unfavorable} \quad
\ok~\textsc{Neutral} \\
\vspace{0.5em}

\begin{tabular}{>{\raggedright\arraybackslash}p{5cm}|
                >{\centering\arraybackslash}m{2cm}|
                >{\centering\arraybackslash}m{2cm}|
                >{\centering\arraybackslash}m{2cm}}
\hline
 & \textbf{Simplex} & \textbf{IPM} & \textbf{PDHG} \\
\hline
Iteration expense & \good & \bad & \good \\
Steps to reach solution & \bad & \good & \ok \\
Ease of coding & \bad & \ok & \good \\
Parallel scaling potential & \bad & \bad & \good \\
Memory demand & \ok & \bad & \good \\
\hline
\end{tabular}

\caption{Qualitative assessment of \textit{Simplex}, \textit{Interior-Point Method} (IPM), and \textit{Primal-Dual Hybrid Gradient} (PDHG) across key performance aspects.}
\label{tab:methods}
\end{table}
\vspace{-0.75cm}
\paragraph{cuPDLP.jl}
Our design is also informed by the recent work of Lu and Yang~\cite{lu2023cupdlp}, whose introduction marked a significant milestone. Their Julia implementation of a restarted PDHG method was among the first to demonstrate that a GPU-based first-order algorithm could achieve performance competitive with commercial solvers on certain classes of large-scale linear programs. Key to its success are several notable enhancements, including adaptive parameter tuning to dynamically balance step sizes and restart strategies to accelerate convergence when progress stalls. Recognizing that the performance of the Julia implementation could be constrained by language overhead and hardware limitations, a subsequent effort focused on re-implementing and optimizing the algorithm in C~\cite{lu2023cupdlp-c}. This second implementation was designed to maximize computational efficiency on state-of-the-art hardware and provide a more definitive performance benchmark. Drawing inspiration from the principles validated by both of these works, we adopt similar strategies, modifying them where necessary to align with our PyTorch implementation and the specific characteristics of our target problems.

\section{Our approach: torchPDLP}\label{sec: torchPDLP}
For decades, the Central Processing Unit (CPU) has been the primary computational engine for LP solvers. As a general-purpose processor optimized for sequential instruction execution, a CPU typically consists of a small number of cores paired with sophisticated control logic. This architecture excels at tasks with complex dependencies, where low latency and fine-grained control are critical. Simplex methods, for example, exhibit inherently sequential behavior, with each pivot step dependent on the results of the previous one. Likewise, interior-point methods repeatedly solve large linear systems, a process that allows only limited parallelism. These characteristics have led to highly tuned CPU-based implementations that set a strong performance benchmark for alternative platforms.


In recent years, however, high-performance computing has undergone a major shift with the GPUs as a general-purpose compute resource. Originally designed for accelerating graphics rendering, GPUs have evolved into highly parallel, high-throughput architectures capable of executing thousands of lightweight threads simultaneously. This transformation has been driven in large part by the explosive growth of machine learning, where GPUs have become indispensable for training and inference in large-scale neural networks. This hardware paradigm offers a compelling opportunity to rethink LP solver design, particularly for first-order methods such as PDHG, which naturally align with GPU-friendly computation patterns.



A key distinction of our work is the hardware and software ecosystem. Whereas most prior GPU-based PDHG implementations are developed for the NVIDIA CUDA platform with custom kernels, our solver is implemented in Python using the PyTorch deep learning framework with ROCm compatibility. This choice yields several advantages: 
\begin{itemize}
    \item \textbf{Portability} — the same code can execute on CPUs, AMD GPUs, and NVIDIA GPUs without modification, enabling broader adoption across heterogeneous computing environments;  
    \item \textbf{Ecosystem integration} — by leveraging PyTorch’s mature tensor operations, automatic differentiation, and memory management, we reduce development complexity while benefiting from ongoing optimizations in the machine learning community;  
    \item \textbf{Accessibility} — PyTorch’s widespread use in AI and ML research lowers the barrier for practitioners and researchers in these fields to adapt our solver for their applications.
\end{itemize}

By combining PDHG’s algorithmic suitability for parallel architectures with PyTorch’s hardware abstraction capabilities, our implementation provides an open-source, high-performance LP solver optimized for AMD GPUs but readily deployable across multiple platforms. This positions our work as a CUDA-free alternative for large-scale optimization, offering a viable path forward for organizations seeking to diversify their GPU hardware ecosystems.

\section{Report overview}
The remainder of this report is organized as follows.  

Chapter~\ref{chap:math_background} presents the mathematical foundations relevant to our work. Beginning with a concise introduction to linear programming, we describe the PDHG algorithm in the context of linear programming. The chapter includes a discussion of the algorithmic enhancements incorporated into our solver, which are designed to improve convergence speed and robustness in large-scale settings.

Chapter~\ref{chap:comp_background} details the experimental setup and computational environment used in this work. We describe the hardware platforms on which our benchmarks were conducted and illustrate the computational architectures relevant to our implementation. The chapter also introduces the benchmark datasets, comprising standard benchmark sets and a real-world problem instance from Gridmatic, an energy trading company, and outlines our approach to ensuring code availability and reproducibility.

Chapter~\ref{chap:results} details our experimental evaluation. We report performance benchmarks comparing our PyTorch-based PDHG solver on AMD GPUs against CPU-based baselines across standard LP test sets such as Netlib. We also assess robustness relative to other solvers, explore scaling behavior on larger problem instances, and discuss the trade-offs between problem size and GPU overhead.  

Chapter~\ref{chap:conclusion} concludes the report and outlines directions for future work. Potential avenues include mixed-precision computation, further parallel algorithmic refinements, integration with ROCm-specific kernels, comparisons with NVIDIA’s cuOpt framework, and the incorporation of insights from recent research in large-scale optimization.  

Finally, the appendix provides a detailed derivation of the PDHG formulation for LP starting from the general primal–dual framework.

\ifthenelse{\boolean{@twoside}}{\myclearpage}{}
\chapter{Mathematical Formulation}\label{chap:math_background}

The Primal-Dual Hybrid Gradient (PDHG) method is an algorithm for solving general convex optimization problems~\cite{chambolle2011first}. The core of the method involves reformulating the optimization problem into an equivalent saddle-point problem~\cite{boyd}. The algorithm then performs alternating gradient-based steps on the primal and dual variables, an iterative process designed to converge to a saddle point of the associated Lagrangian function. 

In this chapter, we begin by presenting the specific mathematical formulation of the LP problem, followed by a derivation of the PDHG algorithm tailored to this context. For a more rigorous treatment, the detailed transformation from the general PDHG algorithm to this specialized version, along with a formal proof of its convergence, is provided in Appendix~\ref{appendixC}. Subsequently, we review key adaptations from the convex optimization literature that enhance the practical performance and robustness of PDHG as an LP solver. Finally, we introduce a novel heuristic we have developed, termed the `fishnet' method, which employs an adaptive strategy to refine the solution search space.


\section{Notation}
Let \(\mathbb{R}^+\) denote the set of nonnegative real numbers and \(\mathbb{R}^-\) denote the set of nonpositive real numbers. Let \(\| \cdot \|_p\) denote the \(\ell_p\) norm for a vector, and \(\| \cdot \|_2\) denote the spectral norm of a matrix. For a vector \(v \in \mathbb{R}^n\), we use \(v^+,v^-\) for their positive and negative parts, i.e., \(v^+_i = \max\{0,v_i\}, v^-_i = \min\{0,v_i\}, \  v = v^++v^-\). The symbol \(v_{1:m}\) denotes the first \(m\) components of vector \(v\). The symbol \(K_{i,\cdot}\) and \(K_{\cdot,j} \) corresponding to the \(i^\text{th}\) column and $j^\text{th}$ row of the matrix \(K\), respectively. The symbol \(\oplus\) denotes the Cartesian product, e.g., \(\mathbb{R}\oplus\mathbb{R} = \mathbb{R}^2\).




\section{Linear Programming}\label{sec:lp}
Any \textbf{Linear Programming} (LP) problem can be written in the form:
\[
\begin{aligned}
    \min_{x \in \mathbb{R}^n} \quad & c^\top x \\
    \text{subject to:} \quad & Gx \ge h \\
    & Ax = b \\
    & l \le x \le u
\end{aligned}
\]
where 
\[
\begin{aligned}
&x \in \mathbb{R}^n, \quad 
G \in \mathbb{R}^{m_1 \times n} \ \text{(inequality constraints matrix)}, \quad
A \in \mathbb{R}^{m_2 \times n} \ \text{(equality constraints matrix)}, \\
&c \in \mathbb{R}^n \ \text{(cost vector)}, \quad
h \in \mathbb{R}^{m_1}, \quad b \in \mathbb{R}^{m_2}, \\
&l \in \bigoplus_{k=1}^{n} \left\{ \mathbb{R} \cup \{-\infty\} \right\}, \quad
u \in \bigoplus_{k=1}^{n} \left\{ \mathbb{R} \cup \{\infty\} \right\} \text{, (which specify the lower and upper bounds of $x$).}\
\end{aligned}
\]
We allow $l_i = -\infty$ and $u_j = \infty$ to indicate an unbounded variable. This formulation is referred to as the \textbf{primal problem}.

The corresponding \textbf{dual problem}, derived via Lagrangian duality (see  Appendix~\ref{App:DualDerivation} for details), is:
\[
\begin{aligned}
    \max_{y \in \mathbb{R}^{m_1+m_2}, \ \lambda \in \mathbb{R}^n} \quad & q^\top y + l^\top \lambda^+ + u^\top \lambda^- \\
    \text{subject to:} \quad & \lambda = c - K^\top y, \\
    & y_{1:m_1} \ge 0, \\
    & \lambda \in \Lambda,
\end{aligned}
\]
where $ K=\begin{pmatrix}
    G\\A
\end{pmatrix},  \Lambda = \bigoplus_{i=1}^{n} \Lambda_i$, with
\[
\Lambda_i =
\begin{cases}
\{0\}, & \text{if } l_i = -\infty \ \text{and} \ u_i = \infty, \\[0.3em]
\mathbb{R}^-, & \text{if } l_i = -\infty \ \text{and} \ u_i \in \mathbb{R}, \\[0.3em]
\mathbb{R}^+, & \text{if } l_i \in \mathbb{R} \ \text{and} \ u_i = \infty, \\[0.3em]
\mathbb{R}, & \text{otherwise}.
\end{cases}
\]

\section{Primal-Dual Hybrid Gradient Algorithm}\label{sec:initial_solver}


The general PDHG algorithm uses gradient information to solve general primal-dual convex optimization problems and when applied to the general LP problem outlined in the section above takes the form of Algorithm \ref{alg:PDHG_alg}; where $\tau,\sigma >0$ satisfy $\tau\sigma \|K\|_2^2 <1$, $\theta \in [0,1]$, $X:=\{x \in \RR^n \; :l \leq x \leq u\}$, and $Y:=\{y\in \RR^{m_1+m_2} \; : y_{1:m_1} \geq 0\}$.
We provide a proof that this algorithm will converge to the optimal primal-dual solution in Appendix~\ref{appendixD}.

\begin{algorithm}
\begin{algorithmic}[1]
\caption{PDHG Algorithm for Linear Programming \cite{applegate2}}
\label{alg:PDHG_alg}
    \Require $(x^0,y^0) \in \mathbb{R}^n \times \mathbb{R}^{m_1+m_2}, \bar{x}^0 \leftarrow x^0, t \leftarrow 0$
    \Repeat 
    \State $x^{t+1} \leftarrow \proj_{X}(x^n-\tau(c+ K^\top y^n))$
    \State $\bar{x}^{n+1} \leftarrow (x^{n+1} + \theta(x^{n+1}- x^n))$
    \State $y^{t+1} \leftarrow \proj_{Y}(y^n+\sigma(q +  K \bar{x}^{n+1})$
    \State $t \leftarrow t + 1$
    \Until the termination criteria hold
    
\end{algorithmic}
\end{algorithm}

\subsection{Re-parameterizing the Step Sizes}
We have thus far introduced the PDHG algorithm and proven its convergence. However, there are still a number of details specific to the algorithm needed to analyze its behavior.

The primal and dual step sizes, $\tau$ and $\sigma$, are re-parameterized to have better control over their sizes \cite{applegate2}. We let $\eta,\omega >0$ and set \[ \tau=\eta/\omega \quad \text{ and } \quad \sigma=\eta\omega.\]
We call $\eta$ the \textit{step size} since it controls the magnitude of the dual and primal sizes, and we call $\omega$ the \textit{primal weight} since it controls the ratio of the primal to the dual step size. From this, it is clear to see that the convergence requirement $\tau\sigma\|K\|_2^2 <1$ is met if and only if $\eta < 1/\|K\|_2$, but $\omega$ is free to be any positive number.

When setting the step size, it is useful to be able to evaluate $\|K\|_2$. An exact computation of the spectral norm of a matrix requires singular value decomposition. However, for larger matrices, this becomes far too computationally expensive, so we employ the power iteration algorithm \cite{golub2013matrix} to estimate the spectral norm. This technique is outlined in algorithm \ref{alg:power_iter} and has the benefit of only using matrix-vector multiplications and norm calculations, which are greatly accelerated with GPUs.
\begin{algorithm}[H]
\begin{algorithmic}[1]
\caption{SpectralNormEstimate($K, N_{iter}$)\cite{golub2013matrix}}
\label{alg:power_iter}
    \State $b \leftarrow [1 \; \cdots \; 1]^\top$
    \For {$i=1,2,\ldots, N_{iter}$}
    \State $b \leftarrow K^\top(Kb)$
    \State $b \leftarrow b/\|b\|_2$
    \EndFor
    \State \Return $\|Kb\|_2$
    
\end{algorithmic}
\end{algorithm}
The primal weight $\omega$ is also used to define the following norm of both the primal and dual variables,
\begin{equation}\|(x,y)\|_\omega :=\sqrt{\omega\|x\|_2^2 + \frac{\|y\|_2^2}{\omega}}.\end{equation}
This equation, called the `omega norm', will become useful in section \ref{math_step_size}.
\subsection{Termination Criteria}\label{sec:term}
To determine whether to terminate our algorithm, having found a solution, we employ termination criteria that takes the `relative KKT error', first derivative checks applicable to many types of mathematical optimization. Setting $\epsilon$ to be our tolerance ($10^{-4}$ throughout this report), our algorithm terminates when the following conditions are met:
\begin{equation}\label{term1}
|q^{\top}y + l^{\top}\lambda^{+}+u^{\top}\lambda^{-}-c^{\top}x| \leq \epsilon (1 + |q^{\top}y + l^{\top}\lambda^{+}+u^{\top}\lambda^{-}| + |c^\top x|)
\end{equation}
\begin{equation}\label{term2}\left\|\begin{pmatrix}
    Ax-b \\ [h-Gx]^+
\end{pmatrix}\right\|_2 \leq \epsilon (1 + ||q||_2)\end{equation}
\begin{equation}\label{term3}||c-K^\top y -\lambda||_2 \leq \epsilon (1+||c||_2)\end{equation}
where $\lambda$ is calculated $\lambda=\proj_\Lambda(c-K^\top y)$ in accordance with the statement of the dual LP in section \ref{sec:lp}. The top condition represents the duality gap between the primal and dual objective values, and the second and third represent primal and dual feasibility conditions. The left-hand sides of these inequalities are 0 if and only if $(x,y)$ is an optimal primal-dual solution to the LP.

For instances where preconditioning is used (see section \ref{precond-section}), termination criteria are checked for the \textit{original} tensors, not the preconditioned tensors. Because these matrix-vector multiplications and norms are not cheap, we do not evaluate termination criteria every iteration, but after a set number of iterations, and after every restart (section \ref{sec:restart}).
\section{Previous PDHG Adaptations}\label{sec:pdlp}
The above analysis defines a general problem of convex optimization, and defines linear programming as a specific example of this class of problems. Then, the PDHG algorithm is introduced, and convergence is proved. However, the PDHG algorithm is a \textit{general solver} for convex-optimization saddle point problems. To develop a robust solver specifically tailored for \textit{linear programs}, Applegate et al. introduce a series of enhancements \cite{applegate2} that render PDHG more effective for LP problems.

These enhancements include: step size modification, adaptive restarting, primal weight updates, presolving, and diagonal preconditioning. As of the writing of this report, we have successfully implemented diagonal preconditioning, primal weight updating, adaptive step size, and adaptive restarting. We explain each in detail in the following subsections.

\subsection{Diagonal Preconditioning and Presolve}\label{precond-section}
Diagonal Preconditioning and Presolve are both steps done before running the algorithm. We employ diagonal preconditioning to improve the numerical stability and convergence rate of the algorithm. This technique rescales the problem by transforming the constraint matrix $K$ into a `well-balanced’ matrix $\tilde{K}$. The transformation is defined by two positive diagonal matrices, a row-scaling matrix $D_r$ and a column-scaling matrix $D_c$, such that:

\[
\tilde{K} = D_r^{-1} K D_c^{-1}
\]

This rescaling of the constraint matrix induces a corresponding change of variables for the primal solution vector $x$, the objective vector $c$, and the right-hand side vector $q$:
\begin{align*}
\tilde{x} &= D_cx \\
\tilde{c} &= D_c^{-1} c \\
\tilde{q} &= D_r^{-1} q
\end{align*}
We are using the Ruiz scaling algorithm \cite{ruiz2001scaling}, detailed in algorithm \ref{alg:ruiz_scaling}.


Preconditioning our matrix reduces our condition number of the matrix, ensuring our PDHG has better-behaved convergence behavior. Geometrically, this ensures the algorithm's iterates stay within a smaller neighborhood, as the singular values vary less dramatically.

Presolve, on the other hand, is a standard technique in LP that involves removing bounds where the upper and lower bounds are the same, removing empty rows or columns, detecting inconsistent bounds, utilizing equality constraints to remove redundant variables, and a number of other techniques to make solving LP more efficient. After solving a presolved LP, the minimizer solution must be `post-solved' to turn that into a minimizer solution to the original LP. Writing our own presolver was beyond the scope of this project, so we use the state-of-the-art open source application PaPILO \cite{papilo}.
\begin{algorithm}[htpb]
\caption{Iterative Ruiz Scaling Algorithm \cite{ruiz2001scaling}}
\label{alg:ruiz_scaling}
\begin{algorithmic}[1]
    \Require Constraint matrix $K \in \mathbb{R}^{m \times n}$, max number of iterations $N_{\text{iter}}$.
    \Ensure The final scaled matrix $\tilde{K}$.
    
    \State Let $\tilde{K} = K$.
    \For{$k = 1, \dots, N_{\text{iter}}$}
        \State Define diagonal scaling matrices $D_r$ and $D_c$:
        \begin{align*}
            (D_r)_{jj} &\leftarrow 1 / \sqrt{\|\tilde{K}_{j,:} \|_{\infty}} \quad \text{for } j=1, \dots, m \\
            (D_c)_{ii} &\leftarrow 1 / \sqrt{\|\tilde{K}_{:,i} \|_{\infty}} \quad \text{for } i=1, \dots, n
        \end{align*}
        
        \State Update the constraint matrix: $\tilde{K} \leftarrow D_r \tilde{K} D_c$.
        
        \State \If{\text{all row and column infinity norms have converged to 1 within tolerance}}
            \State \textbf{break} 
        \EndIf
    \EndFor
    \State \Return $\tilde{K}$
\end{algorithmic}
\end{algorithm}

\subsection{Adaptive Restarts}\label{sec:restart}
A key adaptation of torchPDLP (see section \ref{sec: torchPDLP}) is periodically restarting PDHG. This averages points along the convergence spiral and restarts the algorithm at the midpoint of these chosen points. To simplify notation in this section, we combine the primal solution x and the dual solution y into a single variable, denoted as $z=(x,y)$. The vector $z_n^t$ denotes the $t^{\text{th}}$ iterate of the $n^\text{th}$ restart.

The metric we use to choose when to restart and with which candidate is the KKT error, defined by
\begin{equation}\KKT (z):=\sqrt{\omega^2 \left\|\begin{pmatrix} Ax-b \\ [h-Gx]^+ \end{pmatrix}\right\|_2^2 +\frac{1}{\omega^2}\|c-K^\top y -\lambda\|_2^2 + (q^\top y +l^\top \lambda^+ +u^\top \lambda^- -c^\top x)^2}\end{equation}
which distills all three termination criteria discussed in section \ref{sec:term} into one value.

First, we get a candidate for restarting as follows: \[
\hat{z}_n^{t+1} = \text{GetRestartCandidate}(z_n^{t+1}, \bar{z}_n^{t+1}) :=
    \begin{cases}
        z_n^{t+1}, & \text{if } \KKT(z_n^{t+1}) < \KKT(\bar{z}_n^{t+1}) \\
        \bar{z}_n^{t+1}, & \text{otherwise}
    \end{cases}
\]
where \[\bar{z}_n^{t+1} = \frac{1}{\sum_{i=1}^{t+1} \eta_n^i} \sum_{i=1}^{t+1} \eta_n^i z_n^i\] is a weighted average of iterates with their step sizes (if employing adaptive step size, see section \ref{math_step_size}).

We restart our algorithm at this candidate whenever any of the following criteria are satisfied:
\begin{enumerate}
    \item \textbf{Sufficient Decay}: \(\KKT(\hat{z}_n^{t+1}) \leq \beta_{\text{sufficient}} \KKT(z_n^0)\),
    \item \textbf{Necessary Decay + No Progress}: \(\KKT(\hat{z}_n^{t+1}) \leq \beta_{\text{necessary}} \KKT(z_n^0)\) 
          and \(\KKT(\hat{z}_n^{t+1}) > \KKT(\hat{z}_n^t)\),
    \item \textbf{Long Inner Loop}: the number of iterations this restart is more than \(\beta_{\text{artificial}}\) of the total iterations,
\end{enumerate}
where \(0 < \beta_{\text{sufficient}} < \beta_{\text{necessary}} < 1\) and \(0 < \beta_{\text{artificial}} < 1\) are user-defined constants. The justification for these details regarding restarting is found in \cite{Applegate_2022} and in torchPDLP we use $\beta_{\text{sufficient}}=0.2$, $\beta_{\text{necessary}}=0.8$, and $\beta_{\text{artificial}}=0.36$. 

\subsection{Primal Weight Updates}\label{sec:primal_weight}
Primal weight updating is a method of dynamically adjusting our primal weight $\omega$ over time, to `balance the distances between optimality' in the primal and dual. For example, if the primal guess $x^t$ at time $t$ given by the algorithm is `close' (by Euclidean-distance) to the primal solution $x^*$, but the dual guess $y^t$ at time $t$ is much further, this can cause problems: intuitively, we want to `nudge' the primal to its solution and return the solution quickly. Instead, however, the dual solution at $y^{t}$ is far away, meaning that when we update $x^t \mapsto x^{t+1}$, our step is larger than it needs to be because its gradient step is a function of $y^t$.

First, then, we initialize our primal weight as follows:
\[
\text{InitializePrimalWeight}(c, q) := \begin{cases}
{\frac{\|c\|_2}{\|q\|_2}} & \text{if } \|c\|_2, \|q\|_2 > \epsilon_{\text{zero}} \\
1 & \text{otherwise}
\end{cases}
\] where $c,q$ are the primal and dual objective vectors respectively and $\epsilon_{\text{zero}}$ is some arbitrarily small constant. 

This initialization guesses that the primal and dual step sizes scale with the norm of their objective vectors. Then, we periodically update the primal weight at step $n$ to scale with the primal and dual iterates' respective distance to the optimum. Of course, we do not know this distance exactly because we do not know the LP's optimal solution - instead, we estimate it with $\Delta_n^y/\Delta_n^x$ where $\Delta_n^y := ||y^0_n-y^0_{n-1}||_2$ and $\Delta_n^x := ||x^0_n-x^0_{n-1}||_2$. Thus, our primal weight update algorithm is as follows

\begin{algorithm}[H]
\begin{algorithmic}[1]
\caption{Primal Weight Update($x_n^0, x_{n-1}^0, y_n^0, y_{n-1}^0\omega_{n-1}$)\cite{Applegate_2022}:}
\label{alg:primal_weight}
        \State $\Delta_n^x \leftarrow \|x_n^0 - x_{n-1}^0\|_2$
        \State $\Delta_n^y \leftarrow \|y_n^0 - y_{n-1}^0\|_2$
        \If{$\Delta_n^x > \epsilon_{zero} \;\text{ and } \; \Delta_n^y > \epsilon_{zero}$}
        \State \Return $\exp\Big(0.5 \log \left (\Delta_n^y/\Delta_n^y \right) + 0.5 \log (\omega_{n-1}) \Big)$
        \Else
        \State \Return $\omega_{n-1}$
        \EndIf
        
\end{algorithmic}
\end{algorithm}


\subsection{Adaptive Step Sizes}\label{math_step_size}
Adaptive step size is a technique that involves periodically adjusting the step size to ensure that, on the one hand, the step size ensures convergence, yet is large enough to speed up our convergence. Recall that convergence is guaranteed when $\eta \leq \frac{1}{||K||_2}$. This, it turns out, is an `overly pessimistic estimation' of $\eta$ \cite{applegate2}. Thus, the following algorithm ensures convergence guarantees while dynamically adjusting the step size such that :
\begin{equation}
\eta \leq \frac{\|z^{k+1}-z^{k}\|^2_{\omega}}{2|(y^{k+1}-y^k)^\top K(x^{k+1}-x^k)|}
\end{equation}
is ensured. We suspect that there are less computationally-intensive ways to not be over-pessimistic with step size adjustment while ensuring convergence that have yet to be found.
\begin{algorithm}[H]
\begin{algorithmic}[1]
\caption{AdaptiveStepOfPDHG ($z_{n,t}, \omega_n, \hat{\eta}_{n,t}, k$)\cite{applegate2}:}
\label{alg:adaptive_step}
        \State $(x, y) \leftarrow z_{n,t}$
        \State $\eta \leftarrow \hat{\eta}_{n,t}$
        \While{true}
            \State $x' \leftarrow \proj_{X}\left( x - \frac{\eta}{\omega_n}(c - K^\top y) \right)$
            \State $y' \leftarrow \proj_{Y}\left( y + \eta \omega_n (q - K (2x' - x)) \right)$
            \State $\bar{\eta} \leftarrow \frac{\|(x' - x, y' - y)\|_{\omega_n}^2}{\lvert2(y' - y)^\top K (x' - x)\rvert}$
            \State $\eta' \leftarrow \min\left\{ (1 - (k + 1)^{-0.3})\bar{\eta},\; (1 + (k + 1)^{-0.6})\eta \right\}$
            \If{$\eta \leq \bar{\eta}$}
                \State \Return $(x', y'), \eta, \eta'$
            \EndIf
            \State $\eta \leftarrow \eta'$
        \EndWhile
\end{algorithmic}
\end{algorithm}

\subsection{Infeasibility Detection}\label{sec:infeas_detect}
Infeasibility detection is a method for determining, during the execution of PDHG, whether the primal or dual problem has no feasible solution. Rather than relying on solving auxiliary problems, we exploit the asymptotic behavior of the PDHG iterates themselves. If a problem is infeasible, the iterates diverge along a well-defined ray whose direction is called the \emph{infimal displacement vector} \(v=(v_x,v_{y_{ineq}},v_{y_{eq}},v_{\lambda})\). The vector encodes \emph{certificates of infeasibility} for the primal and/ or dual problem.

Concretely:
\begin{itemize}
    \item if \(v_{y_{ineq}} \ge0,\) and satisfies
    \[
    G^\top v_{y_{ineq}} + A^\top v_{y_{eq}} - v_{\lambda} = 0
    \]
    and
    \[
    h^\top v_{y_{ineq}} + b^\top v_{y_{eq}} - l^\top v_{\lambda}^- - u^\top v_{\lambda}^+ > 0
    \]
    this certifies \textbf{primal infeasibility}.
    \item for \[V_i=
\begin{cases}
\mathbb{R}, & \text{if } l_i = -\infty \ \text{and} \ u_i = \infty, \\[0.3em]
\mathbb{R}^-, & \text{if } l_i = -\infty \ \text{and} \ u_i \in \mathbb{R}, \\[0.3em]
\mathbb{R}^+, & \text{if } l_i \in \mathbb{R} \ \text{and} \ u_i = \infty, \\[0.3em]
\{0\}, & \text{otherwise},
\end{cases}
\]
if $v_x$ satisfies
\[
Gv_x \ge 0,\quad Av_x=0,\quad c^\top v_x <0,\quad v_{x_i}\in V_i
\]
this certifies \textbf{dual infeasibility}.
\end{itemize}

Our formulation adopts certificates analogous to the standard-form certificates \cite{applegate2021infeasibility}, and in this setting, the components \(v_x\) and \((v_{y_{\mathrm{ineq}}}, v_{y_{\mathrm{eq}}}, v_{\lambda})\) of the infimal displacement vector serve as valid infeasibility certificates.

\paragraph{Detection Sequences.}
As proposed in~\cite{applegate2021infeasibility}, we keep track of the `difference of iterates' which converge to $v$ for any infeasible problem.

\begin{algorithm}[H]
\caption{InfeasibilityDetection($x_k,y_k,\lambda_k,x_{k-1},y_{k-1},\lambda_{k-1},tol$)}
\begin{algorithmic}[1]
\State $\Delta x \gets x_k-x_{k-1},\quad \Delta y \gets y_k-y_{k-1},\quad \Delta\lambda \gets \lambda_k-\lambda_{k-1}$
\State Split $\Delta y=(\Delta y_{\mathrm{ineq}},\,\Delta y_{\mathrm{eq}})$

\State \textbf{Dual-infeasibility test (certificate via $\Delta x$)}
\If{$\|A\Delta x\|_2\le tol$ \textbf{and} $\min(G\Delta x)\ge -tol$ \textbf{and} $c^\top\Delta x\le -tol$ \textbf{and} $\Delta x_i\in V_i\ \forall i$}
    \State \Return \textsc{DUAL\_INFEAS}
\EndIf

\State \textbf{Primal-infeasibility test (certificate via $\Delta y,\Delta\lambda$)}
\If{$\|G^\top\Delta y_{\mathrm{ineq}}+A^\top\Delta y_{\mathrm{eq}}-\Delta\lambda\|_2\le tol$ \textbf{and} $\min(\Delta y_{\mathrm{ineq}})\ge -tol$}
    \State $\psi \gets h^\top\Delta y_{\mathrm{ineq}} + b^\top\Delta y_{\mathrm{eq}} - l^\top(\Delta\lambda)^- - u^\top(\Delta\lambda)^+$
    \If{$\psi \ge -tol$}
        \State \Return \textsc{PRIMAL\_INFEAS}
    \EndIf
\EndIf
\end{algorithmic}
\end{algorithm}

\section{Our Enhancement - Fishnet Casting} \label{sec:Fishnet}
While algorithms like the PDHG can leverage GPU acceleration for core operations like matrix-vector multiplication, their fundamental structure remains sequential. Each iteration, $(x^{k+1}, y^{k+1})$, is strictly dependent on the result of the previous iteration, $(x^k, y^k)$. This inherent iterative dependency limits the parallelization potential within a single solution trajectory. 

To overcome this limitation and better exploit massively parallel architecture, we introduce a novel multi-starting heuristic we call \textbf{Fishnet Casting}. The strategy abandons a single-trajectory approach in favor of computing multiple solution paths simultaneously. Conceptually, the process mirrors casting a wide net and progressively tightening it to isolate a target. It begins by generating a diverse set of initial candidate points and then iteratively refines this set using a process of evaluation, culling, and repopulation loosely inspired by multi-start optimization \cite{MARTI20131} and genetic algorithms \cite{Manning-genetic}. The final surviving candidate is then used as an initial point in the torchPDHG solver.

The framework consists of two main phases: an initial spectral casting phase and an iterative fishnet loop, detailed below and in Algorithm~\ref{alg:fishnet}. In spectral casting, we create a ball $\mathcal{B}_r$ in primal space of radius $r$, sufficiently large that we hope it contains the primal LP's feasible region. The radius $r$ is set using a computationally inexpensive estimate of the spectral norm of the constraint matrix, $r := ||K||_2$. While this choice is a heuristic, it allows the search radius to scale with the problem's characteristics without incurring the high computational cost of more theoretically grounded methods, such as computing Löwner-John ellipsoids~\cite{henk2012lowner}. We then sample $2^p$ points from a normal distribution around $\mathcal{B}_r$, for $p$ some integer parameter. Let these points be $X^0 := \{x_0^0,x_1^0,\cdots,x_{2^p}^0\}$.

The second part of the technique consists of a main loop. Within this loop, we run $k$ iterations of PDHG on a set of points $X^i$. Then, we measure the duality gap for all of these points and `delete' the worst-performing half of these points with respect to the duality gap. Then, if the main loop is an even iteration, we `breed' new points as random convex combinations of the best-performing half. This ensures that we have some measure of randomness, yet our new points are likely not far worse than the remaining half. Even if they are, after the next $k$ PDHG iterations, the worst-performing points will be culled regardless.

\begin{algorithm}[H]
\begin{algorithmic}[1]
\caption{fishnet($\hat{X}^0,K,\eta,\omega,c,q,l,u,k$):} \label{alg:fishnet}
        \State $\hat{Y}^0 \leftarrow K\hat{X}^0$
        \State j $\leftarrow$ columns(K)
        \State $i \leftarrow 0$
        \While{$j > 1$}
            \State $t \leftarrow 0$
            \For{$t < k$}
                \State $\hat{X}^i,\hat{Y}^i \leftarrow \text{PDHG}(\hat{X}^i,\hat{Y}^i,K,\eta,\omega,c,q)$
                \State $t \leftarrow t+1$
                \EndFor
            \State $\hat{X}^{i+1},\hat{Y}^{i+1} \leftarrow \text{cull\_points}(K,\hat{X}^i,\hat{Y}^i,c,q,l,u)$
            \If{$i = 1 \pmod{2} \ \& \ j > 2$}
                \State $\hat{X}^{i+1},\hat{Y}^{i+1} = \text{repopulate}(\hat{X}^{i+1},\hat{Y}^{i+1})$
            \EndIf
            \State $i \leftarrow i+1$, j $\leftarrow$ columns(K)
        \EndWhile
        \State $x^0,y^0 \leftarrow \hat{X}^i,\hat{Y}^i$
        \Return $x^0,y^0$
\end{algorithmic}
\end{algorithm}
The function `columns()' takes as input a matrix and returns the number of columns of the matrix. PDHG(X,Y) runs PDHG with fixed step-size on matrices X,Y, whose columns correspond to primal and dual iterates, respectively. This transforms PDHG from a matrix-vector-dominated algorithm to a matrix-matrix-multiplication-based algorithm. 

`cull\_points()' is a function that takes as input two matrices X and Y, and halves their number of columns, removing columns whose duality gaps are larger and thus worse-performing. Thus, the fraction culled is 1/2, ensuring we reduce column count in an efficient manner. Finally, repopulate(X,Y) takes columns of X and Y, and randomly creates convex combinations of their columns, doubling the column count in the returned matrices. Culling points and repopulating them in this manner ensures we balance reducing the number of points with enough `genetic diversity' in between points, ensuring we get a good starting vector.

\ifthenelse{\boolean{@twoside}}{\myclearpage}{}
\chapter{Experimental Setup and Computational Environment}\label{chap:comp_background}  

In this chapter, we describe the hardware platforms, solver architecture and datasets used in our experiments. We begin with a comparison of CPU and GPU architectures, focusing on the specific hardware employed in our study: an AMD EPYC 7V13 64-core CPU, AMD Instinct MI210 and MI325X GPUs, and an NVIDIA A100 GPU. An illustration of the computational architecture is provided to highlight the parallelism opportunities relevant to our implementation.  We then detail the benchmark datasets used for performance evaluation. These datasets enable us to evaluate both general-purpose solver performance and application-specific optimization in the energy domain. Finally, our solver is implemented in Python using PyTorch with ROCm support, and all source code, scripts, and configuration files are made available in a 
public repository \cite{github} to facilitate replication of results and further research.

\section{Computational Devices and Accelerators}
A core component of our study involves understanding how the architectural differences between CPUs and GPUs influence the performance of linear programming solvers. While both are capable of executing general-purpose computations, CPUs are engineered for low-latency, sequential workloads, excelling at complex control flow and diverse instruction sets. GPUs, by contrast, are optimized for high-throughput, massively parallel workloads, making them well-suited for operations that can be decomposed into many independent, homogeneous tasks.
\begin{table}[h!]
\centering
\begin{tabular}{lll}
\hline
\textbf{Architectural Feature} & \textbf{CPU} & \textbf{GPU} \\
\hline
Core Architecture    & Few, complex cores                & Many, simple cores \\
Execution Model      & Complex, branching tasks          & Parallel data batches \\
Memory Subsystem     & Low-latency memory                & High-bandwidth memory \\
Primary Workload     & Serial        & Massive data parallelism \\
\hline
\end{tabular}
\caption{Key architectural distinctions between CPU and GPU paradigms.}
\label{tab:gpu_cpu_summary}
\end{table}
\subsection{The Central Processing Unit (CPU)}
The CPU is the most common general-purpose processor, often referred to as the `brain' of a computer, responsible for executing the commands and processes essential for the operating system and all user-facing applications. Its architectural design has been honed over decades with a primary of minimizing latency, which is the time required to complete a single, discrete task. This focus on speed for individual operations makes the CPU proficient at handling complex, branching, and unpredictable instruction streams that characterize general-purpose computing.

The philosophy of low-latency optimization is embedded in the CPU's core design. A modern high-performance CPU is constructed with a relatively small number of powerful cores. The AMD EPYC 7V13 CPU employed in our study, for instance, features 64 cores based on the Zen 3 architecture~\cite{amd:epyc7003}. While this number is substantial, it is orders of magnitude smaller than the core counts found in its GPU counterparts. A significant portion of the silicon die area within each CPU core is dedicated not to raw arithmetic computation, but to sophisticated control logic, deep instruction pipelines, advanced branch prediction units, and speculative execution engines. These components work in concert to accelerate a single thread of execution, ensuring that complex, conditional logic is navigated with maximum speed.

These strengths however represent a fundamental design trade-off. The complexity and power consumption of each core make it physically and economically impractical to place thousands of them on a single chip. Complementing this complex core design is an intricate cache hierarchy, such as the 256 MB of L3 cache in the EPYC 7V13, designed for rapid data access by its cores. Consequently, the CPU's architecture, while well suited for its intended purpose of managing complex and varied \textit{sequential} tasks, is inherently limited in its capacity for computational parallelism. 

\subsection{The Graphics Processing Unit (GPU)}
In contrast to the latency-optimized design of the CPU, the GPU represents an architectural paradigm engineered for maximum total volume of computational work completed per unit of time. The GPU's design philosophy prioritizes this aggregate rate of computation over the speed of any single, individual operation. GPUs have evolved from their original purpose (graphics processing) into powerful, general-purpose parallel processors: their architecture is now the cornerstone of modern high performance computing, artificial intelligence, and large-scale scientific simulation, where problems can be decomposed into a vast number of simple, repetitive and \textit{parallelizable} calculations.

Our computational framework utilizes several units, including the NVIDIA A100, the AMD Instinct MI210, and, most centrally, the AMD Instinct MI325X. The NVIDIA A100 (80 GB model), built on the Ampere architecture, features 6,912 CUDA cores~\cite{nvidia:a100}. The AMD Instinct MI210, based on the CDNA2 architecture, provides a comparable array of 104 Compute Units (CUs), which contain 6,656 Stream Processors~\cite{amd:mi210}. The primary processing unit for this project, the AMD Instinct MI325X, represents a significant advancement. Built on the newer CDNA3 architecture, it integrates 304 CUs, housing a total of 19,456 Stream Processors~\cite{amd:mi325x}. This massive number of simpler, specialized cores is the key to the GPU's parallel processing power.


To prevent this multitude of cores from becoming data-starved, GPUs are equipped with a memory subsystem optimized for extremely high bandwidth rather than low latency. The GPUs in our study showcase this design principle. The MI210 is equipped with 64 GB of HBM2e memory delivering up to 1.6 TB/s of bandwidth. The NVIDIA A100 offers a similar capability with 80 GB of HBM2e memory and a peak bandwidth of approximately 2.0 TB/s. The MI325X, however, substantially increases this capability, featuring 256 GB of state-of-the-art HBM3e memory. This is connected via a wide memory bus, yielding a peak theoretical memory bandwidth of up to 6TB/s. This high bandwidth is indispensable for continuously supplying the stream processors with operands and writing back their results, ensuring that the computational resources are fully utilized.
\subsection{Compute Access}
The computational experiments for this project were conducted on high-performance computing resources from several platforms. Our primary benchmarking were performed on the AMD AI \& HPC Research Cluster. Supplementary experimental results were obtained using the Heterogeneous Accelerated Compute Clusters (HACC) at UCLA, specifically leveraging the AMD Instinct MI210 accelerators. Furthermore, we were kindly introduced to the AMD Developer Cloud to ensure research continuity during a period of scheduled infrastructure migration of our primary compute cluster. It offers an ideal platform for developers and open-source contributors building and optimizing AI, machine learning, and HPC workloads on AMD hardware, with access to high-performance AMD Instinct MI300X GPUs.

\subsection{Multithreading}\label{multithread}
The most computationally intensive step in the PDHG algorithm involves the matrix-vector multiplications with the constraint matrix $K$ and its transpose, $K^\top$. To confirm the performance benefits of parallel execution, we conducted a preliminary experiment on a CPU architecture by varying the number of computational threads allocated to our program in PyTorch. 
\begin{figure}[htbp]
    \centering
    \includegraphics[width=0.8\linewidth]{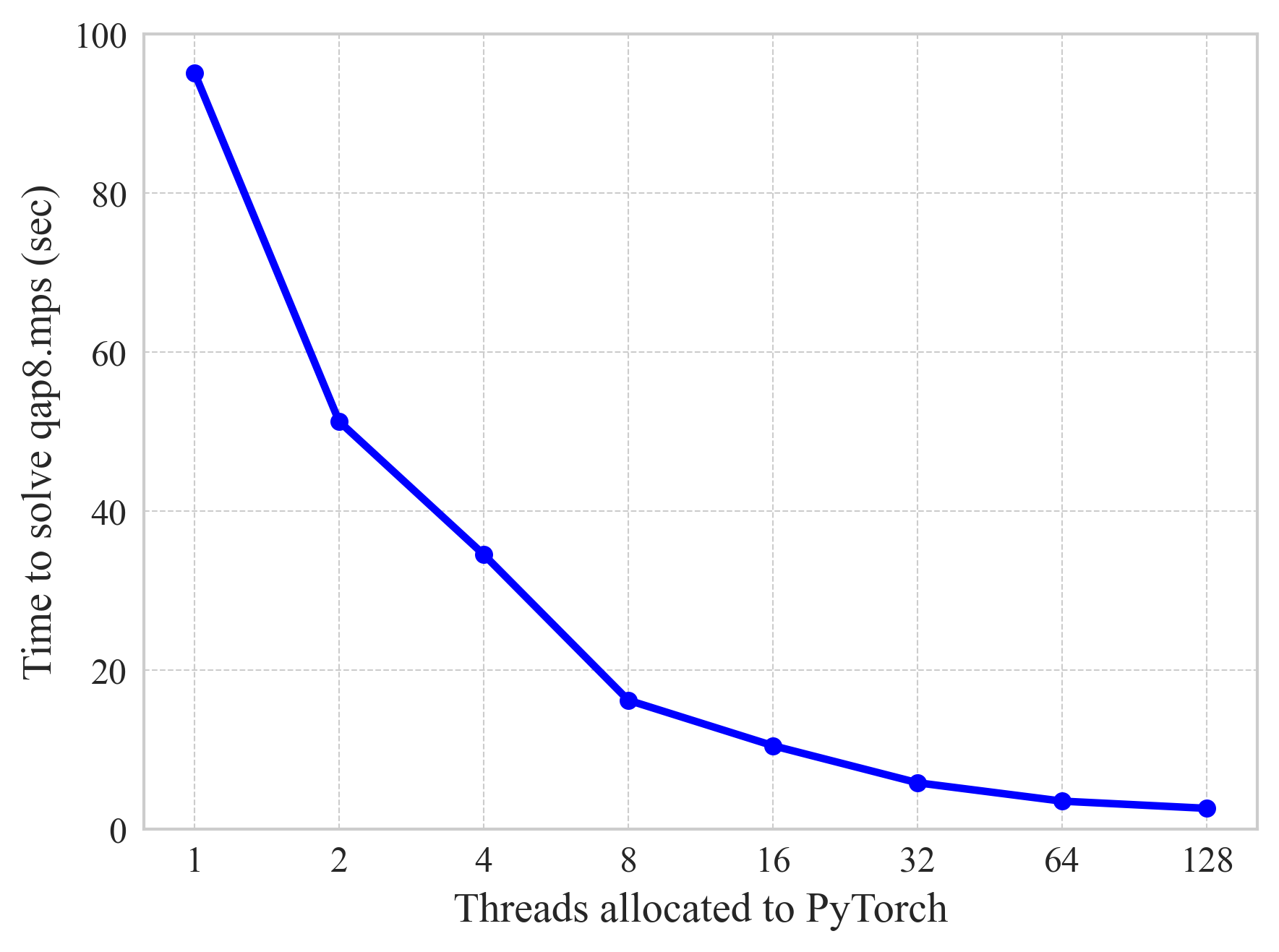}
    \caption{Wall-clock time vs thread count}
    \label{fig:threads}
\end{figure} 

As illustrated in Figure \ref{fig:threads}, increasing the number of threads significantly reduces the computation time, confirming the expected speed-up from multithreading. This successful demonstration provides a strong rationale for leveraging the massively parallel architecture of a GPU to achieve even greater acceleration.
\section{Computational Architecture}
The design of our solver's computational architecture, depicted in Figure~\ref{fig:design}, is fundamentally based on a heterogeneous computing model that strategically allocates tasks to either the CPU or the GPU to exploit their respective strengths. The primary objective of this architecture is to maximize computational throughput by leveraging the GPU's massively parallel processing capabilities for the algorithm's core numerical operations, while simultaneously minimizing the latency overhead associated with data transfers. The workflow follows a hybrid division of labor: the CPU is responsible for inherently sequential tasks such as reading the LP instance from disk, performing any necessary presolving, and handling final output. Once the input data is prepared, it is transferred to the GPU in a single bulk operation.
\begin{figure}[htbp]
        \centering
        \includegraphics[width=0.8\linewidth]{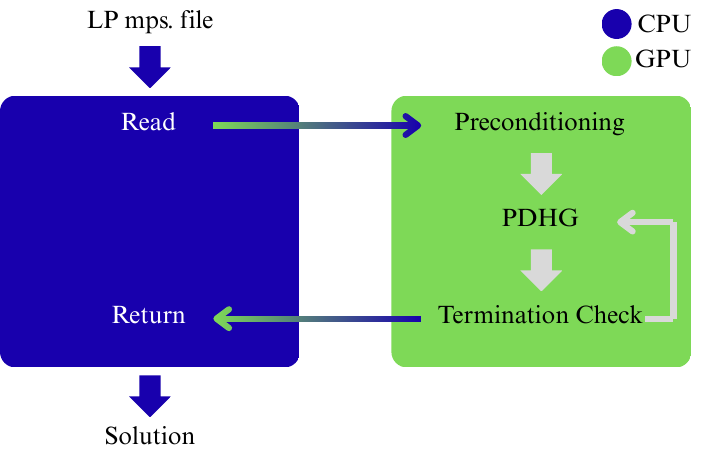}
        \caption{Illustration of the computational architecture of our program}
        \label{fig:design}
\end{figure}

From this point onward, all performance-critical components—including preconditioning to improve numerical conditioning, the restarted PDHG iterations, and termination checks—are executed entirely on the GPU. Meanwhile, the CPU’s role is limited to tasks that are inherently sequential or control-oriented, such as reading the LP instance from an Mathematical Programming System (MPS) file, performing presolve operations, and postprocessing the final output. Importantly, no intermediate results are sent back to the CPU during computation, eliminating transfer-induced latency from the main computational loop. The only CPU–GPU communication after initialization occurs upon algorithm termination, when the final solution is transferred from the GPU back to the CPU for postprocessing and return to the user.

This GPU-resident design leverages the GPU’s massively parallel architecture for the dense matrix–vector operations that dominate PDHG’s runtime, while leaving sequential control and data management to the CPU. By assigning each processor type to the tasks it executes most efficiently, we ensure that computational throughput is maximized and communication overhead is minimized.\\

\section{Benchmark Datasets}
\subsection{Netlib}\label{sec:Netlib}
For general benchmarking and algorithm validation, we utilize problems from the Netlib library \cite{Netlib}, which serves as a standard test set for comparing Linear Programming (LP) software and algorithms. The Netlib collection is a well-known suite of real-world LP problems with 114 feasible and 29 infeasible problems. The Netlib is also notable for its numerical challenges. Ordónez and Freund have shown that a significant percentage of the Netlib instances are ill-conditioned~\cite{ordonez2003computational} , presenting potential difficulties for numerical solvers, and therefore providing a robust challenge to benchmark our LP solver and test the effectiveness of the enhancements.
\subsection{MIPLIB 2017}
To evaluate the performance and scalability of our solver on larger and more structurally complex problems, we used the linear programming relaxations of instances from the MIPLIB 2017 collection~\cite{miplib2017}. Originally created to provide researchers with challenging mixed-integer programming problems, the LP relaxations of this dataset offer a diverse set of large-scale, real-world problems. These instances were specifically chosen to validate our hypothesis that the computational speedup provided by the GPU architecture scales effectively with problem size.

\subsection{Security-Constrained Economic Dispatch (SCED)} \label{sec:SCED}
To assess our solver's performance on a practical, large-scale industrial application, we used a real-world instance of a Security-Constrained Economic Dispatch (SCED) problem. SCED is a critical optimization task in power systems engineering that determines the most cost-effective allocation of electricity generation while satisfying operational and reliability constraints. This specific LP instance, which features 17,682 variables and 17,706 constraints, was provided by Gridmatic, a company focused on optimizing clean energy markets~\cite{gridmatic}. Its significant scale presents a challenging test case that is representative of modern optimization problems encountered in industry.

\section{Code Availability}
To promote reproducibility and facilitate further research, the source code for this project is publicly available in a \href{https://github.com/SimplySnap/torchPDLP/tree/pypi-package}{GitHub repository}. For ease of instillation, the package is also uploaded to the Python Package Index under the name ``torchPDLP'', and so can easily be installed with the command \texttt{pip install torchPDLP}.

\ifthenelse{\boolean{@twoside}}{\myclearpage}{}
\chapter{Benchmarking Results}\label{chap:results}
In this chapter, we present a comprehensive empirical evaluation of our torchPDLP software on a variety of benchmark linear programming problems. We begin by establishing a performance baseline of the standard PDHG algorithm. Following this, we detail the hyperparameter tuning process undertaken to optimize the solver's configuration. We then systematically evaluate the practical impact of each algorithmic enhancement. Finally we assess the fully-integrated solver on standard LP benchmark collections and on a large-scale Security-Constrained Economic Dispatch (SCED) instance to validate its performance on a practical, real-world application with a comparative analysis of its performance across different CPU and GPU architectures. Error tolerances for termination criteria as discussed are set for the most part at $10^{-4}$, but for some smaller experiments they are set to $10^{-2}$.

\section{Baseline Solver}\label{sec:initial solver}
\begin{figure}[h]
    \centering
    \includegraphics[width=0.8\textwidth]{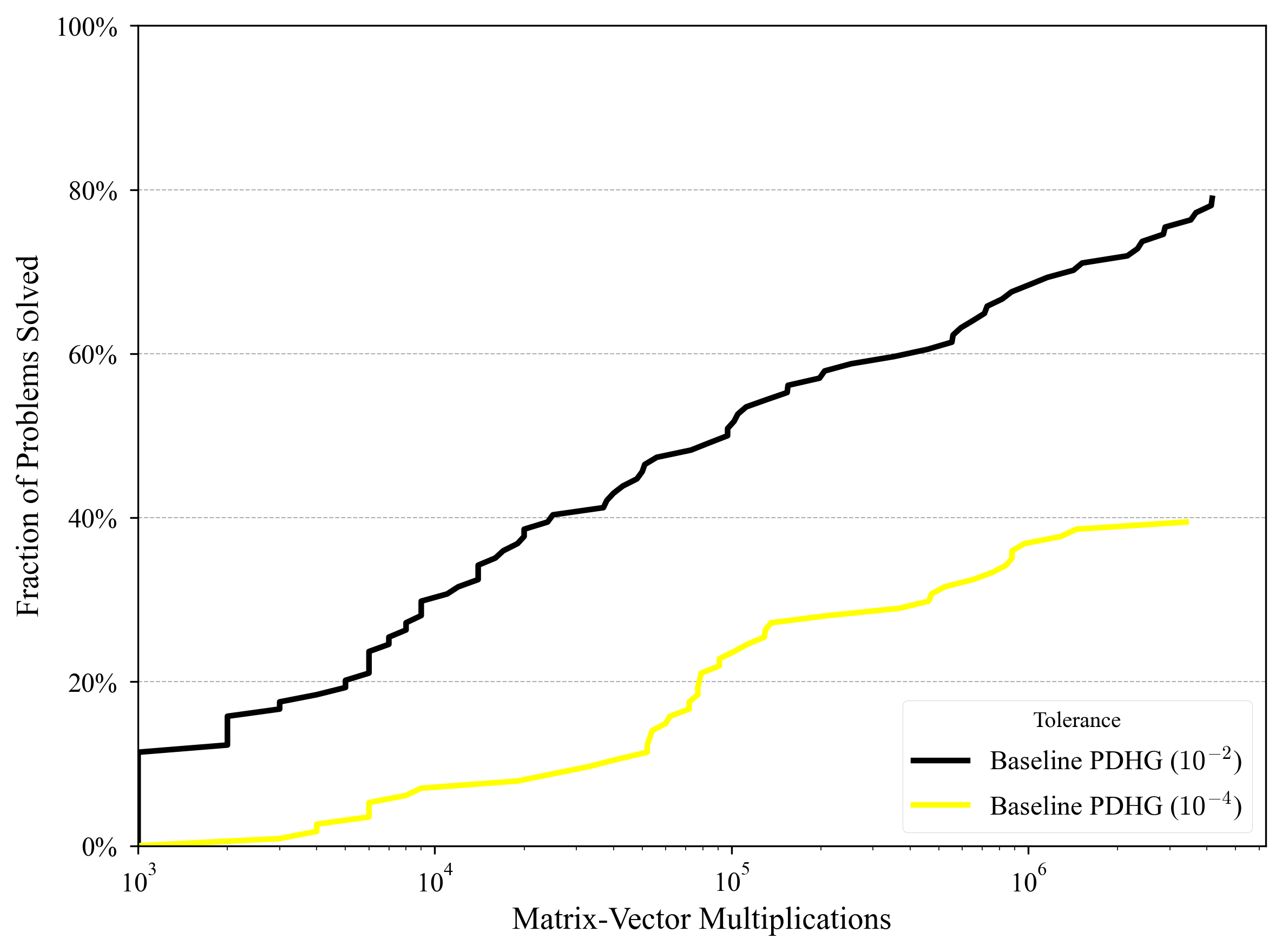}
    \caption{Baseline PDHG Performance}
    \label{fig:iterations}
\end{figure} 

We benchmark the PDHG algorithm without any practical enhancements from sections \ref{sec:pdlp}, and in later sections, we use it as a baseline to measure the effectiveness of the enhancements. The parameters we set for this implementation are a primal weight of $\omega=\frac{\|c\|_2}{\|q\|_2}$, a step size of $\eta=0.9/\|K\|_2$, and an extrapolation parameter of $\theta=1$. These choices are justified by numerical results that we detail in section \ref{sec:parameters}. 

Figure \ref{fig:iterations} compares the convergence rate of the initial solver in terms of iterations at two different error tolerance levels on the feasible problems of the Netlib dataset. In this test, we check if the solution is within the error tolerance every thousand iterations according to equations \eqref{term1}, \eqref{term2} \& \eqref{term3}, and set no maximum on the number of iterations it will run. However, for the sake of time, we set a maximum time limit of 750 seconds per problem. Note that this graph looks very similar to the results of \cite{applegate2} for their baseline PDHG solver benchmarked on the same dataset.

\section{Hyperparameter Tuning}\label{sec:parameters}
There are several hyperparameters in the PDHG algorithm that can vary without changing the theoretical convergence guarantees. We performed sensitivity tests on these parameters to determine what the optimal value was for our algorithm to minimize the number of iterations for convergence.

For the results in this section, we used a sample of twenty feasible Netlib problems, that can be found on our GitHub \cite{github}, and ran our solver with a tolerance of $10^{-2}$ on these problems with varying values for the parameter of interest. Every other parameter was held constant at the values specified in section \ref{sec:initial solver}.

\begin{figure}[h]
    \centering
    \begin{minipage}[b]{0.32\textwidth}
        \centering
        \includegraphics[width=\linewidth]{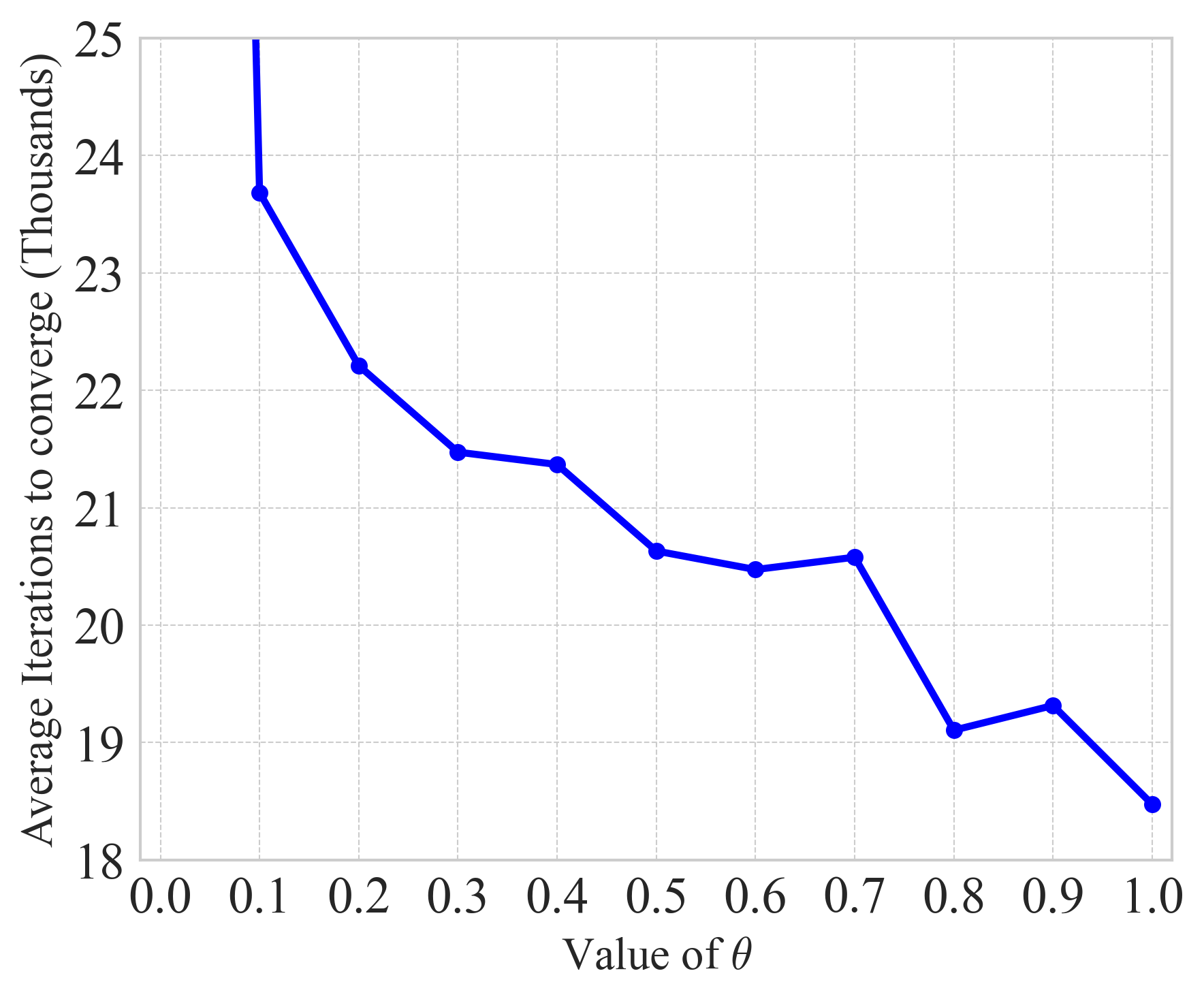}
        \caption{Effect of varying the extrapolation parameter}
        \label{fig:theta}
    \end{minipage}
    \hfill
    \begin{minipage}[b]{0.32\textwidth}
        \centering
        \includegraphics[width=\linewidth]{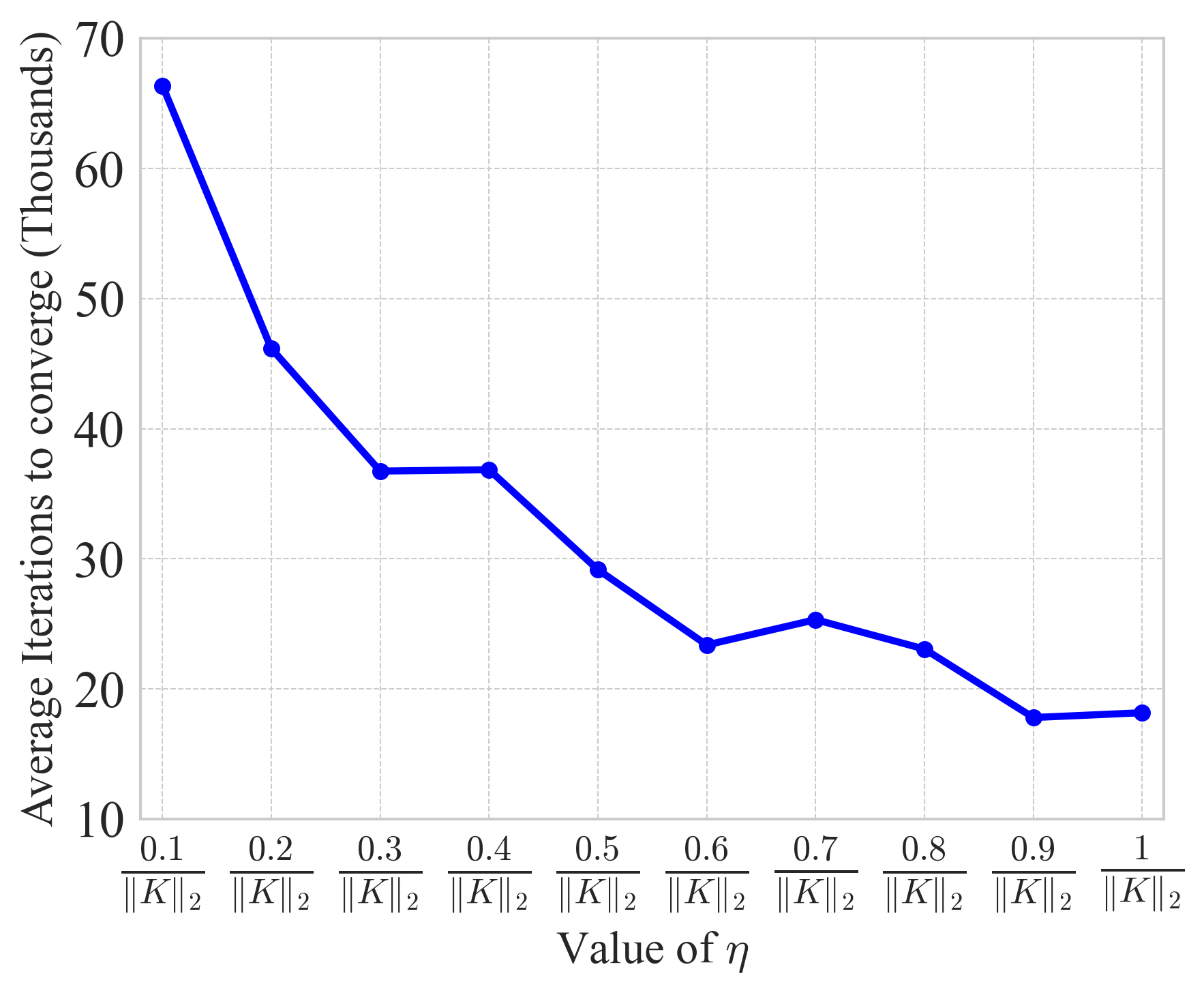}
        \caption{Effect of varying the step size parameter}
        \label{fig:eta}
    \end{minipage}
    \hfill
    \begin{minipage}[b]{0.32\textwidth}
        \centering
        \includegraphics[width=\linewidth]{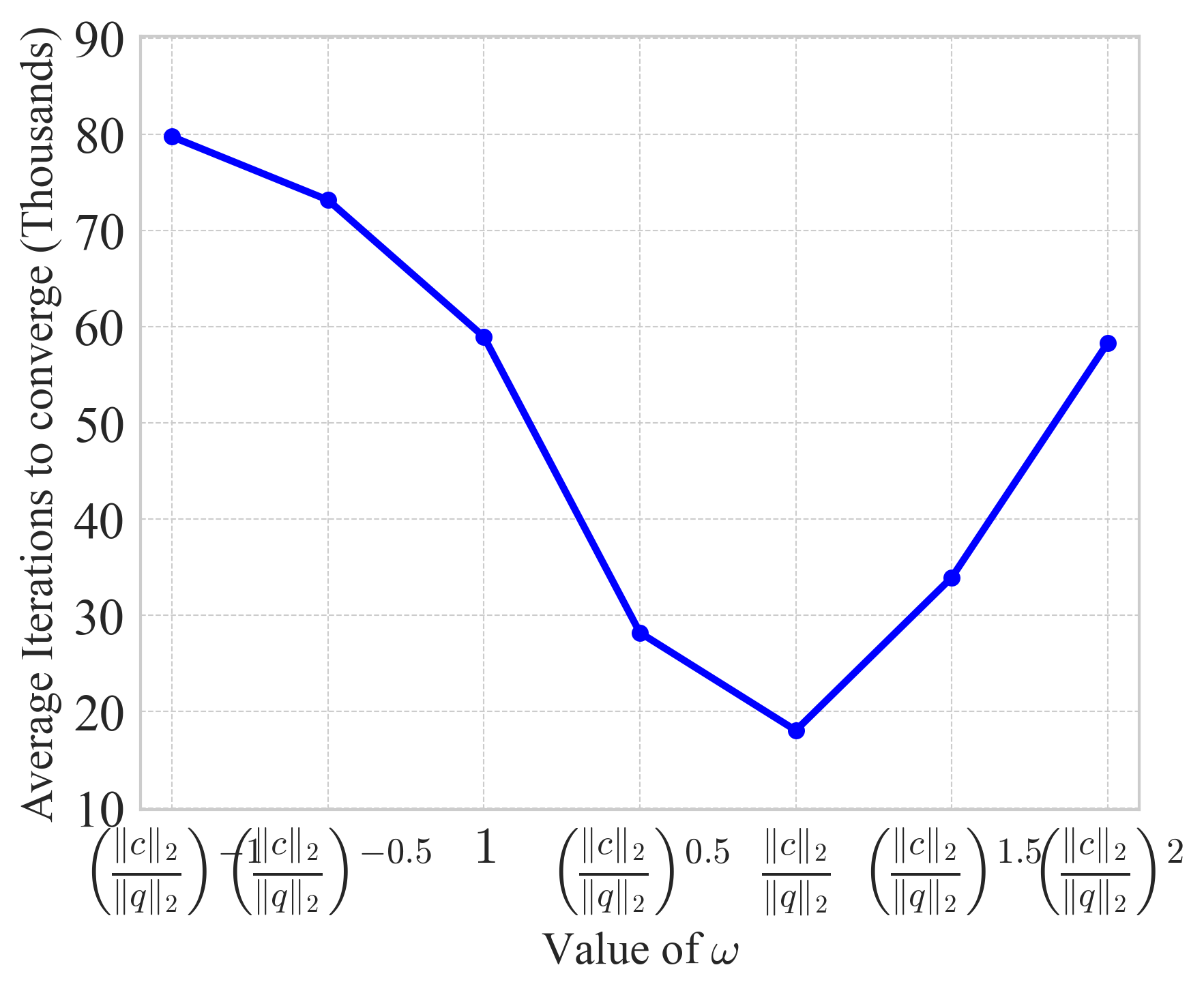}
        \caption{Effect of varying the primal weight parameter}
        \label{fig:omega}
    \end{minipage}
\end{figure}

\subsection{Extrapolation Parameter}
In the PDHG algorithm, convergence is guaranteed theoretically for any $\theta \in [0,1]$. It can be shown that for $\theta=0$ the algorithm converges with a rate of $\mathcal{O}(1/\sqrt{N})$, and for $\theta=1$ the rate is $\mathcal{O}(1/N)$ \cite{chambolle2011first}. These rates of convergence are supported by our numerical experiments, which suggest that increasing $\theta$ from 0 to 1 results in continuous improvement in the convergence rate of the PDHG algorithm.

Figure \ref{fig:theta} shows the average number of iterations it took the problems in our sample to converge with $\theta$ increasing from 0 to 1 with steps of size 0.1. The number of iterations at $\theta=0$ was $57.6$ thousand, over double the iterations at $\theta = 0.1$, and so it was omitted from this graph. As expected, setting $\theta=1$ resulted in the fastest convergence of the algorithm, and so this is the value we set in our solver.

\subsection{Step Size Parameter}
For convergence of the algorithm to be guaranteed, it is required that the step size $\eta \leq 1/\|K\|_2$. Traditionally, this requirement is met by setting $\eta=1/\|K\|_2$, but we test the difference in convergence achieved by setting smaller values of $\eta$ by changing the multiplier in the numerator we divide by $\|K\|_2$.

Figure \ref{fig:eta} shows the average number of iterations it took the problems in our sample to converge with $\eta$ increasing from $0.1/\|K\|_2$ to $1/\|K\|_2$ with steps of size $0.1/\|K\|_2$. For the most part, increasing the step size resulted in faster convergence, until getting to $\eta=1/\|K\|_2$ which resulted in convergence that was slightly slower than that of $\eta=0.9/\|K\|_2$. And so this is the value we set for the step size in our initial solver.
\subsection{Primal Weight Parameter}
The primal weight $\omega$ has no restrictions on its value to guarantee convergence of the algorithm, but we've found that changing its value does affect the rate of convergence. One option is to weight the primal and dual step sizes the same, corresponding to the value $\omega=1$. However, this choice disregards the relative magnitudes of the primal and dual LP and so inspired by \cite{applegate2}, we tested our algorithm with the primal weight set at $\omega =\frac{\|c\|_2}{\|q\|_2}$ (unless either of $\|c\|_2$ or $\|q\|_2$ is within the small zero tolerance ($10^{-6}$), in which case we set the primal weight as $\omega=1$). Essentially, this choice scales each primal iteration with the magnitude of the primal objective function, compared to that of the dual, leading to faster convergence as the algorithm's iterations are better suited to the constraint norms for both the primal and dual.

For completeness, we tested raising $\frac{\|c\|_2}{\|q\|_2}$ to a variety of powers from $-1$ to $2$ and these results are shown in Figure \ref{fig:omega}. As expected, setting $\omega=\frac{\|c\|_2}{\|q\|_2}$ resulted in the fastest convergence, and so this is the value we set in our algorithm.


\section{Improvements with adaptations}\label{sec:pdlp_results}
As noted in section \ref{sec:pdlp}, we have implemented a number of adaptations to the base algorithm to improve its performance. In this section, we compare the baseline solver as outlined in section \ref{sec:initial_solver} with the change in performance induced by implementing a particular adaptation. We use an error tolerance of $10^{-2}$, and for the sake of time, we set a max iteration count of $10^5$. The dataset we used was the feasible problems from the Netlib library. The experiments in this section were run on the AMD Instinct MI210 GPU.

As we did not develop the presolving software PaPILO, we do not document the performance gain from using the software. Numerical experiments run using PaPILO can be found here \cite{papilo}. 
\subsection{Preconditioning}
As discussed in section \ref{sec:Netlib}, most of the Netlib problems are `ill-conditioned', meaning the magnitudes in their constraint matrix vary wildly between entries. Therefore, we expect that implementing our preconditioning step, algorithm \ref{alg:ruiz_scaling}, should have a significant effect on the performance on this benchmark set, which is exactly what we observe in figure \ref{fig:preconditioning}. This graph shows the increased convergence speed of the PDHG algorithm after implementing preconditioning.

\subsection{Adaptive Step Size}
The adaptive step size enhancement adds the most computational expense to the solver, as is clear from algorithm \ref{alg:adaptive_step}. This expense can be seen in Figure \ref{fig:adaptive_step_size}, which shows an increase in the time to solve easier problems with the enhancement. But for the harder problems, which take the PDHG algorithm more than ten seconds to solve with a fixed step size, the implementation of the adaptive step size algorithm leads to faster convergence and more problems solved.

\subsection{Primal Weight Updating}
Implementing primal weight updating as described in section \ref{sec:primal_weight} would require first implementing adaptive restarting as described in section \ref{sec:restart}. However, since we want to measure the effect that each individual enhancement has on the rate of convergence, in this test, our primal weight updating works a little differently.

Instead of updating the primal weight every restart with $\Delta_x^n$ and $\Delta_y^n$ set as the change in the primal and dual variables respectively since the last restart, we update the primal weight each time the algorithm checks the error of the current iterate, and we set $\Delta_x^n$ and $\Delta_y^n$ as the change in those respective variables since the last error tolerance check.

As shown in Figure \ref{fig:primal_weight}, this implementation had almost no effect on the convergence rate of the algorithm. One explanation of this could be that there is some reason that primal weight updating must be done in conjunction with the adaptive restarting scheme to be effective.

Another possible explanation is that the main benefit of using this enhancement comes just from initializing the primal weight with $\omega=\frac{\|c\|_2}{\|q\|_2}$, which we already do in our baseline algorithm. In the literature on this method \cite{applegate2}, the baseline solver they compare their adaptations with sets the primal weight as $\omega=1$, and they never compare just the effect of initializing the primal weight without implementing the updating algorithm. More testing will be needed to better understand this discrepancy.

\subsection{Adaptive Restarting}
Given the restart criteria in section \ref{sec:restart}, it is theoretically guaranteed that restarted PDHG will perform as well or better than the baseline algorithm in terms of the number of iterations to converge, though this comes at the expense of costly KKT error calculations. However, as shown by Figure \ref{fig:restart}, the decrease in iterations outweighs the computational expense and the restarted algorithm is consistently faster than the baseline and solves a higher percentage of problems.

\subsection{torchPDLP}
As described in section \ref{sec: torchPDLP} our implementation torchPDLP incorporates all four of the previously discussed enhancements, as well as presolving capabilities. In the implementation, however, only the adaptive restarting enhancement is fully integrated into the solver by default. The other three enhancements are optional and can be enabled based on user preference. It is important to note that the subsequent enhancements described in sections \ref{sec:infeas_detect} (infeasibility detection) and \ref{sec:Fishnet} (Fishnet Casting) are not included because the current benchmarks are run exclusively on feasible problems, and the Fishnet Casting procedure is still under development. 
\begin{figure}[htbp]
    \centering

    \begin{subfigure}[t]{0.48\textwidth}
        \centering
        \includegraphics[width=\textwidth]{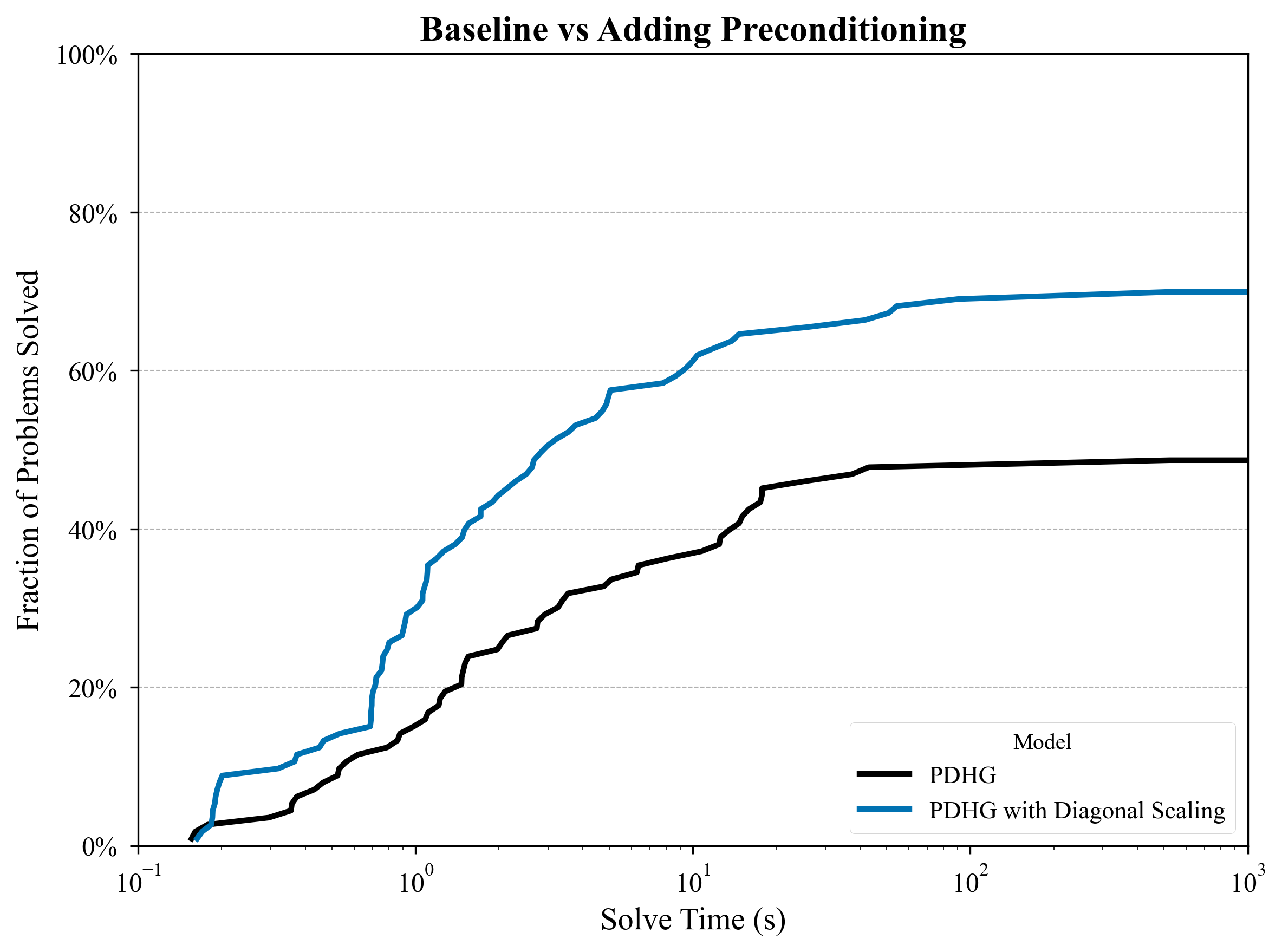}
        \subcaption{Preconditioning}
        \label{fig:preconditioning}
    \end{subfigure}
    \hfill
    \begin{subfigure}[t]{0.48\textwidth}
        \centering
        \includegraphics[width=\textwidth]{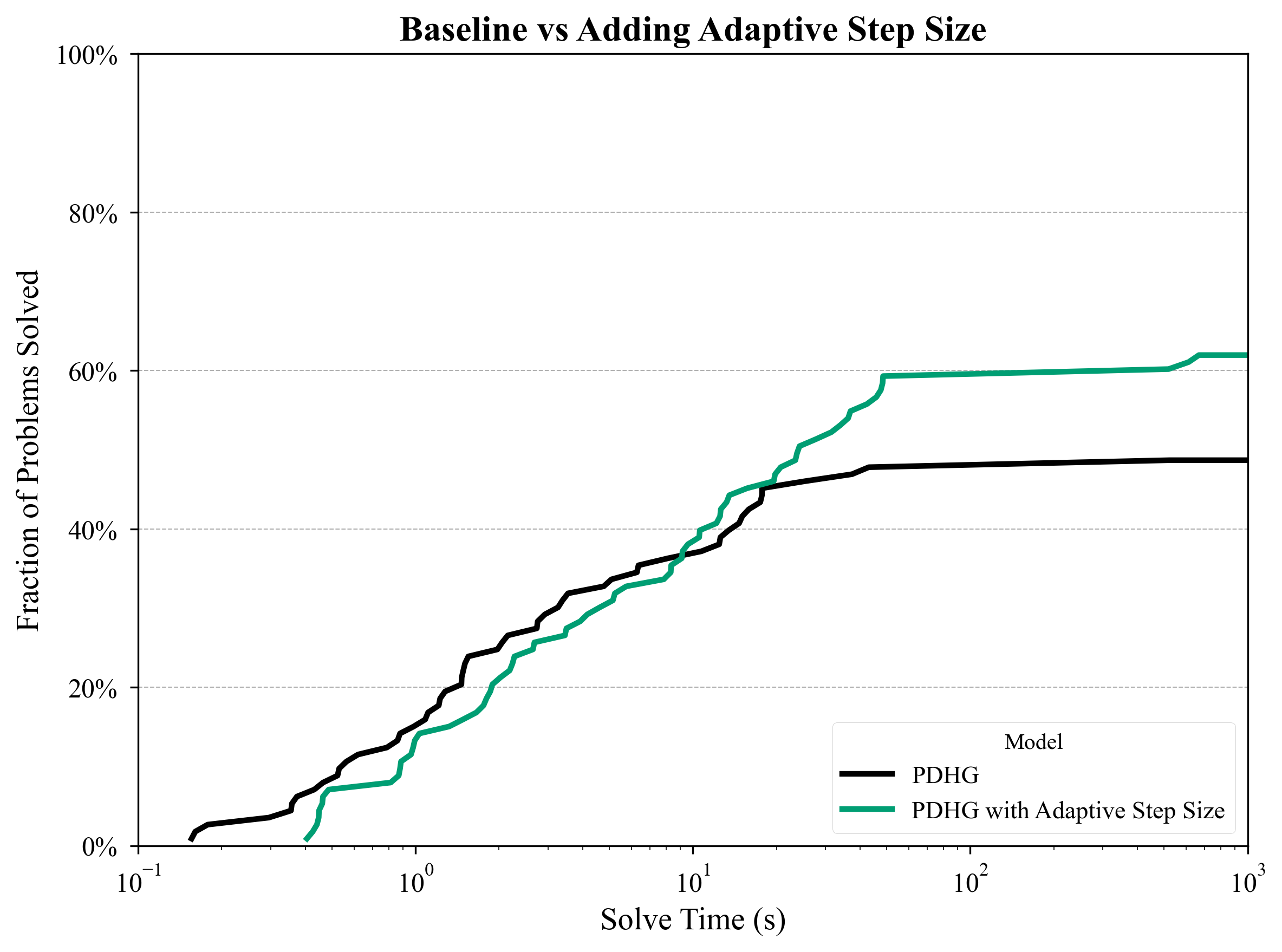}
        \subcaption{Adaptive step size}
        \label{fig:adaptive_step_size}
    \end{subfigure}

    \vspace{0.3cm} 

    \begin{subfigure}[t]{0.48\textwidth}
        \centering
        \includegraphics[width=\textwidth]{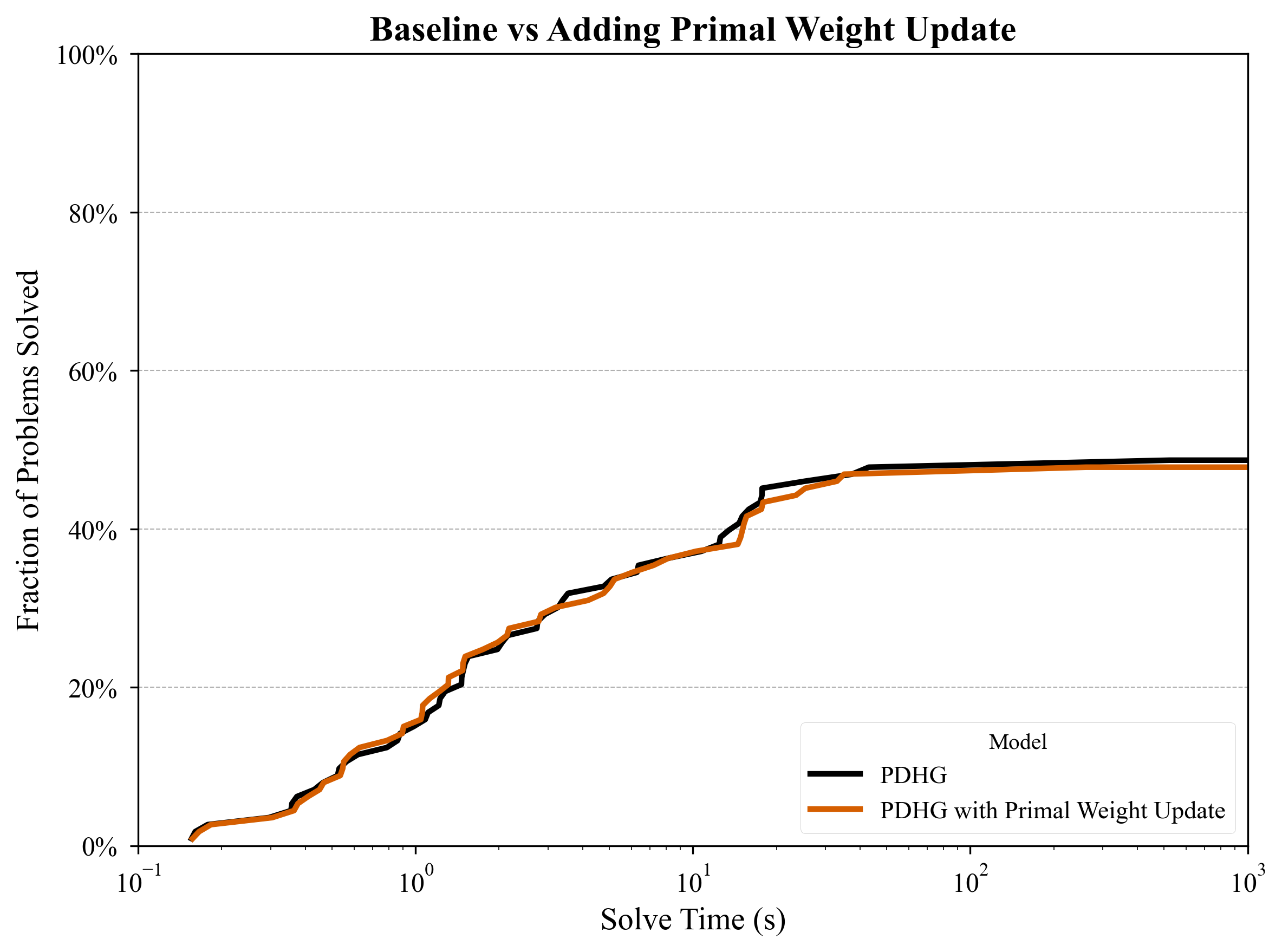}
        \subcaption{Primal weight updates}
        \label{fig:primal_weight}
    \end{subfigure}
    \hfill
    \begin{subfigure}[t]{0.48\textwidth}
        \centering
        \includegraphics[width=\textwidth]{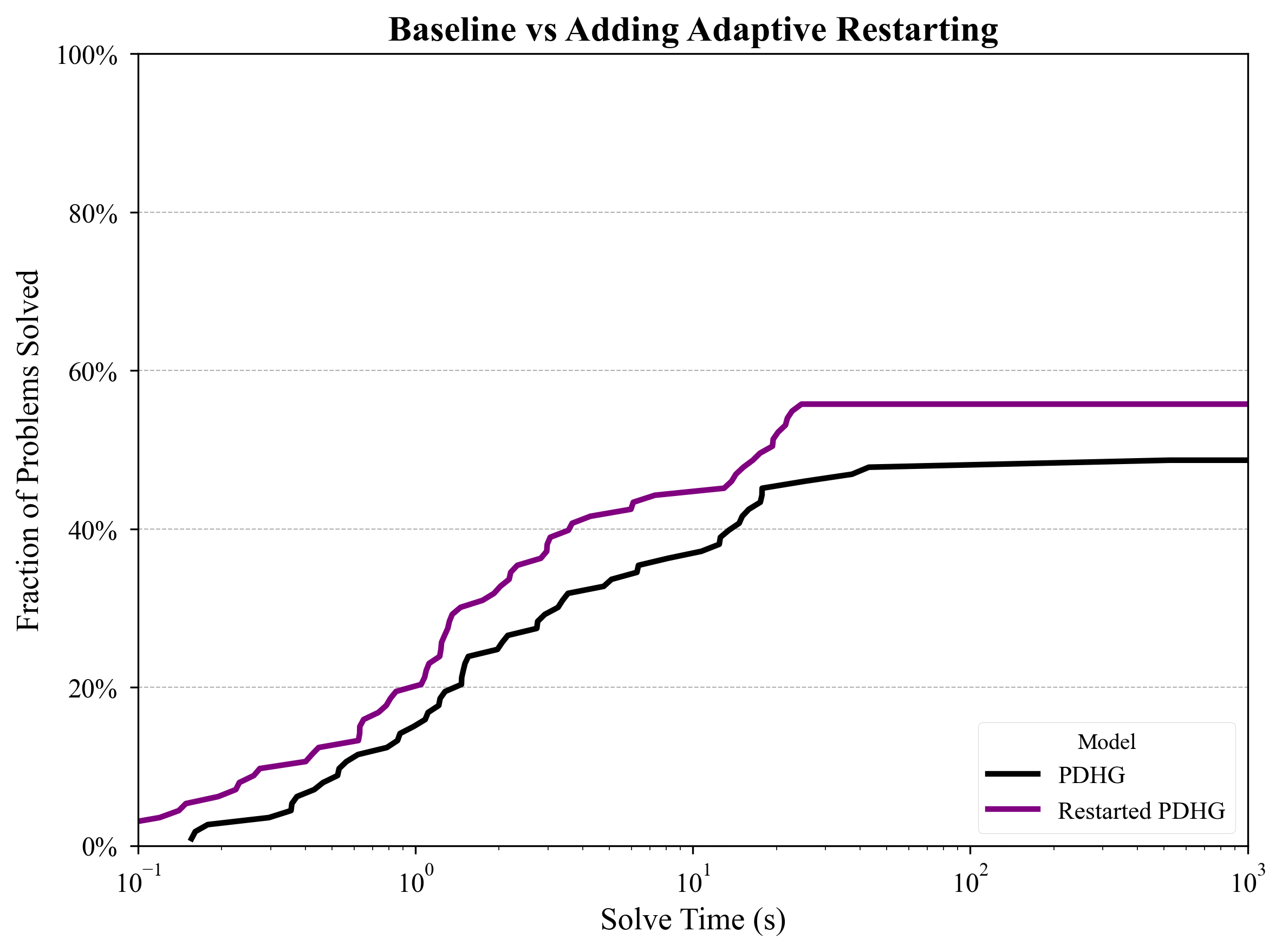}
        \subcaption{Adaptive restarting}
        \label{fig:restart}
    \end{subfigure}

    \vspace{0.5cm} 
    \includegraphics[width=0.8\textwidth]{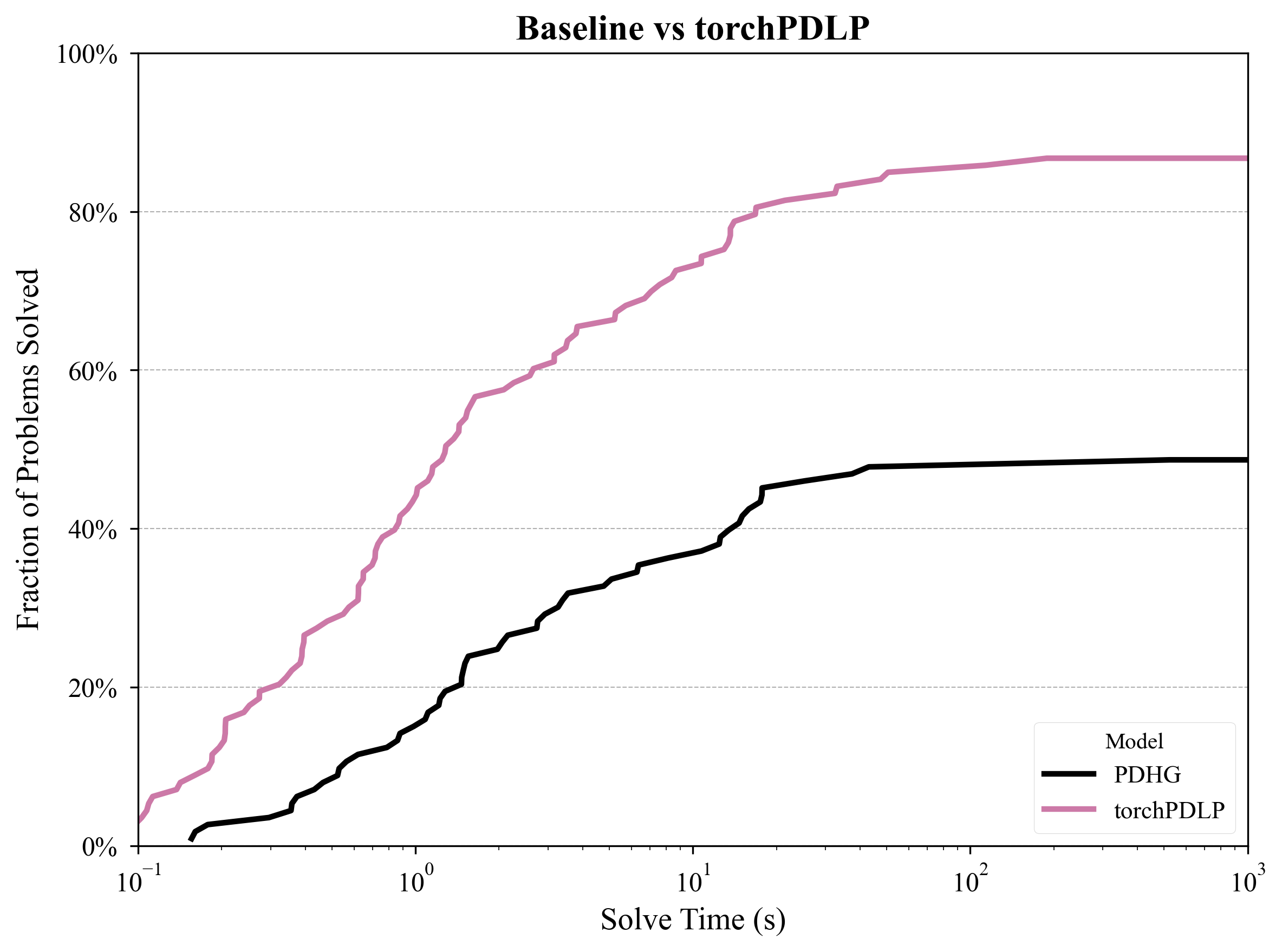}
    \caption[Performance improvements of PDHG solver]{Performance improvements of PDHG solver under different enhancements. Top four subfigures show individual enhancements, while the bottom figure compares PDHG with no enhancements vs torchPDLP.}

    \label{fig:pdhg_enhancements_full}
\end{figure}

    



\subsection{Infeasibility Detection}
Figure~\ref{fig:infeasibility} reports the cumulative fraction of detected infeasible instances as a function of solve time on the Netlib infeasible set \cite{Netlib_infeas}. A run is considered a detection when the solver produces a primal or dual infeasibility certificate. Within 100,000 iterations, torchPDLP successfully detects 22 of the 29 infeasible instances, with the exceptions being \texttt{bgindy}, \texttt{chemcom}, \texttt{forest6}, \texttt{greenbea}, \texttt{mondou2}, \texttt{qual}, and \texttt{reactor}. This result demonstrates a strong detection capability for this benchmark set.
\begin{figure}[htpb]
    \centering
    \includegraphics[width=0.8\linewidth]{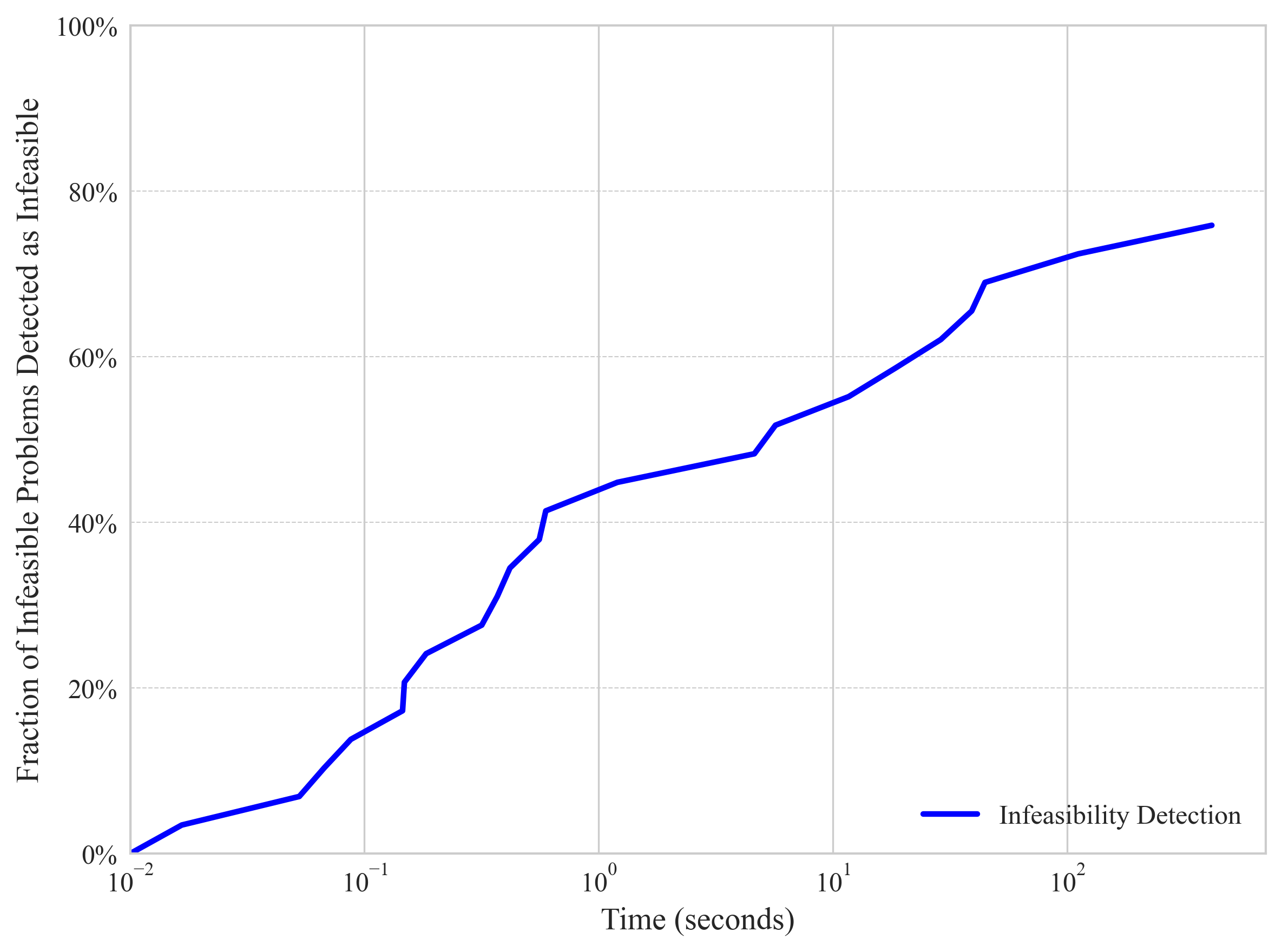}
    \caption{Cumulative fraction of infeasible instances detected vs solve time, tol = $10^{-4}$}
    \label{fig:infeasibility}
\end{figure}

\subsection{Fishnet Casting}\label{4.3.7}
As noted in Chapter 2, fishnet casting is an original, experimental adaptation designed to provide a better starting point for torchPDLP. We collected results from both an instance of Security Constrained Economic Dispatch, with tens of thousands of variables and constraints, and a subset of Netlib problems. Our spectral casting function initially cast 32 random points, regardless of problem size, and we culled 50\% of points at any given cull\_points() call. Similarly, to repopulate a single primal point, our convex combination rule sampled k random normal variables for k the number of primal/dual points, normalized them so they summed to one, and created a new primal point with $\sum_{i=1}^{32}\omega_i \cdot x_i$ for $x_i$ the $i^{th}$ point and $\omega^{i}$ the normalized $i^{th}$ weight. These same weights were applied to the equivalent dual variables in the same manner to produce a comparable point in dual-space.

For the subset of Netlib problems, we found that fishnet casting did not provide a consistent speedup, sometimes resulting in slower convergence or more iterations, and sometimes providing a speedup. This is likely due to the spectral norm not being a good estimate for enclosing the primal feasible region; crucially, the time taken to run fishnet casting was \textit{not} significant for all tested problems, as even for larger problems, fishnet casting never took more than 2 seconds to run. Many Netlib problems are ill-conditioned \cite{ordonez2003computational}, and so the spectral norm might not accurately reflect problem difficulty.

Some improvements might be to make the origin part of the initial spectral casting set, so our set of points is \textit{no worse} than without the adaptation, and scaling the casting ellipsoid with the constraint vectors to more accurately capture the feasible region.

\begin{figure}[htbp]
    \centering
\includegraphics[width=0.8\linewidth]{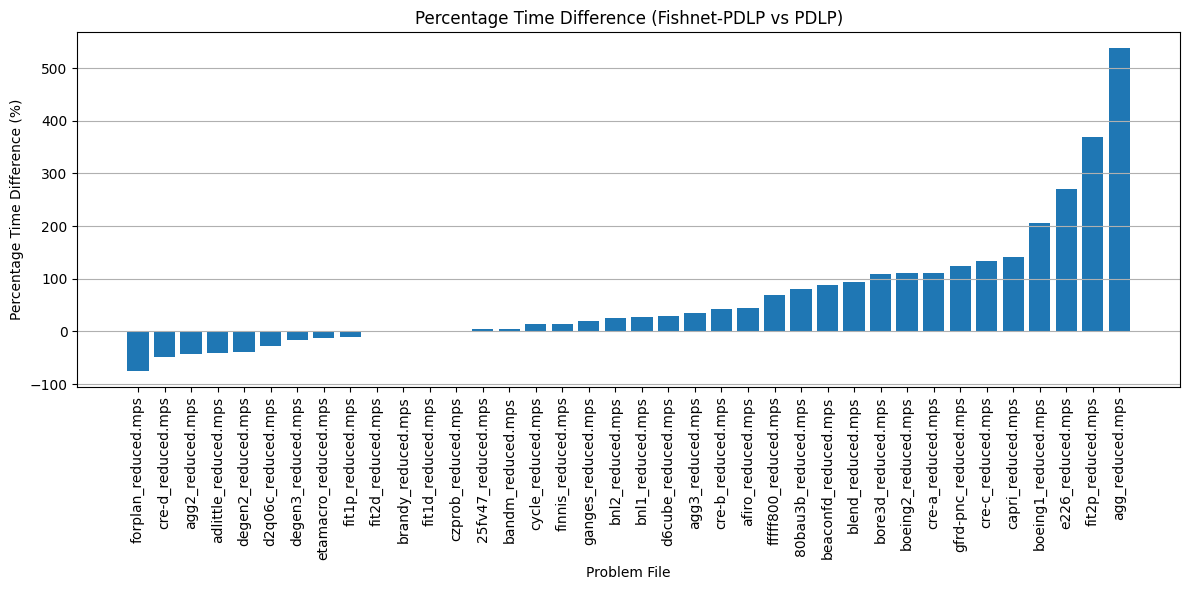}
    \caption{Netlib sample with/without fishnet, percent time change}
    \label{fig:fishnet_vs_normal}
\end{figure}
\vspace{-0.8cm}
\section{Comparison between CPU and GPU}
Building upon the multithreading analysis in Section \ref{multithread}, we conducted a comprehensive performance comparison between torchPDLP executing on the AMD EPYC 7V13 CPU and the AMD MI325X GPU using the standardized datasets. Our evaluation demonstrates that the MI325X GPU successfully solves a greater number of problems within the one-hour time limit while maintaining the same convergence tolerance of $10^{-4}$. 

\begin{figure}[h]
    \begin{subfigure}[t]{0.49\textwidth}
        \centering
        \includegraphics[width=\textwidth]{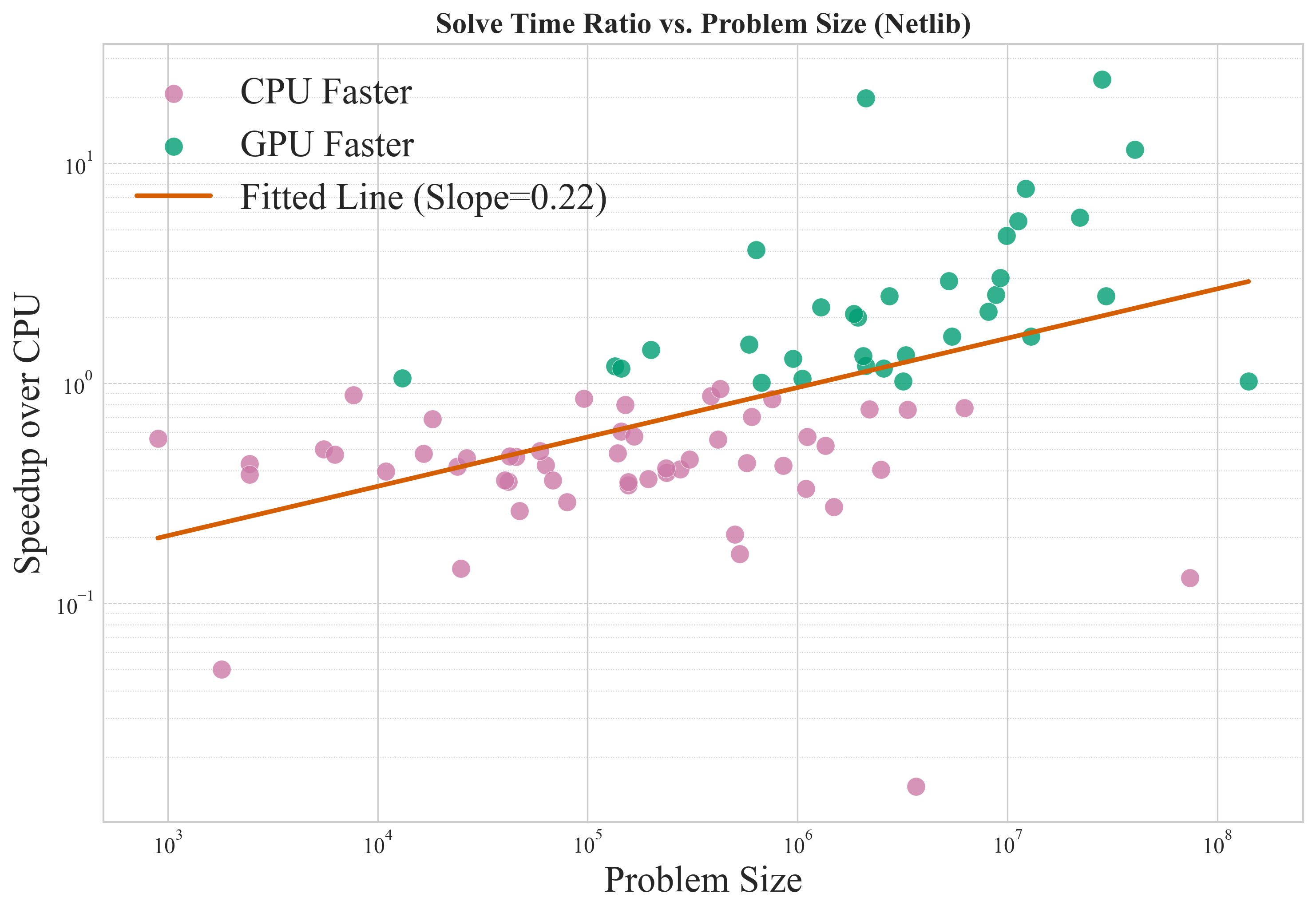}
        \caption{Netlib dataset}
        \label{fig:Netlib_ratio}
    \end{subfigure}
    \hfill
    \begin{subfigure}[t]{0.49\textwidth}
        \centering
        \includegraphics[width=\textwidth]{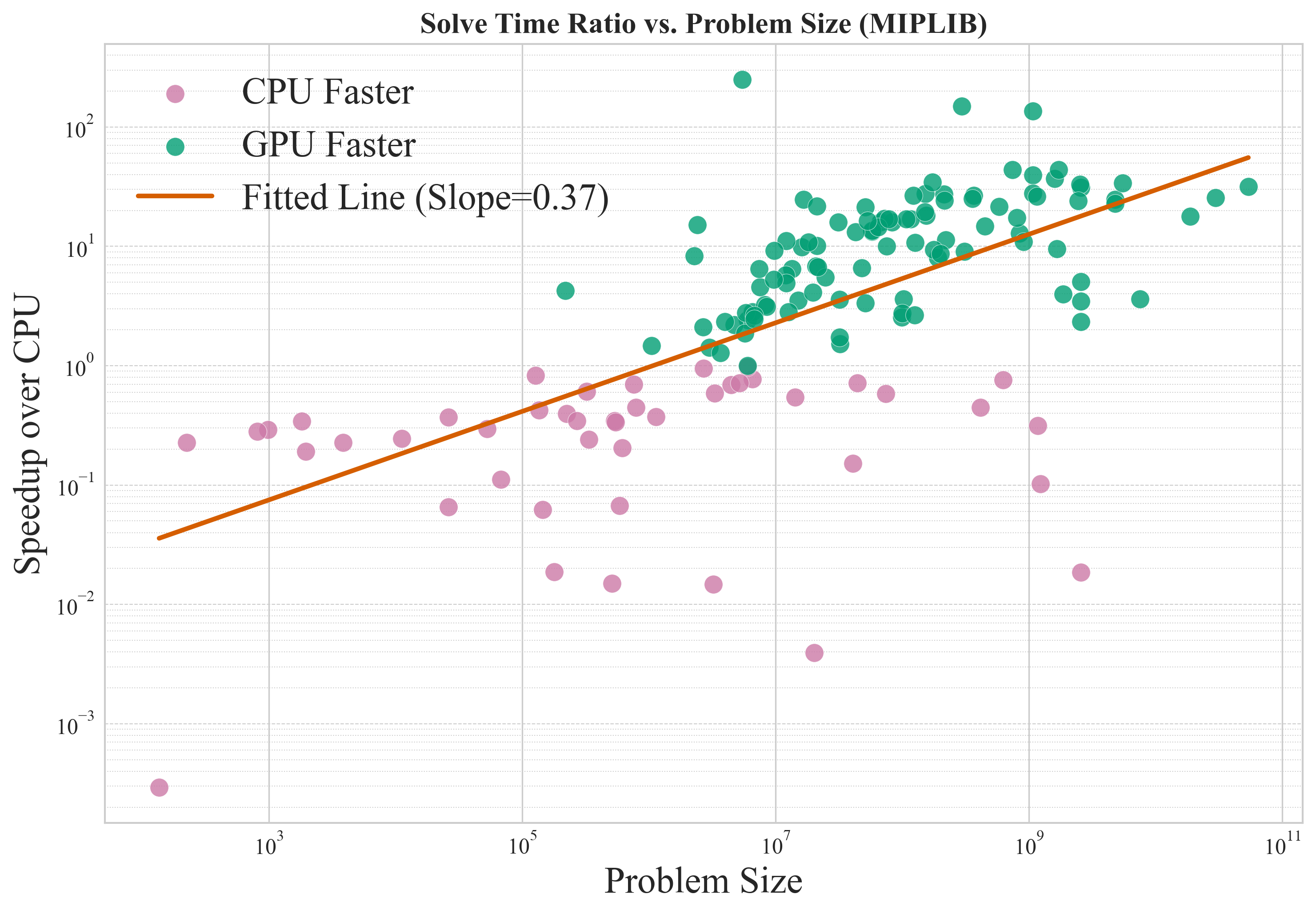}
        \caption{MIPLIB dataset}
        \label{fig:mip_ratio}
    \end{subfigure}

    \caption{CPU vs GPU performances for torchPDLP.}
    \label{fig:cpu_gpu_combined}
\end{figure}

The computational advantage of GPU-based solving becomes increasingly pronounced with problem size, as illustrated in Figures \ref{fig:Netlib_ratio} and \ref{fig:mip_ratio}, where speedup over CPU is defined as the ratio of the time to solve on CPU over the time to solve on GPU. This scaling behavior reflects the GPU's superior capacity for parallel computation, where larger problems provide more opportunities for effective parallelization of the underlying linear algebra operations. 

Figure \ref{fig:Netlib_cpu} illustrates this phenomenon of GPU kernel launch overhead, which introduces a fixed initialization cost before optimal computational throughput is achieved. This overhead becomes negligible relative to total solve time as complexity increases, enabling the GPU's massively parallel architecture to fully exploit the computational workload. 

\begin{figure}[htbp]
    \centering

    \includegraphics[width=0.8\textwidth]{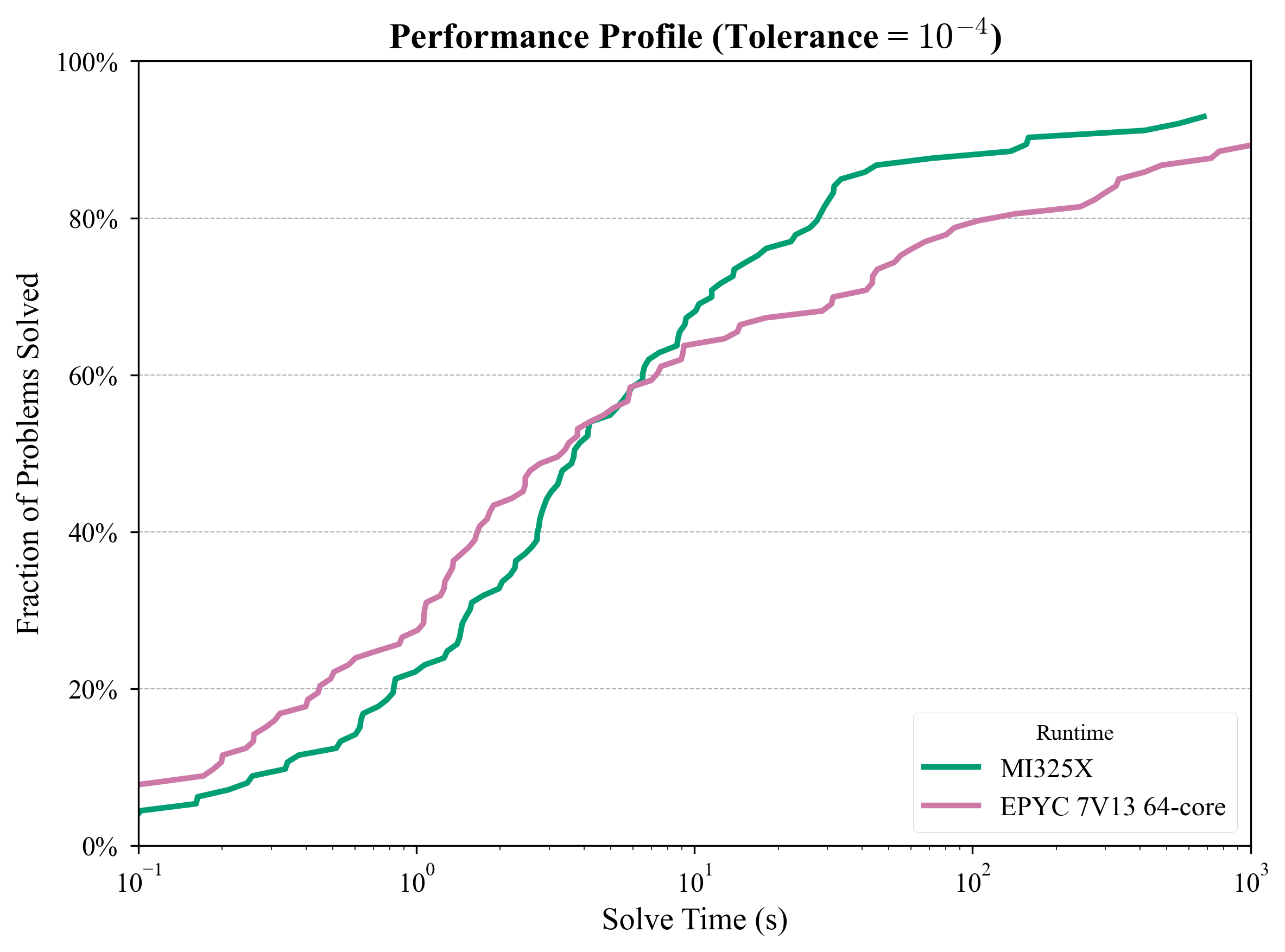}
    \caption{CPU vs GPU torchPDLP Performances on Netlib}
    \label{fig:Netlib_cpu}
\end{figure}

\vspace{-0.8cm}
\section{Comparison between GPUs}
Figure \ref{fig:325_210_Netlib} illustrates that the AMD MI325X achieves a faster overall time-to-solution on the Netlib dataset compared to the MI210, and it exhibits substantially greater computational throughput. While the MI325X completes a higher total number of matrix-vector multiplications per solve, its architecture executes these parallel operations in significantly less cumulative time. 
\begin{figure}[h]
    \centering
\includegraphics[width=\textwidth]{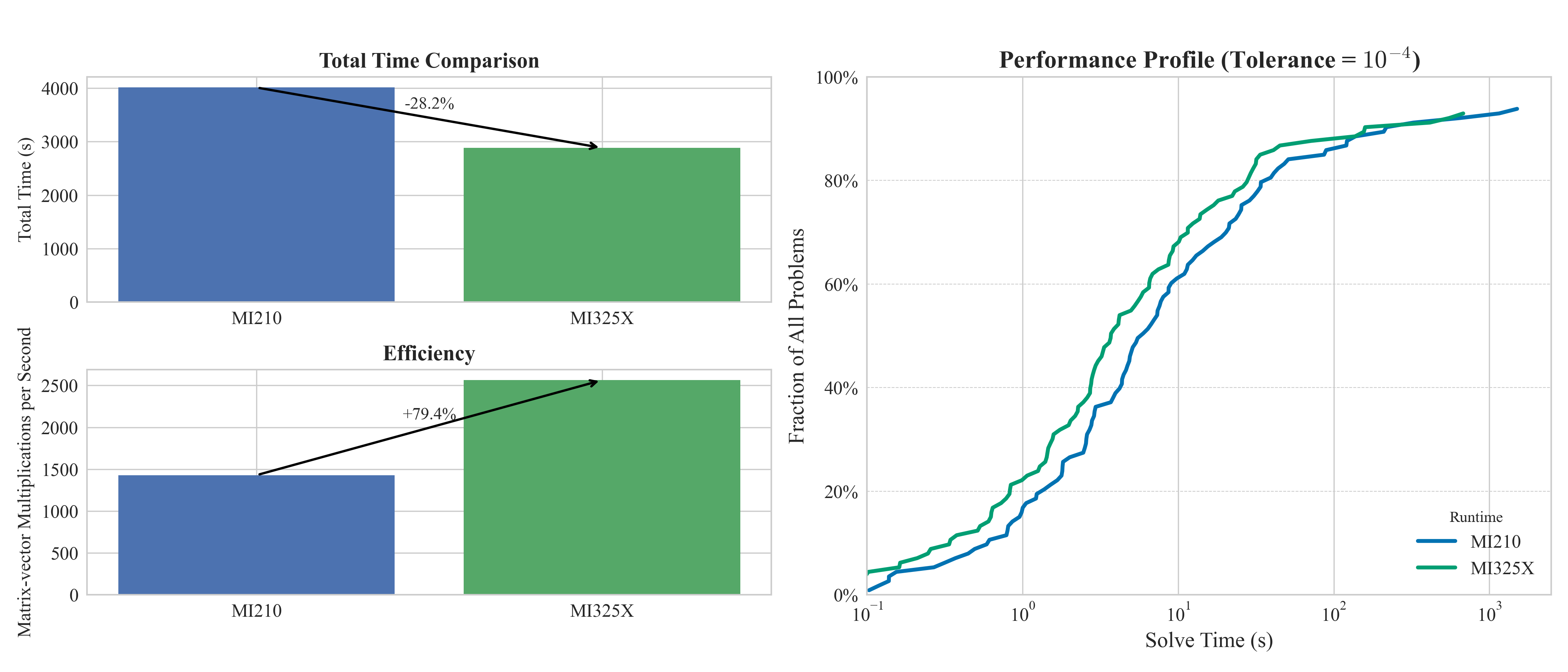}
    \caption{AMD GPU torchPDLP Performances on Netlib}
    \label{fig:325_210_Netlib}
\end{figure}

A key advantage of implementing the solver in PyTorch is the inherent code portability, which facilitates direct, equitable comparisons across hardware from different vendors. The PyTorch framework enables the same torchPDLP codebase to run seamlessly on NVIDIA hardware without manual code modification.

Figure \ref{fig:325_a100_Netlib} presents this cross-vendor analysis, comparing the performance of the AMD MI325X against an NVIDIA A100 GPU (accessed via Google Colab). The results show that the MI325X consistently solves a larger fraction of the Netlib benchmark set in less time than the A100, demonstrating a highly competitive performance profile for linear programming tasks.
\begin{figure}[h]
    \centering
\includegraphics[width=0.8\textwidth]{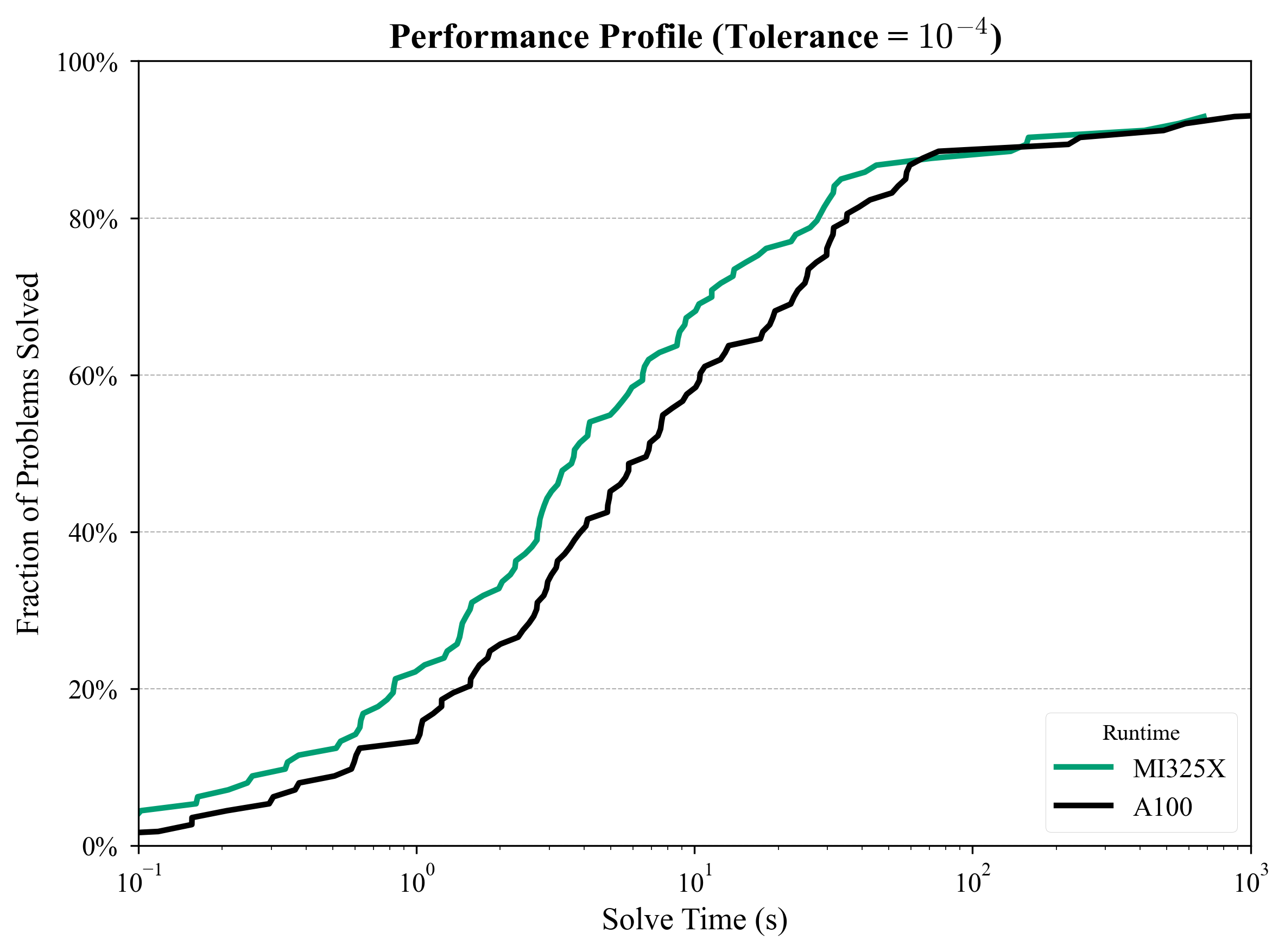}
    \caption{AMD vs NVIDIA GPU torchPDLP Performances on Netlib}
    \label{fig:325_a100_Netlib}
\end{figure}
\vspace{-0.8cm}
\section{SCED Instance}
\begin{table}[htbp]
\centering
\begin{tabular}{l l}
\hline
\textbf{Device} & \textbf{Time (sec)} \\
\hline
EPYC 7V13 64-Core & 20801.16 \\
MI325X            & 579.50 \\
MI210             & 1007.10 \\
A100              & 727.04 \\
\hline
\end{tabular}
\caption{Time to solve the instance to a tolerance of $10^{-4}$}
\label{tab:SCED}
\end{table}
To evaluate torchPDLP on a real-world application, we used a large-scale SCED instance provided by \emph{Gridmatic}, as mentioned in Section \ref{sec:SCED}. This industrial LP problem serves as an excellent benchmark for comparing solver performance across four distinct hardware accelerators: an AMD EPYC 7V13 CPU, an AMD Instinct MI210 GPU, an AMD Instinct MI325X GPU, and an NVIDIA A100 GPU. Table \ref{tab:SCED} shows the time it took for torchPDLP to solve this instance and Figure \ref{fig:SCED} compares the convergence of the solver as running on these devices.
\begin{figure}[htbp]
\centering
\includegraphics[width=0.8\textwidth]{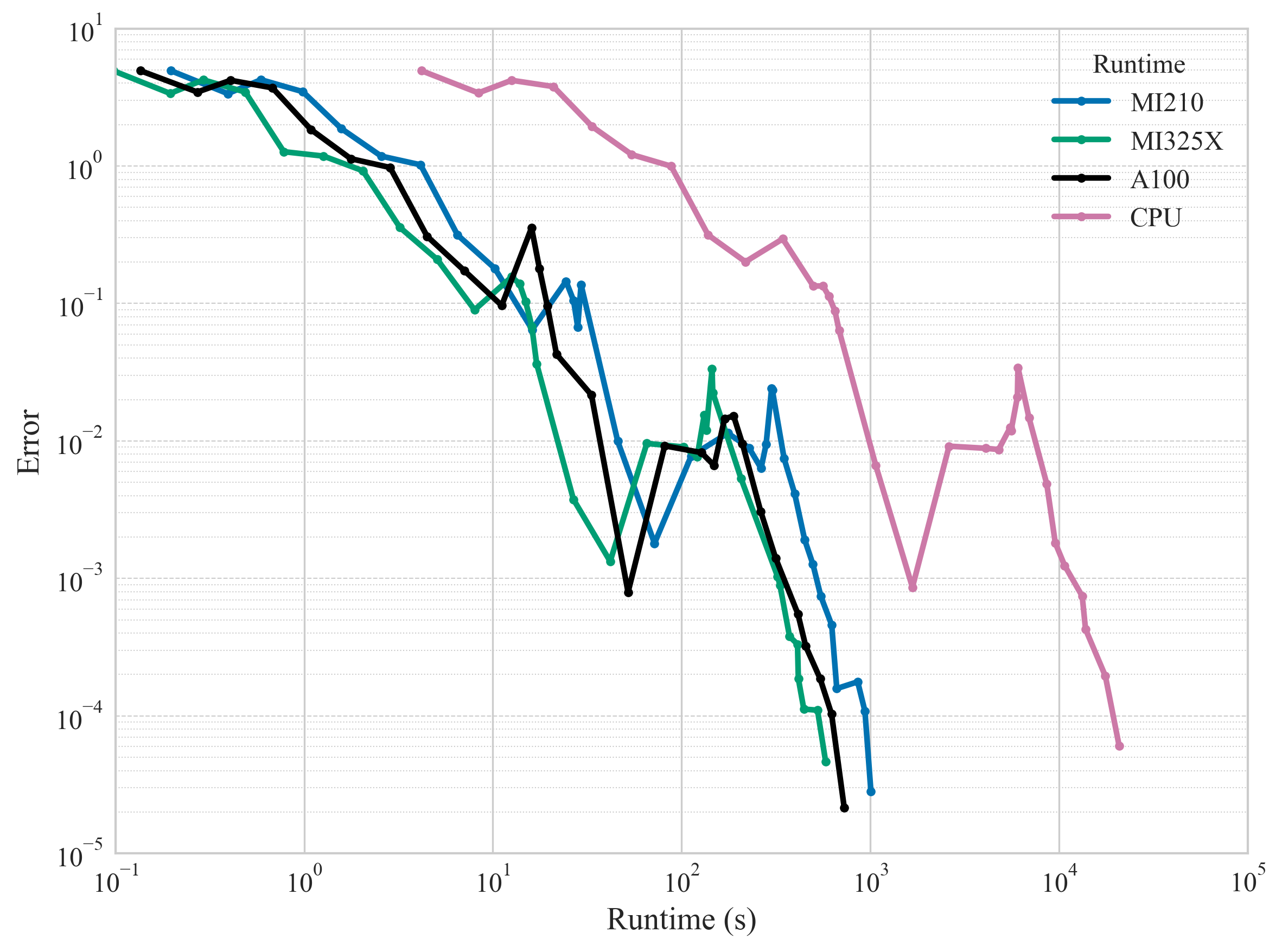}
\caption{Convergence on the Security-Constrained Economic Dispatch instance with error tolerance $10^{-4}$}
\label{fig:SCED}
\end{figure}

As predicted by the scaling trends observed in Figure \ref{fig:cpu_gpu_combined}, the GPUs provide a significant performance uplift over the CPU for a problem of this magnitude, roughly 36x for the MI325X in total runtime. Furthermore, Table \ref{tab:SCED} and Figure \ref{fig:SCED} visually confirms the superior performance of the MI325X relative to both the MI210 and the NVIDIA A100.

We also investigated the effect of `fishnet casting' on the SCED instance using the AMD MI325X. With this heuristic, torchPDLP solved the problem in two separate runs in 6 minutes and 14 seconds and 4 minutes and 37 seconds. The variability in solve time is due to the inherent stochasticity of the heuristic. Both results, however, represent a substantial improvement over the 9 minutes and 36 seconds required to solve the instance from a `cold start' (i.e., without the heuristic).

This preliminary finding suggests that larger, more complex problems may be particularly sensitive to warm-starting strategies. However, the variable performance of the fishnet casting adaptation indicates that further testing is required to validate its robustness and better understand its behavior across different problem structures.

\ifthenelse{\boolean{@twoside}}{\myclearpage}{}
\chapter{Conclusion}\label{chap:conclusion}

In this work, we introduced \texttt{torchPDLP}, an adaptation of the Primal-Dual Hybrid Gradient (PDHG) algorithm for solving large-scale linear programming (LP) problems. Our implementation, developed in the PyTorch framework, integrates both established techniques and novel practical enhancements, including pre-existing ideas in Section \ref{sec:pdlp}: diagonal preconditioning, adaptive restarting, primal weight updating, adaptive step size, and infeasibility detection and a novel enhancement: fishnet casting (Section \ref{sec:Fishnet}). Through extensive benchmarking on diverse datasets and hardware configurations, we evaluated the performance, scalability, and robustness of the solver.

This research makes several contributions to the field of optimization. First, it demonstrates the feasibility of employing modern machine learning frameworks, such as PyTorch, for classical optimization tasks, thereby lowering the barrier for cross-disciplinary research between optimization and machine learning communities. Second, our results reinforce the growing body of evidence that GPU-accelerated first-order methods can be competitive with, and in many cases surpass, traditional CPU-based solvers for large-scale LPs. Consistent with recent literature \cite{lu2023cupdlp,lu2025overview}, we observed that the performance advantage of our GPU implementation increases with problem size, underscoring its suitability for addressing the increasingly large and complex optimization challenges encountered in practice.

There are several promising directions for extending this research. One avenue is a comprehensive comparison with other GPU-based PDLP implementations, such as \texttt{cuOpt} and \texttt{cuPDLP.jl} \cite{lu2023cupdlp,lu2023cupdlp-c}. While \texttt{cuOpt} is primarily implemented in Cand CUDA, \texttt{torchPDLP} relies on PyTorch's kernel management; comparing runtime performance, memory usage, and algorithmic robustness could reveal opportunities for further optimization, including targeted kernel-level tuning.

Another direction involves refining fishnet parameter tuning and developing a more scalable primal-ball initialization strategy. As noted in Section~\ref{4.3.7}, the fishnet method exhibits varying performance across different \texttt{Netlib} problems, with no apparent correlation to problem size. Identifying effective initialization strategies—such as improved multistarting techniques—could enhance the robustness of fishnet casting.

Finally, incorporating recent advancements in first-order GPU-based LP solvers, particularly methods proposed in \cite{chen2025hprl,lu2024restartedhalpern}, may yield further performance gains. These enhancements could be directly integrated into \texttt{torchPDLP}, potentially enabling it to outperform current state-of-the-art solvers on a broader class of large-scale problems.




\ifthenelse{\boolean{@twoside}}{\myclearpage}{\newpage}
\addcontentsline{toc}{chapter}{References}  

\bibliographystyle{siam}     
\renewcommand\bibname{References}
\bibliography{references}

\begin{thebibliography}{10}

\bibitem{gridmatic}
{\em Gridmatic}.
\newblock Accessed: 2025-08-15.

\bibitem{Netlib_infeas}
{\em Netlib lp collection - infeasible problems}.
\newblock Accessed: 2025-08-15.

\bibitem{amd:epyc7003}
{\sc {Advanced Micro Devices}}, {\em {AMD EPYC™ 7003 Series Processors}},
  2022.

\bibitem{amd:mi210}
\leavevmode\vrule height 2pt depth -1.6pt width 23pt, {\em {AMD Instinct™
  MI210 Accelerator}}, 2022.

\bibitem{amd:mi325x}
\leavevmode\vrule height 2pt depth -1.6pt width 23pt, {\em {AMD Instinct™
  MI325X Accelerators}}, 2024.

\bibitem{applegate2021infeasibility}
{\sc D.~Applegate, M.~D{\'\i}az, H.~Lu, and M.~Lubin}, {\em Infeasibility
  detection with primal-dual hybrid gradient for large-scale linear
  programming}, arXiv preprint arXiv:2102.04592,  (2021).

\bibitem{applegate2}
{\sc D.~Applegate, M.~Díaz, O.~Hinder, H.~Lu, M.~Lubin, B.~O'Donoghue, and
  W.~Schudy}, {\em Practical large-scale linear programming using primal-dual
  hybrid gradient}, 2022.

\bibitem{Applegate_2022}
{\sc D.~Applegate, O.~Hinder, H.~Lu, and M.~Lubin}, {\em Faster first-order
  primal-dual methods for linear programming using restarts and sharpness},
  Mathematical Programming, 201 (2022), p.~133–184.

\bibitem{bartels1969simplex}
{\sc R.~H. Bartels and G.~H. Golub}, {\em The simplex method of linear
  programming using lu decomposition}, Communications of the ACM, 12 (1969),
  pp.~266--268.

\bibitem{boyd}
{\sc S.~Boyd and L.~Vandenberghe}, {\em Convex Optimization}, {Cambridge
  University Press}, March 2004.

\bibitem{chambolle2011first}
{\sc A.~Chambolle and T.~Pock}, {\em A first-order primal-dual algorithm for
  convex problems with applications to imaging}, Journal of mathematical
  imaging and vision, 40 (2011), pp.~120--145.

\bibitem{chen2025hprl}
{\sc K.~Chen, D.~Sun, Y.~Yuan, G.~Zhang, and X.~Zhao}, {\em Hpr-lp: An
  implementation of an hpr method for solving linear programming}, 2025.

\bibitem{dantzig1948programming}
{\sc G.~B. Dantzig}, {\em Programming in a linear structure}, in Bulletin of
  the American Mathematical Society, vol.~54, AMER MATHEMATICAL SOC 201 CHARLES
  ST, PROVIDENCE, RI 02940-2213, 1948, pp.~1074--1074.

\bibitem{dantzig2016linear}
\leavevmode\vrule height 2pt depth -1.6pt width 23pt, {\em Linear programming
  and extensions}, Princeton university press, 2016.

\bibitem{dikin1967iterative}
{\sc I.~I. Dikin}, {\em Iterative solution of problems of linear and quadratic
  programming}, in Soviet Math. Dokl., vol.~8, 1967, pp.~674--675.

\bibitem{fiacco1964computational}
{\sc A.~V. Fiacco and G.~P. McCormick}, {\em Computational algorithm for the
  sequential unconstrained minimization technique for nonlinear programming},
  Management science, 10 (1964), pp.~601--617.

\bibitem{forrest1992steepest}
{\sc J.~J. Forrest and D.~Goldfarb}, {\em Steepest-edge simplex algorithms for
  linear programming}, Mathematical programming, 57 (1992), pp.~341--374.

\bibitem{Netlib}
{\sc D.~M. Gay}, {\em netlib/fp}.

\bibitem{gill2021numerical}
{\sc P.~E. Gill, W.~Murray, and M.~H. Wright}, {\em Numerical Linear Algebra
  and Optimization}, Society for Industrial and Applied Mathematics,
  Philadelphia, PA, 2021.

\bibitem{papilo}
{\sc A.~Gleixner, L.~Gottwald, and A.~Hoen}, {\em {PaPILO}: A parallel
  presolving library for integer and linear programming with multiprecision
  support}, INFORMS Journal on Computing,  (2023).

\bibitem{miplib2017}
{\sc A.~Gleixner, G.~Hendel, G.~Gamrath, T.~Achterberg, M.~Bastubbe,
  T.~Berthold, P.~Christophel, K.~Jarck, T.~Koch, J.~Linderoth, et~al.}, {\em
  Miplib 2017: data-driven compilation of the 6th mixed-integer programming
  library}, Mathematical Programming Computation, 13 (2021), pp.~443--490.

\bibitem{golub2013matrix}
{\sc G.~H. Golub and C.~F. Van~Loan}, {\em Matrix Computations}, Johns Hopkins
  University Press, 4th~ed., 2013.

\bibitem{henk2012lowner}
{\sc M.~Henk}, {\em L{\"o}wner-john ellipsoids}, Documenta Math, 95 (2012),
  p.~4.

\bibitem{github}
{\sc X.~Hu, T.~Parker, C.~Phillips, and Y.~Yu}, {\em Pdlp-amd-rips}.
\newblock \url{https://github.com/SimplySnap/PDLP-AMD-RIPS}, 2025.
\newblock Accessed: 2025-08-15.

\bibitem{karmarkar1984new}
{\sc N.~Karmarkar}, {\em A new polynomial-time algorithm for linear
  programming}, in Proceedings of the sixteenth annual ACM symposium on Theory
  of computing, 1984, pp.~302--311.

\bibitem{lu2023cupdlp}
{\sc H.~Lu and J.~Yang}, {\em cupdlp.jl: A gpu implementation of restarted
  primal-dual hybrid gradient for linear programming in julia}, arXiv preprint
  arXiv:2311.12180,  (2023).

\bibitem{lu2024restartedhalpern}
{\sc H.~Lu and J.~Yang}, {\em Restarted halpern pdhg for linear programming},
  2024.

\bibitem{lu2025overview}
\leavevmode\vrule height 2pt depth -1.6pt width 23pt, {\em An overview of
  gpu-based first-order methods for linear programming and extensions}, 2025.

\bibitem{lu2023cupdlp-c}
{\sc H.~Lu, J.~Yang, H.~Hu, Q.~Huangfu, J.~Liu, T.~Liu, Y.~Ye, C.~Zhang, and
  D.~Ge}, {\em cupdlp-c: A strengthened implementation of cupdlp for linear
  programming by c language}, arXiv preprint arXiv:2312.14832,  (2023).

\bibitem{Manning-genetic}
{\sc T.~Manning, R.~D. Sleator, and P.~Walsh}, {\em Naturally selecting
  solutions}, Bioengineered, 4 (2013), pp.~266--278.
\newblock PMID: 23222169.

\bibitem{MARTI20131}
{\sc R.~Martí, M.~G. Resende, and C.~C. Ribeiro}, {\em Multi-start methods for
  combinatorial optimization}, European Journal of Operational Research, 226
  (2013), pp.~1--8.

\bibitem{nvidia:a100}
{\sc {NVIDIA}}, {\em {NVIDIA A100 Tensor Core GPU Datasheet}}, 2020.

\bibitem{ordonez2003computational}
{\sc F.~Ord{\'o}nez and R.~M. Freund}, {\em Computational experience and the
  explanatory value of condition measures for linear optimization}, SIAM
  Journal on Optimization, 14 (2003), pp.~307--333.

\bibitem{rockafellar1970convex}
{\sc R.~T. Rockafellar}, {\em Convex Analysis}, vol.~28 of Princeton
  Mathematical Series, Princeton University Press, Princeton, NJ, 1970.

\bibitem{ruiz2001scaling}
{\sc D.~Ruiz}, {\em A scaling algorithm to equilibrate both rows and columns
  norms in matrices}, tech. rep., CM-P00040415, 2001.

\bibitem{wright1997primal}
{\sc S.~J. Wright}, {\em Primal-dual interior-point methods}, SIAM, 1997.

\bibitem{doi:10.1137/S1052623496304712}
{\sc S.~J. Wright}, {\em Modified cholesky factorizations in interior-point
  algorithms for linear programming}, SIAM Journal on Optimization, 9 (1999),
  pp.~1159--1191.

\end{thebibliography}

\addcontentsline{toc}{chapter}{Appendix}
\appendix
\ifthenelse{\boolean{@twoside}}{\myclearpage}{}
\chapter{Convex Analysis Background}

In this appendix, we summarize several foundational concepts from convex analysis that are necessary to understand the derivation and implementation of our PDHG method for linear programming. While the primal-dual framework of Chambolle and Pock~\cite{chambolle2011first} provides the algorithmic foundation, it does not cover certain mathematical preliminaries in detail. To bridge this gap, we have drawn on \emph{Convex Analysis} by R.~Tyrrell Rockafellar~\cite{rockafellar1970convex}, extracting and rephrasing in our own words only those results essential for our purposes. These include definitions, basic properties, and optimality conditions that are directly relevant to the saddle-point formulation of LPs and to the proximal operators used in PDHG.
\section{Explaining the Saddle-Point LP}
In this section, we present the theory of convex saddle-point problems and review the definitions and properties needed for our algorithm. We then explain how linear programs can be expressed in saddle-point form.
\subsection{Convex Saddle-Point Problems}
We first introduce the general convex-concave saddle-point problem that our algorithm is designed to solve. Let \( X \) and \( Y \) be finite-dimensional real vector spaces equipped with inner products \(\langle \cdot, \cdot \rangle\) and induced norms \(\| \cdot \| = \langle \cdot, \cdot \rangle ^ \frac{1}{2} \). Let \( K : X \to Y \) be a continuous linear operator, whose operator norm is defined by
\begin{equation}
    \|K\| = \max \{ \|Kx\| : x \in X, \|x\| \le 1 \}.
\end{equation}

The generic form of the problem is given by
\begin{equation}
    \min_{x \in X} \max_{y \in Y} \; \langle Kx, y \rangle + G(x) - F^*(y),
    \label{eq:saddle-point}
\end{equation}
where the functions \( G : X \to [0, +\infty] \) and \( F^* : Y \to [0, +\infty] \) are assumed to be \emph{proper}, \emph{convex}, and \emph{lower semicontinuous (l.s.c.)}. 

\begin{itemize}
    \item \emph{Domain} of a function \( f : V \to (-\infty, +\infty] \) is \( dom(f) = \{v \in V | f(v) \neq + \infty\}\)
    
    \item A function \(f\) is \emph{proper} if \(dom(f) \neq \emptyset \).
    
    \item A set \(S\) is \emph{convex} if for all \(s_1,s_2 \in S\) and \( \theta \in [0,1] \), 
    \[
    [ \theta s_1 + (1-\theta) s_2 ] \in S.
    \]
    
    \item \( f \) is \emph{convex} if \(dom(f)\) is convex, and for all \( v_1, v_2 \) in \(dom(f)\), \( \theta \in [0,1] \),
    \[
    f(\theta v_1 + (1 - \theta) v_2) \le \theta f(v_1) + (1 - \theta) f(v_2).
    \]
    \item \( f \) is \emph{lower semicontinuous} if for any \(v_0 \in V\),
    \[
    \liminf_{v \to v_0} f(v) \ge f(v_0).
    \]
\end{itemize}

The function \( F^* \) here denotes the \emph{convex conjugate} of a \emph{convex l.s.c. function} \( F : Y \to (-\infty, +\infty] \), defined by
\[
F^*(y) = \sup_{v \in Y} \left\{ \langle y, v \rangle - F(v) \right\}.
\]

Example: \(f : \mathbb{R} \to (-\infty, +\infty] \), \(f(y) = \frac{y^2}{2}\)
\[
\begin{aligned}
f^*(y) &= \sup_{z \in \mathbb{R}} \left\{ \langle y, z \rangle - f(z) \right\} \\
       &= \sup_{z \in \mathbb{R}} \left\{ yz - \frac{z^2}{2} \right\} \\
       &= \sup_{z \in \mathbb{R}} \left\{z(y - \frac{z}{2}) \right\} \quad \text{(by taking derivative w.r.t. } z \text{)} \\
       &= y(y - \frac{y}{2}) \\
       &= \frac{y^2}{2}
\end{aligned}
\]

This saddle-point problem can be seen as the primal-dual formulation of the following primal problem:
\begin{equation}
    \min_{x \in X} \; F(Kx) + G(x),
\end{equation}
whose dual takes the form
\begin{equation}
    \max_{y \in Y} \; -G^*(-K^* y) - F^*(y),
\end{equation}
where \( K^* : Y \to X \) is the adjoint of \( K \). (i.e. for any \((x,y) \in X \oplus Y \), we have \(\langle Kx, y \rangle_Y = \langle x, K^* y \rangle_X\)). When \(X\) and \(Y\) are both real vector space, \(K^* = K^\top \).

For a \emph{convex} funciton \(f : \mathbb{R}^n \to (-\infty, +\infty]\), and \(x \in dom(f)\), the \emph{subgradient} of \(f\) at \(x\) is a set: \(\partial f(x) = \left\{ w \in \mathbb{R}^n|f(z) \ge f(x) + \langle w, z-x \rangle, \forall z \in \mathbb{R}^n \right\}\).
When \(f\) is differentiable at \(x\), we have \(\partial f(x) = \left\{ \bigtriangledown f(x) \right\}\)

Example: \(f : \mathbb{R} \to (-\infty, +\infty] \), \(f(y) = |y|\)
\[
\partial f(y) = \begin{cases}
1, & \text{if } x > 0\\
[-1,1], & \text{if } x = 0\\
-1, & \text{if } x<0
\end{cases}
\]
\(\partial f(0) = [-1,1]\),since \(\forall w \in \partial f(0), z \in \mathbb{R}\):
\[
\begin{aligned}
\\(|z| \ge|x|+wz-wx \ \& \  x=0 &\implies |z| \ge wz \\
                  &\implies \begin{cases}
\frac{|z|}{z} \ge w, & \text{if } z > 0\\
\frac{|z|}{z} \le w, & \text{if } z < 0
\end{cases}\\
&\implies 1 \ge w \ \& \ -1 \le w \\
&\implies w \in [-1,1]
\end{aligned}
\]

Given a saddle point problem:
\[
\min_{x \in X} \max_{y \in Y} \; \langle Kx, y \rangle + G(x) - F^*(y),
\]
The Lagrangian function \(\Phi(x,y) := \langle Kx, y \rangle + G(x) - F^*(y)\), then saddle point \(( \hat{x},\hat{y} )\) satisfies the following subgradient conditions:

\begin{equation}
    \begin{cases}
K \hat{x} \in \partial F^*(\hat{y})\\
-(K^* \hat{y}) \in \partial G(\hat{x})
\end{cases}
\end{equation}
Since the saddle point \(( \hat{x},\hat{y} )\) satisfies:
\[
\Phi(\hat{x},y) \le \Phi(\hat{x},\hat{y}) \le \Phi(x,\hat{y})
\]

For the right hand inequality, fixed \( \hat{y}\), 
\[
\hat{x} \textbf{ minimizes}: \Phi(x,\hat{y}) = \langle Kx, \hat{y} \rangle + G(x) - F^*(\hat{y}),\\
\]
\[
\begin{aligned}
\hat{x} &= {\arg \min}_{x \in X}{\{ \langle Kx, \hat{y} \rangle + G(x) - F^*(\hat{y})} \}\\
&= {\arg \min}_{x \in X}{\{ \langle x, k^* \hat{y} \rangle + G(x)} \}, by \  \langle Kx, y \rangle_Y = \langle x, K^* y \rangle_X
\end{aligned}
\]

For the left hand inequality, fixed \( \hat{x}\), 
\[
\hat{y} \textbf{ maximizes}: \Phi(\hat{x},y) = \langle K\hat{x}, y \rangle + G(\hat{x}) - F^*(y),\\
\]

\[
\begin{aligned}
\hat{y} &= {\arg \max}_{y \in Y}{\{ \langle K\hat{x}, y \rangle + G(\hat{x}) - F^*(y)} \}\\
&= {\arg \max}_{y \in Y}{\langle K\hat{x}, y \rangle - F^*(y)} \}\\
&= {\arg \min}_{y \in Y}{\langle (-K)\hat{x}, y \rangle + F^*(y)} \}
\end{aligned}
\]

For a convex function \( H : \mathbb{R}^n \to \mathbb{R} \cup \{\infty\},\hat{x}\) is global minimizer if and only if \(0 \in \partial H (\hat{x})\).

If \(\hat{x}\) is minimizer, then \(\forall z \in \mathbb{R}^n, H(z) \ge H(\hat{x}), \implies H(z) \ge H(\hat{x}) + \langle 0, z-\hat{x}\rangle \implies 0 \in \partial H(\hat{x})\).

If \(0 \in \partial H(\hat{x})\), then \(H(z) \ge H(\hat{x}) + \langle0,z-\hat{x}\rangle = H(\hat{x}) \ ,\forall z \in\mathbb{R} \implies\hat{x}\) is a minimizer.

For a function: \(H(x) = \langle x, k^* \hat{y} \rangle + G(x)\), \(H(x)\) reaches its minimum at \(\hat{x} \iff 0 \in \partial H(\hat{x})\). 
\[
\begin{aligned}
\partial H(x) &= \partial (\langle x, K^* \hat{y} \rangle ) + \partial G(x) \\
&= \partial (K^* \hat{y} \ x^\top  ) + \partial G(x) \\
&= K^* \hat{y} + \partial G(x)
\end{aligned}
\]
Hence \( 0 \in \partial H(\hat{x}) \iff -K^*\hat{y} \in \partial G(\hat{x})\).

For a function: \(J(y) = \langle (-K)\hat{x}, y \rangle + F^*(y), J(y)\) reaches its minimum at \(\hat{y} \iff 0 \in \partial J(\hat{y})\) 
\[
\begin{aligned}
\partial J(x) &= \partial (\langle (-K)\hat{x}, y \rangle ) + \partial F^*(y) \\
&= \partial ((-K)\hat{x} \ y^\top  ) + \partial F^*(y) \\
&= (-K)\hat{x} + \partial F^*(y)
\end{aligned}
\]
Hence \( 0 \in \partial J(\hat{y})) \iff K\hat{x} \in \partial F^*(\hat{y})\).

We generally assume that the problem admits at least one primal-dual solution \((\hat{x}, \hat{y}) \in X \oplus Y\), satisfying the optimality conditions
\[
K \hat{x} \in \partial F^*(\hat{y}), \quad -K^* \hat{y} \in \partial G(\hat{x})
\]

The definition of Proximal Operator (resolvent operator):
\[
Prox_{\tau F}(y) := \arg\min_{x}\left\{ \frac{1}{2\tau}\|x - y\|^2 + F(x) \right\}
\]
Claim: \(Prox_{\tau F}(y) = (I + \tau \partial F)^{-1}(y)\).

\begin{proof}: Let \(x^* = \arg\min_{x}\left\{ \frac{1}{2\tau}\|x - y\|^2 + F(x) \right\}\), then \begin{align*}
0 &\in \partial \left(\frac{1}{2\tau}\|x - y\|^2 + F(x)\right)_{x=x^*} \\
  &\implies 0 \in \left(\frac{x - y}{\tau} + \partial F(x)\right)\big|_{x = x^*} \\
  &\implies 0 \in x^* - y + \tau \partial F(x^*) \\
  &\implies y \in x^* + \tau \partial F(x^*) = (I + \tau \partial F)(x^*) \\
  &\implies (I + \tau \partial F)^{-1}(y) = x^* \\
  &\implies Prox_{\tau F}(y) := \arg\min_{x}\left\{ \frac{1}{2\tau}\|x - y\|^2 + F(x) \right\} = (I + \tau \partial F)^{-1}(y) \ \ \ 
\end{align*}\end{proof}

Moreau's identity (\(\tau > 0\)): 
\[
\begin{aligned}
x = (I + \tau \partial F)^{-1}(x) + \tau(I + \frac{1}{\tau} \partial F^*)^{-1}(\frac{x}{\tau}) \\
\iff x = Prox_{\tau F}(x)+ \tau Prox_{\frac{1}{\tau}F^*}(\frac{x}{\tau}) \ \ \ \ \ \ \ \ \ \ \ \ \ 
\end{aligned}
\]
\begin{proof}

Let \(p = Prox_{\tau F}(x) = \arg\min_{z}\left\{ \frac{1}{2\tau}\|z - x\|^2 + F(z) \right\} \implies 0 \in \frac{p-x}{\tau} + \partial F(p) \implies x-p \in \tau \partial F(p)\)

\(F^*(y) =  \sup_{z} \left\{ \langle y, z \rangle - F(z) \right\}, \textbf{using the property of subgradient:} y \in \partial F(z) \iff z \in F^*(y), \implies \left\{x-p \in \tau \partial F(p) \iff p \in \partial F^*(\frac{x-p}{\tau})\right\}\)

Let \(q = Prox_{\frac{1}{\tau}F^*}(\frac{x}{\tau}) = \arg\min_{y}\left\{ \frac{\tau}{2}\|y - \frac{x}{\tau}\|^2 + F^*(y) \right\} \implies 0 \in \tau (q-\frac{x}{\tau}) + F^*(q)) \implies x-\tau q \in \partial F^*(q)\)

Since: 
\[
\begin{cases}
p \in \partial F^*\left(\frac{x-p}{\tau}\right)\\
x-\tau q\in \partial F^*(q)
\end{cases}
\]

Let \(v = \frac{x-p}{\tau}\). Then we have:
\[
p \in \partial F^*(v) \quad \text{and} \quad p = x - \tau v
\]
which implies:
\[
x - \tau v \in \partial F^*(v)
\]

Notice that both \(v\) and \(q\) satisfy the same inclusion:
\[
x - \tau \cdot \in \partial F^*(\cdot)
\]

To prove uniqueness, suppose there are two solutions $u_1$ and $u_2$ satisfying:
\[
x - \tau u_1 \in \partial F^*(u_1) \quad \text{and} \quad x - \tau u_2 \in \partial F^*(u_2)
\]
By definition of subgradient:
\[
F^*(z) \geq F^*(u_1) + \langle x - \tau u_1, z - u_1 \rangle \quad \forall z
\]
\[
F^*(z) \geq F^*(u_2) + \langle x - \tau u_2, z - u_2 \rangle \quad \forall z
\]
Set $z = u_2$ in the first and $z = u_1$ in the second:
\[
F^*(u_2) \geq F^*(u_1) + \langle x - \tau u_1, u_2 - u_1 \rangle
\]
\[
F^*(u_1) \geq F^*(u_2) + \langle x - \tau u_2, u_1 - u_2 \rangle
\]
Adding these inequalities:
\[
F^*(u_2) + F^*(u_1) \geq F^*(u_1) + F^*(u_2) + \langle x - \tau u_1, u_2 - u_1 \rangle + \langle x - \tau u_2, u_1 - u_2 \rangle
\]
Simplify the rightmost terms:
\[
\begin{aligned}
&\langle x - \tau u_1, u_2 - u_1 \rangle + \langle x - \tau u_2, u_1 - u_2 \rangle \\
&= \langle x - \tau u_1, u_2 - u_1 \rangle - \langle x - \tau u_2, u_2 - u_1 \rangle \\
&= \langle (x - \tau u_1) - (x - \tau u_2), u_2 - u_1 \rangle \\
&= \langle -\tau(u_1 - u_2), u_2 - u_1 \rangle \\
&= -\tau \langle u_1 - u_2, u_2 - u_1 \rangle \\
&= -\tau \langle u_1 - u_2, -(u_1 - u_2) \rangle \\
&= \tau \| u_1 - u_2 \|^2
\end{aligned}
\]
Thus:
\[
0 \geq \tau \| u_1 - u_2 \|^2
\]
Since $\tau > 0$ and $\| u_1 - u_2 \|^2 \geq 0$, we must have $\| u_1 - u_2 \|^2 = 0$, so $u_1 = u_2$. Therefore the solution is unique.

This inclusion has a unique solution. Therefore:
\[
v = q
\]

Thus \(p = x - \tau v = x - \tau q\) and:
\[
\begin{aligned} 
x &= p + \tau q\\
&= Prox_{\tau F}(x) + \tau Prox_{\frac{1}{\tau}F^*}\left(\frac{x}{\tau}\right) \\
&= (I + \tau \partial F)^{-1}(x) + \tau\left(I + \frac{1}{\tau} \partial F^*\right)^{-1}\left(\frac{x}{\tau}\right) \quad 
\end{aligned}
\]
\end{proof}
Throughout the paper we will assume that \(F^*\) and \(G\)  are "\emph{simple}", in the sense that their resolvent operator defined through:
\[
x = (I + \tau\partial F)^{-1}(y) = {\arg \min}_{x \in X}{ \left\{ \frac{||x-y||^2}{2\tau} + F(x) \right\}}
 \]
has a closed form representation (or can be efficiently solved up to a high precision, e.g. Using a Newton method in low dimension). 

\bigskip
\noindent\textbf{Remark.}  
The definitions and propositions presented here are adapted from Rockafellar's \emph{Convex Analysis}~\cite{rockafellar1970convex}, specifically those parts relevant to understanding the saddle-point formulation and the optimality conditions employed in our algorithm. All statements are rephrased in our own words, with notation aligned to the rest of the paper.

\ifthenelse{\boolean{@twoside}}{\myclearpage}{}
\chapter{Dual Problem Derivation via Lagrangian Duality}\label{App:DualDerivation}

\noindent\textbf{Remark on notation:} 
The matrices \(K\) and \(K'\) used in this appendix differ from the \(K\) defined in the main text. 
Specifically, in this appendix and in later appendices, we set
\[
K = \begin{pmatrix} -G \\ -A \end{pmatrix},  
\quad 
K' = \begin{pmatrix} G \\ A \end{pmatrix}.
\]
This choice is made to simplify the derivations and to eliminate certain ambiguities in our later proofs and embeddings.

\bigskip
\noindent\textbf{Remark on source:}  
The notation and some definitions in this appendix follow those in \emph{Numerical Linear Algebra and Optimization} by Gill, Murray, and Wright~\cite{gill2021numerical}, which provides the standard form of the dual linear program.  
Here, we re-derive the dual problem for our specific LP formulation, adapting their approach to match the problem structure and sign conventions used in our work.

\section{General Formulation}

Any \textbf{Linear Programming} (LP) problem can be expressed in the form:
\[
\begin{aligned}
    \min_{x \in \mathbb{R}^n} \quad & c^\top x \\
    \text{subject to:} \quad & Gx \ge h, \\
    & Ax = b, \\
    & l \le x \le u,
\end{aligned}
\]
where 
\[
\begin{aligned}
&x \in \mathbb{R}^n, \quad 
G \in \mathbb{R}^{m_1 \times n} \ \text{(inequality constraints)}, \quad
A \in \mathbb{R}^{m_2 \times n} \ \text{(equality constraints)}, \\
&c \in \mathbb{R}^n \ \text{(cost vector)}, \quad
h \in \mathbb{R}^{m_1}, \quad b \in \mathbb{R}^{m_2}, \\
&l \in \bigoplus_{k=1}^{n} \left\{ \mathbb{R} \cup \{-\infty\} \right\}, \quad
u \in \bigoplus_{k=1}^{n} \left\{ \mathbb{R} \cup \{\infty\} \right\}.
\end{aligned}
\]
We allow $l_i = -\infty$ and $u_j = \infty$ to indicate an unbounded variable.  
This formulation is referred to as the \textbf{primal problem}.

\section{Lagrangian Construction}
Let $y_{\mathrm{ineq}} \in \bigoplus_{k=1}^{m_1} \mathbb{R}^+$ be the multipliers for $Gx \ge h$, and  
$y_{\mathrm{eq}} \in \bigoplus_{k=1}^{m_2} \mathbb{R}$ be the multipliers for $Ax = b$.  

Combine them into 
\[
y = \begin{pmatrix} y_{\mathrm{ineq}} \\ y_{\mathrm{eq}} \end{pmatrix} \in \bigoplus_{k=1}^{m_1+m_2} \mathbb{R}_k, 
\quad \mathbb{R}_k =
\begin{cases}
\mathbb{R}^+, & k \le m_1, \\
\mathbb{R}, & k > m_1.
\end{cases}
\]

Define
\[
K' = \begin{pmatrix} G \\ A \end{pmatrix}, 
\quad q = \begin{pmatrix} h \\ b \end{pmatrix},
\]
so that $K' x = \begin{pmatrix} Gx \\ Ax \end{pmatrix}$ and $q^\top y = h^\top y_{\mathrm{ineq}} + b^\top y_{\mathrm{eq}}$.

The \textbf{Lagrangian function} is:
\[
\begin{aligned}
\mathcal{L}(x, y) 
&= c^\top x - y_{\mathrm{ineq}}^\top (Gx - h) - y_{\mathrm{eq}}^\top (Ax - b) \\
&= c^\top x - y^\top (K' x - q) \\
&= c^\top x - y^\top K' x + q^\top y.
\end{aligned}
\]

\section{Dual Function Derivation}
The dual function is obtained by minimizing $\mathcal{L}(x, y)$ over $x$ subject to $l \le x \le u$:
\[
g(y) = \min_{l \le x' \le u} \mathcal{L}(x', y)
      = q^\top y + \min_{l \le x' \le u} \left\{ (c - K'^\top y)^\top x' \right\}.
\]
Let $\lambda = c - K'^\top y \in \mathbb{R}^n$. Since the constraints are separable,
\[
\min_{l \le x' \le u} \lambda^\top x' 
= \sum_{i=1}^n \min_{l_i \le x'_i \le u_i} \lambda_i x'_i.
\]

Each term satisfies:
\[
\min_{l_i \le x'_i \le u_i} \lambda_i x'_i =
\begin{cases}
\lambda_i l_i, & \lambda_i > 0, \\
\lambda_i u_i, & \lambda_i < 0, \\
0, & \lambda_i = 0.
\end{cases}
\]
Defining $\lambda^+_i = \max\{0, \lambda_i\}$ and $\lambda^-_i = \min\{0, \lambda_i\}$, this becomes
\[
\min_{l_i \le x'_i \le u_i} \lambda_i x'_i = \lambda_i^+ l_i + \lambda_i^- u_i.
\]
Summing over $i$:
\[
\min_{l \le x' \le u} \lambda^\top x' = l^\top \lambda^+ + u^\top \lambda^-.
\]
Thus,
\[
g(y) = q^\top y + l^\top \lambda^+ + u^\top \lambda^-, \quad \lambda = c - K'^\top y.
\]

\section{Feasibility Conditions for $\lambda$}
For $g(y)$ to be finite, $\lambda$ must satisfy:
\[
\Lambda_i =
\begin{cases}
\{0\}, & l_i = -\infty, \ u_i = \infty, \\
\mathbb{R}^-, & l_i = -\infty, \ u_i \in \mathbb{R}, \\
\mathbb{R}^+, & l_i \in \mathbb{R}, \ u_i = \infty, \\
\mathbb{R}, & \text{otherwise}.
\end{cases}
\]
Let $\Lambda = \bigoplus_{i=1}^n \Lambda_i$.

\section{Dual Problem}
The dual problem is:
\[
\begin{aligned}
\max_{y \in \mathbb{R}^{m_1+m_2},\ \lambda \in \mathbb{R}^n} \quad 
& q^\top y + l^\top \lambda^+ + u^\top \lambda^- \\
\text{subject to:} \quad 
& \lambda = c - K'^\top y, \\
& y_{1:m_1} \ge 0, \\
& \lambda \in \Lambda.
\end{aligned}
\]

\ifthenelse{\boolean{@twoside}}{\myclearpage}{}
\def\proj{\textnormal{proj}}

\chapter{Embedding Linear Programming into the General Saddle-point Framework}\label{appendixC}

\noindent\textbf{Acknowledgement:} 
We thank the work of Chambolle and Pock~\cite{chambolle2011first} for providing the Primal-Dual Hybrid Gradient (PDHG) method, which forms the basis of our algorithmic framework. 
We also acknowledge Applegate \emph{et al.}~\cite{applegate2} for their formulation of the PDHG for Linear Programming (PDLP) approach, which inspired the structure used here.

In this section, we demonstrate how the Linear Programming problem can be embedded into the general saddle-point framework. This connection is crucial for applying the Primal-Dual Hybrid Gradient (PDHG) algorithm to Linear Programming problems.

Recall the primal Linear Programming Formulation:
\[
\begin{aligned}
    \min_{x \in \mathbb{R}^n} \; &c^\top x\\
    \text{subject to: } &Gx\ge h \\
    &Ax = b\\
    &l\le x \le u
\end{aligned}
\]
where \(x \in \mathbb{R}^n,G \in \mathbb{R}^{m_1 \times n}\), \(A \in \mathbb{R}^{m_2 \times n}\), \(c \in \mathbb{R}^n\), \(h \in \mathbb{R}^{m_1}, b \in \mathbb{R}^{m_2}, l \in \displaystyle\bigoplus_{k=1}^{n}\left\{\mathbb{R} \cup \{-\infty\} \right\}, u \in \displaystyle\bigoplus_{k=1}^{n}\left\{\mathbb{R} \cup \{\infty\} \right\}  \).

and the general saddle point problem:
\[
\min_{x \in X} \max_{y \in Y} \; \langle Kx, y \rangle + G(x) - F^*(y),
\]
where \(X,Y\) are real vector spaces, \( K : X \to Y \) is a linear operator, the functions \( G : X \to [0, +\infty] \) and \( F^* : Y \to [0, +\infty] \) are assumed to be \emph{proper}, \emph{convex}, and \emph{lower semicontinuous (l.s.c.)}. 

\textbf{Primal Space:} \(X = \mathbb{R}^n\)

\textbf{Dual Space:} \(Y = \mathbb{R}^{m_1+m_2}\)

\textbf{Linear Operator:} Define \( K : X \to Y \) as:
\[
Kx = \begin{pmatrix}
    -Gx\\-Ax
\end{pmatrix}
\]

The adjoint operator \( K^* : Y \to X \) is given by \(K^*y=K^\top y=-G^\top y_{ineq}-A^\top y_{eq},\textit{ for } y =(y_{ineq},y_{eq})\in \mathbb{R}^{m_1+m_2} \)

\textbf{Define G(x),} G(x) captures the primal objective and box constraints:
\[
G(x) = c^\top x + \delta_{l \le x \le u}(x)
\]
where \(\delta_A(x) \) is the \text{indicator function}, 
\[
\delta_A(x) = \begin{cases}
    0, \text{if } x\in A \\
    \infty, \text{otherwise}
    
\end{cases}
\]

\textbf{Define} \( \mathbf{F^*(y)}, F^*(y)\) handle dual constraints and constant terms from primal constraints:
\[
F^*(y) = -q^\top y + \delta_{y_{ineq \ge 0}}(y)
\]
where \(q=\begin{pmatrix}
    h\\b
\end{pmatrix}\) (combine primal constraint constants).

Substituting \(G(x)\) and \(F^*(y)\) into the general framework:
\[
\begin{aligned}
    &\min_{x \in X} \max_{y \in Y} \; \langle Kx, y \rangle + G(x) - F^*(y)\\
    &= \min_{x \in X} \max_{y \in Y} \; \left\langle \begin{pmatrix}
        -Gx\\
        -Ax
    \end{pmatrix}, \begin{pmatrix}
        y_{ineq}\\
        y_{eq}
    \end{pmatrix} \right\rangle + c^\top x + \delta_{l\le x\le u}(x) +q^\top y - \delta_{y_{ineq} \ge0}(y)\\
    &= \min_{x \in X} \max_{y \in Y} \; -y^\top  \begin{pmatrix}
        Gx\\
        Ax
    \end{pmatrix} + c^\top x + \delta_{l\le x\le u}(x) +q^\top y - \delta_{y_{ineq} \ge0}(y)\\
    &= \min_{x \in X} \max_{y \in Y} \; -y^\top  K'x + c^\top x + \delta_{l\le x\le u}(x) +q^\top y - \delta_{y_{ineq} \ge0}(y)\\
    &= \min_{x \in X} \max_{y \in Y} \; -y^\top  K'x + c^\top x + \delta_{l\le x\le u}(x) +y^\top q - \delta_{y_{ineq} \ge0}(y)\\
    &= \min_{x \in X} \max_{y \in Y} \; c^\top x - y^\top (K'x-q) + \delta_{l\le x\le u}(x)- \delta_{y_{ineq} \ge0}(y)\\
\end{aligned}
\]
where \(K'x=\begin{pmatrix}
    Gx\\
    Ax
\end{pmatrix}\), this formulation implicitly enforces: \textbf{Primal feasibility:} \(Gx \ge h, Ax = b, l\le x \le u\), and \textbf{Dual feasibility:} \(y_{ineq} \ge 0\).

For practical use, we introduce the partial primal-dual gap:
\[
\begin{aligned}
    \mathcal{G}_{B_1 \oplus B_2}(x,y) = &\max_{y' \in B_2}\left\{ \langle y',Kx \rangle - F^*(y')+G(x)\right\} - \min_{x' \in B_1}\left\{ \langle y,Kx' \rangle - F^*(y)+G(x')\right\} \\
    =&\max_{y' \in B_2}\left\{ -\langle y',K'x \rangle +q^\top y' - \delta_{y_{ineq} \ge0}(y
    ) + c^\top x + \delta_{l\le x\le u}(x)\right\} \\
    &- \min_{x' \in B_1}\left\{ -\langle y,K'x' \rangle +q^\top y - \delta_{y_{ineq} \ge0}(y)+ c^\top x' + \delta_{l\le x\le u}(x')\right\} 
\end{aligned}
\]

\noindent \textbf{Algorithm 1}
\begin{itemize}
    \item \textbf{Initialization:} Choose \(\tau, \sigma > 0, \theta \in[0,1], (x^0,y^0) \in X \oplus Y,\) and set \(\bar{x}^0 = x^0\).
    \item \textbf{Iterations (\(n \ge 0\))} Update \(x^{n+1},y^{n+1},\bar{x}^{n+1}\) as follows:
\end{itemize}
\[
\left\{
\begin{aligned}
    x^{n+1} &= (I + \tau \partial G)^{-1} (x^n - \tau K^* y^n)\\
    \bar{x}^{n+1} &= x^{n+1}+\theta(x^{n+1}-x^n)\\
    y^{n+1} &= (I + \sigma \partial F^*)^{-1}(y^n+\sigma K \bar{x}^{n+1})
\end{aligned}
\right\}
\]

\noindent \textbf{Algorithm 1 for LP}
From the section 6, we define \(G(x) = c^\top x + \delta_{l \le x \le u}(x)\), and \(F^*(y) = -q^\top y + \delta_{y_{ineq \ge 0}}(y)\). Substituting \(G(x)\) and \(F^*(y)\) into Algorithm:
\[
\begin{aligned}
    (I + \tau \partial G)^{-1}(v) &=  \arg\min_{x}\left\{ \frac{1}{2\tau}\|x - v\|^2 + G(x) \right\}\\
    &= \arg\min_{x}\left\{ \frac{1}{2\tau}\|x - v\|^2 + c^\top x + \delta_{l \le x \le u}(x) \right\}\\
    &= \arg\min_{l\le x \le u}\left\{ \frac{1}{2\tau}\|x - v\|^2 + c^\top x \right\}, \textit{since  x $<$ l or x $>$ u, we have }\delta_{l \le x \le u}(x) = \infty\\
\end{aligned}
\]
Taking the derivative of the above objective function (\( \frac{1}{2\tau}\|x - v\|^2 + c^\top x\)), ignoring the box constraints (\(l \le x \le u\)) for now, and setting it equal to zero yields the first order optimality condition:
\[
\nabla_x(\frac{1}{2\tau}\|x - v\|^2 + c^\top x) = \frac{x-v}{\tau} + c = 0
\implies x = v-c\tau
\]
Thus, we obtain the \textbf{unconstrained} minimizer: \(x^{\#} = v-c\tau\).

If \(l \le x^{\#} \le u\), then it's optimal solution. Otherwise, if some components of \(x^{\#}\) fall outside the interval, the optimal solution will lie on the boundary. Intuitively, this minimization problem asks for a point \(x\) with in the box constraints \(\displaystyle\bigoplus_{k=1}^{n}\left\{[l_k,u_k]\right\}\) that is as close as possible to \(v-c\tau\), since the linear term \(c^\top x\) merely shifts the unconstrained minimizer from \(v\) to \(v - c\tau\). Therefore, in the constrained case, the solution is obtained by clipping \(x^{\#} = v-c\tau\) to the interval \(\displaystyle\bigoplus_{k=1}^{n}\left\{[l_k,u_k]\right\}\) component-wise. That is projecting each coordinates of \(x^{\#}\) onto the interval\([u_k,l_k]\). In other words, 
\[
x^* = \proj_{\mathcal{X}}(v-c\tau), \text{where } \mathcal{X} = \displaystyle\bigoplus_{k=1}^{n}\left\{[l_k,u_k]\right\}
\]
where \(\proj_{\mathcal{X}}\) denotes the component-wise projection operator that restricts each coordinate \(k\) to lie within the interval \([l_k,u_k]\). This is a well-known result: the proximal operator of a indicator function is the projection operator.

Hence, 
\[
(I + \tau \partial G)^{-1}(v) = \proj_{\mathcal{X}}(v-c\tau)
\]

Let \(v = x^n-\tau K^*y^n\). Use the result above, we obtain:
\[
\begin{aligned}
    x^{n+1} &= (I + \tau \partial G)^{-1}(x^n-\tau K^*y^n)\\
    &= \proj_{\mathcal{X}}(x^n-\tau K^*y^n-c\tau) \\
    &= \proj_{\mathcal{X}}(x^n-\tau(c+ K^*y^n))
\end{aligned}
\]

Similarly, we analyze the proximal operator for the dual step, defined by:
\[
\begin{aligned}
    (I + \sigma \partial F^*)^{-1}(w) &=  \arg\min_{y}\left\{ \frac{1}{2\sigma}\|y - w\|^2 + F^*(x) \right\}
\end{aligned}
\]

From the earlier definitions, we know:
\[
F^*(y) = -q\top y + \delta_{y_{ineq}\ge0}(y)
\]

Substituting this expression into the proximal operator yields:
\[
\begin{aligned}
    (I + \sigma \partial F^*)^{-1}(w) &=  \arg\min_{y}\left\{ \frac{1}{2\sigma}\|y - w\|^2 -q\top y + \delta_{y_{ineq}\ge0}(y) \right\} \\
    &= \arg\min_{y_{ineq}\ge0}\left\{ \frac{1}{2\sigma}\|y - w\|^2 -q\top y \right\}, \textit{since  y$_{ineq}$ $<$ 0, we have }\delta_{y_{ineq}\ge0}(y) = \infty\\ 
\end{aligned}
\]

This objective function is separable across the components of \(y\), so we can solve each block independently. For the equality constrained \(y_{eq}\), there is no constraint, so we take the gradient and set it to zero:
\[
\nabla_{y_{eq}} (\frac{1}{2\sigma}\|y - w\|^2 -q\top y ) = \frac{y_{eq}-w_{eq}}{\sigma} -q_{eq} = 0
\]

Solve gives:
\[
y_{eq}^* = w_{eq} + \sigma q_{eq}
\]

For the inequality constrained \(y_{ineq}\), we minimize:

\[
\arg\min_{y}\left\{ \frac{1}{2\sigma}\|y_{ineq} - w_{ineq}\|^2 - q\top y_{ineq} \right\}
\]

As in the equality case, the unconstrained minimizer is:
\[
y_{ineq} ^{\#} = w_{ineq} + \sigma q_{ineq}
\]

But due to the constraint \(y_{ineq} \ge 0\), the optimal solution is the projection of this unconstrained minimizer onto the nonnegative interval:
\[
{y_{ineq \ }^*} = \proj_{y_{ineq} \ge 0} (w_{ineq} + \sigma q_{ineq})
\]

Combining both parts, the full solution is:
\[
y^* = \proj_{\mathcal{Y} } (w + \sigma q), \text{where} \mathcal{Y} = \displaystyle\bigoplus_{k=1}^{m_1}\left\{\mathbb{R}^+ \right\} \oplus \displaystyle\bigoplus_{k=1}^{m_2}\left\{\mathbb{R} \right\}
\]

That is, the projection operator applies elementwise truncation to enforce \(y_i \ge0\) for the inequality components,  and leaves the equality components unchanged.

Therefore the proximal operator becomes:
\[
(I + \sigma \partial F^*)^{-1}(w) = \proj_{\mathcal{Y} } (w + \sigma q)
\]

Now substitute this into the PDHG dual update step:
\[
\begin{aligned}
    y^{n+1} &= (I + \sigma \partial F^*)^{-1}(y^n+\sigma K \bar{x}^{n+1}) \\
    &= \proj_{\mathcal{Y} }(y^n+\sigma K \bar{x}^{n+1} + \sigma q) \\
    &= \proj_{\mathcal{Y} }(y^n+\sigma(q +  K \bar{x}^{n+1}))
\end{aligned}
\]

From the definition of K, we know:
\[
K = \begin{pmatrix}
    -G \\
    -A
\end{pmatrix} = -K',
K^* = K^\top = \begin{pmatrix}
    -G^\top,-A^\top 
\end{pmatrix}
\]

Substituting K into the PDHG for Linear Programing:
\[
\begin{aligned}
    x^{n+1} &= (I + \tau \partial G)^{-1} (x^n - \tau K^* y^n)\\
    &= \proj_{\mathcal{X}}(x^n-\tau(c+ K^*y^n)) \\
    &= \proj_{\mathcal{X}}(x^n-\tau(c+ K^\top y^n)) \\
    &= \proj_{\mathcal{X}}(x^n-\tau(c- K'^\top y^n))
\end{aligned}\\
\begin{aligned}
    \ \ \ \ \ \ y^{n+1} &= (I + \sigma \partial F^*)^{-1}(y^n+\sigma K \bar{x}^{n+1}) \\
    &= \proj_{\mathcal{Y} }(y^n+\sigma(q +  K \bar{x}^{n+1})) \\
    &= \proj_{\mathcal{Y} }(y^n+\sigma(q -  K' \bar{x}^{n+1})) \\
    &= \proj_{\mathcal{Y} }(y^n+\sigma(q -  K' [x^{n+1}+\theta(x^{n+1}-x^n)])))
\end{aligned}
\]



    

\ifthenelse{\boolean{@twoside}}{\myclearpage}{}
\chapter{Convergence analysis for \(\theta = 1\)}\label{appendixD}

\noindent\textbf{Remark on Proof Origin:} 
Our convergence proof is adapted from the analysis framework of Chambolle and Pock~\cite{chambolle2011first}, which studies a general form of the Primal-Dual Hybrid Gradient (PDHG) method.  
The variant of PDHG implemented in the PDLP framework of Applegate \emph{et al.}~\cite{applegate2} differs in certain algorithmic details from the original Chambolle--Pock scheme.  
While the original PDHG proof in~\cite{chambolle2011first} does not directly address this LP-oriented variant, we extend and specialize their technique to establish convergence for the PDLP formulation described in Appendix~\ref{App:DualDerivation}.

As soon as \(B_1 \oplus B_2 \subseteq \mathcal{X} \oplus \mathcal{Y}\) contains a saddle point \((\hat{x},\hat{y})\), we have:

\begin{align*}
    \mathcal{G}_{B_1 \oplus B_2}(x,y) 
    =&\max_{y' \in B_2}\left\{ -\langle y',K'x \rangle +q^\top y' - \delta_{y_{ineq} \ge0}(y
    ) + c^\top x + \delta_{l\le x\le u}(x)\right\} \\
    &- \min_{x' \in B_1}\left\{ -\langle y,K'x' \rangle +q^\top y - \delta_{y_{ineq} \ge0}(y)+ c^\top x' + \delta_{l\le x\le u}(x')\right\} \\
    =&\max_{y' \in B_2}\left\{ -\langle y',K'x \rangle +q^\top y' + c^\top x \right\} - \min_{x' \in B_1}\left\{ -\langle y,K'x' \rangle +q^\top y (y)+ c^\top x' \right\} \\
    =&\max_{y' \in B_2}\left\{ \mathcal{L}(x,y') \right\} - \min_{x' \in B_1}\left\{ \mathcal{L}(x',y) \right\} \\
    \ge& \mathcal{L}(x,\hat{y}) - \mathcal{L}(\hat{x},y) \\
    \ge& 0
\end{align*}

and it vanishes only if \((x,y)\) is itself a saddle point.
\section{Proof of Convergence for LP}
Algorithm 1 when applied to the general LP formulation becomes:
\begin{align*}
    x^{n+1} &= \proj_{\mathcal{X}}(x^n-\tau(c- K'^\top y^n))\\
    y^{n+1} &= \proj_{\mathcal{Y} }(y^n+\sigma[q -  K' (2x^{n+1}-x^n)])
\end{align*}
for some $\tau,\sigma >0$, where $\mathcal{X} = \displaystyle\bigoplus_{k=1}^{n}\left\{[l_k,u_k]\right\}$ and $\mathcal{Y} = \displaystyle\bigoplus_{k=1}^{m_1}\left\{\mathbb{R}^+ \right\} \oplus \displaystyle\bigoplus_{k=1}^{m_2}\left\{\mathbb{R} \right\}$. Sets $\mathcal{X}$ and $\mathcal{Y}$ are closed and convex so these projections are well defined by the Hilbert projection theorem. We will first show that this algorithm does converge to a saddle-point, given a minor condition on the step sizes.
\begin{theorem}
    Assume there exists a saddle point $(\hat{x},\hat{y})$ of the Lagrangian and let $\tau,\sigma >0$ be such that $\tau\sigma\|K'\|^2<1$. Then for $x^n$ and $y^n$ defined by the algorithm above, there exists a saddle point $(x^*,y^*)$ with $x^n \to x^*$ and $y^n \to y^*$.
\end{theorem}
\begin{proof}
    For convenience, notice that the algorithm can be equivalently expressed as
    \begin{align*}
        y^{n} &= \proj_{\mathcal{Y}}(y^{n-1} + \sigma(q-K'\bar{x}))\\
        x^{n+1} &= \proj_{\mathcal{X}}(x^n-\tau(c-K'^\top \bar{y}))
    \end{align*}
    for some $\bar{x} \in X, \bar{y}\in Y$.
    From the characterization of projections onto closed convex sets, we have that for any $x \in X$ , $y \in Y$ and \(n \in \mathbb{N}\) (\(x^0=x^{-1},y^0=y^{-1}\))
    \[
\left\{
\begin{aligned}
    \langle x^n-\tau(c-K'^\top \bar{y})-x^{n+1}, x -x^{n+1}\rangle &\leq 0 \\
    \langle y^{n-1}+\sigma(q-K'\bar{x})-y^{n}, y -y^{n}\rangle &\leq 0
\end{aligned}
\right\}
\]
    from which it follows that     
\[
\begin{aligned}
&\left\{
\begin{aligned}
\langle x^n - \tau c + \tau K'^\top \bar{y} - x^{n+1},\, x - x^{n+1} \rangle &\leq 0 \\
\langle y^{n-1} + \sigma q - \sigma K' \bar{x} - y^{n},\, y - y^{n} \rangle &\leq 0
\end{aligned}
\right. \\[1ex]
&\Rightarrow
\left\{
\begin{aligned}
\langle x^n - x^{n+1} + \tau K'^\top \bar{y},\, x - x^{n+1} \rangle - \tau c^\top x + \tau c^\top x^{n+1} &\leq 0 \\
\langle y^{n-1} - y^{n} - \sigma K' \bar{x},\, y - y^{n} \rangle + \sigma q^\top y - \sigma q^\top y^{n} &\leq 0
\end{aligned}
\right. \\[1ex]
&\Rightarrow
\left\{
\begin{aligned}
\langle x^n - x^{n+1},\, x - x^{n+1} \rangle + \langle \tau K'^\top \bar{y},\, x - x^{n+1} \rangle - \tau c^\top x + \tau c^\top x^{n+1} &\leq 0 \\
\langle y^{n-1} - y^{n},\, y - y^{n} \rangle - \langle \sigma K' \bar{x},\, y - y^{n} \rangle + \sigma q^\top y - \sigma q^\top y^{n} &\leq 0
\end{aligned}
\right. \\[1ex]
&\Rightarrow
\left\{
\begin{aligned}
\langle x^n - x^{n+1},\, x - x^{n+1} \rangle + \langle \tau K'^\top \bar{y},\, x - x^{n+1} \rangle + \tau c^\top x^{n+1} &\leq \tau c^\top x \\
\langle y^{n-1} - y^{n},\, y - y^{n} \rangle - \langle \sigma K' \bar{x},\, y - y^{n} \rangle - \sigma q^\top y^{n} &\leq -\sigma q^\top y
\end{aligned}
\right. \\[1ex]
&\Rightarrow
\left\{
\begin{aligned}
\frac{1}{\tau} \langle x^n - x^{n+1},\, x - x^{n+1} \rangle + \langle K'^\top \bar{y},\, x - x^{n+1} \rangle + c^\top x^{n+1} &\leq c^\top x \\
\frac{1}{\sigma} \langle y^{n-1} - y^{n},\, y - y^{n} \rangle - \langle K' \bar{x},\, y - y^{n} \rangle - q^\top y^{n} &\leq -q^\top y
\end{aligned}
\right. \\[1ex]
&\Rightarrow
\left\{
\begin{aligned}
c^\top x &\geq c^\top x^{n+1} + \langle x - x^{n+1},\, K'^\top \bar{y} \rangle + \frac{1}{\tau} \langle x^n - x^{n+1},\, x - x^{n+1} \rangle \\
-q^\top y &\geq -q^\top y^{n} - \langle K' \bar{x},\, y - y^{n} \rangle + \frac{1}{\sigma} \langle y^{n-1} - y^{n},\, y - y^{n} \rangle
\end{aligned}
\right.
\end{aligned}
\]

    For any \(u,v \in V\), we have:
    \[
    \langle u,v\rangle = \frac{1}{2}(\|u\|^2 + \|v\|^2 -\|u-v\|^2)
    \]
    
    Since:
    \[
    \begin{aligned}
        \langle u, v \rangle &= \langle u, v -u+u \rangle \\
        &=  \langle u, u \rangle + \langle u, v -u \rangle \\
        &= \|u\|^2 + \langle u + v - v, v -u \rangle \\
        &= \|u\|^2 + \langle u - v , v -u \rangle + \langle v, v -u \rangle \\
        &= \|u\|^2 - \langle v-u , v-u \rangle + \langle v, v \rangle + \langle v,-u \rangle \\
        &= \|u\|^2 - \|v-u\|^2 + \|v\|^2 - \langle u, v \rangle \\
    \end{aligned}
    \implies 2\langle u,v\rangle = \|u\|^2 + \|v\|^2 -\|u-v\|^2
    \]

    With this equality, the above inequalities can be written as:

\begin{align*}
&\left\{
\begin{aligned}
c^\top x &\geq c^\top x^{n+1} + \langle x - x^{n+1},\, K'^\top \bar{y} \rangle + \frac{1}{\tau} \langle x^n - x^{n+1},\, x - x^{n+1} \rangle \\
-q^\top y &\geq -q^\top y^{n} - \langle K' \bar{x},\, y - y^{n} \rangle + \frac{1}{\sigma} \langle y^{n-1} - y^{n},\, y - y^{n} \rangle
\end{aligned}
\right. \\[2ex]
\Rightarrow\quad
&\left\{
\begin{aligned}
c^\top x &\geq c^\top x^{n+1} + \langle x - x^{n+1},\, K'^\top \bar{y} \rangle 
+ \frac{1}{\tau} \left( \frac{\|x^n - x^{n+1}\|^2}{2} + \frac{\|x - x^{n+1}\|^2}{2} - \frac{\|x^n - x\|^2}{2} \right) \\
-q^\top y &\geq -q^\top y^{n} - \langle K' \bar{x},\, y - y^{n} \rangle 
+ \frac{1}{\sigma} \left( \frac{\|y^{n-1} - y^{n}\|^2}{2} + \frac{\|y - y^{n}\|^2}{2} - \frac{\|y^{n-1} - y\|^2}{2} \right)
\end{aligned}
\right. \\[2ex]
\Rightarrow\quad
&\left\{
\begin{aligned}
c^\top x &\geq c^\top x^{n+1} + \langle x - x^{n+1},\, K'^\top \bar{y} \rangle 
+ \frac{\|x^n - x^{n+1}\|^2}{2\tau} + \frac{\|x - x^{n+1}\|^2}{2\tau} - \frac{\|x^n - x\|^2}{2\tau} \\
-q^\top y &\geq -q^\top y^{n} - \langle K' \bar{x},\, y - y^{n} \rangle 
+ \frac{\|y^{n-1} - y^{n}\|^2}{2\sigma} + \frac{\|y - y^{n}\|^2}{2\sigma} - \frac{\|y^{n-1} - y\|^2}{2\sigma}
\end{aligned}
\right. \\[2ex]
\Rightarrow\quad
&\left\{
\begin{aligned}
\frac{\|x^n - x\|^2}{2\tau} &\geq -c^\top x + c^\top x^{n+1} + \langle x - x^{n+1},\, K'^\top \bar{y} \rangle 
+ \frac{\|x^n - x^{n+1}\|^2}{2\tau} + \frac{\|x - x^{n+1}\|^2}{2\tau} \\
\frac{\|y^{n-1} - y\|^2}{2\sigma} &\geq q^\top y - q^\top y^{n} - \langle K' \bar{x},\, y - y^{n} \rangle 
+ \frac{\|y^{n-1} - y^{n}\|^2}{2\sigma} + \frac{\|y - y^{n}\|^2}{2\sigma}
\end{aligned}
\right.
\end{align*}

    Summing both inequalities , it follows that 
\[
\begin{aligned}
\Rightarrow \quad
 \frac{\|y - y^{n-1}\|^2}{2\sigma} + \frac{\|x - x^n\|^2}{2\tau} &\geq [q^\top y + c^\top x^{n+1}] - [q^\top y^{n} + c^\top x] \\
& \quad + \frac{\|x^n - x^{n+1}\|^2}{2\tau} + \frac{\|x - x^{n+1}\|^2}{2\tau} + \frac{\|y^{n-1} - y^{n}\|^2}{2\sigma} + \frac{\|y - y^{n}\|^2}{2\sigma} \\
& \quad + \langle x - x^{n+1},\, K'^\top \bar{y} \rangle - \langle K' \bar{x},\, y - y^{n} \rangle
\end{aligned}
\]

\[
\begin{aligned}
\Rightarrow \quad
 \frac{\|y - y^{n-1}\|^2}{2\sigma} + \frac{\|x - x^n\|^2}{2\tau} &\geq [q^\top y + c^\top x^{n+1}] - [q^\top y^{n} + c^\top x] \\
& \quad + \frac{\|x^n - x^{n+1}\|^2}{2\tau} + \frac{\|x - x^{n+1}\|^2}{2\tau} + \frac{\|y^{n-1} - y^{n}\|^2}{2\sigma} + \frac{\|y - y^{n}\|^2}{2\sigma} \\
& \quad + \langle x,\, K'^\top \bar{y} \rangle - \langle x^{n+1},\, K'^\top \bar{y} \rangle - \langle K' \bar{x},\, y \rangle + \langle K' \bar{x},\, y^{n} \rangle
\end{aligned}
\]

\[
\begin{aligned}
\Rightarrow \quad
 \frac{\|y - y^{n-1}\|^2}{2\sigma} + \frac{\|x - x^n\|^2}{2\tau} &\geq [q^\top y + c^\top x^{n+1}] - [q^\top y^{n} + c^\top x] \\
& \quad + \frac{\|x^n - x^{n+1}\|^2}{2\tau} + \frac{\|x - x^{n+1}\|^2}{2\tau} + \frac{\|y^{n-1} - y^{n}\|^2}{2\sigma} + \frac{\|y - y^{n}\|^2}{2\sigma} \\
& \quad + \langle K' x,\, \bar{y} \rangle - \langle K' x^{n+1},\, \bar{y} \rangle - \langle K' \bar{x},\, y \rangle + \langle K' \bar{x},\, y^{n} \rangle
\end{aligned}
\]

We want to construct \([-\langle K'x^{n+1},y\rangle - (-q^\top y) + c^\top x^{n+1}] - [-\langle K'x,y^{n}\rangle - (-q^\top y^{n}) + c^\top x] \) in the inequality, since we have the Lagrangian function: 
\[
\begin{aligned}
    &[-\langle K'x^{n+1},y\rangle - (-q^\top y) + c^\top x^{n+1}] - [-\langle K'x,y^{n}\rangle - (-q^\top y^{n}) + c^\top x] \\
    &= \mathcal{L}(x^{n+1},y)-\mathcal{L}(x,y^n)
\end{aligned}
\]

So our inequality becomes:
\[
\begin{aligned}
 \frac{\|y - y^{n-1}\|^2}{2\sigma} + \frac{\|x - x^n\|^2}{2\tau} &\geq [-\langle K'x^{n+1},y\rangle - (-q^\top y) + c^\top x^{n+1}] - [-\langle K'x,y^{n}\rangle - (-q^\top y^{n}) + c^\top x] \\
& \quad + \frac{\|x^n - x^{n+1}\|^2}{2\tau} + \frac{\|x - x^{n+1}\|^2}{2\tau} + \frac{\|y^{n-1} - y^{n}\|^2}{2\sigma} + \frac{\|y - y^{n}\|^2}{2\sigma} \\
& \quad + \langle K' x,\, \bar{y} \rangle - \langle K' x^{n+1},\, \bar{y} \rangle - \langle K' \bar{x},\, y \rangle + \langle K' \bar{x},\, y^{n} \rangle + \langle K'x^{n+1},y\rangle - \langle K'x,y^{n}\rangle
\end{aligned}
\]

Inner product part of our inequality:
\[
\begin{aligned}
    &\quad \langle K' x,\, \bar{y} \rangle - \langle K' x^{n+1},\, \bar{y} \rangle - \langle K' \bar{x},\, y \rangle + \langle K' \bar{x},\, y^{n} \rangle + \langle K'x^{n+1},y\rangle - \langle K'x,y^{n}\rangle \\
    &= \langle K' (x-x^{n+1}),\, \bar{y} \rangle + \langle K' (x^{n+1}-\bar{x}),\, y \rangle + \langle K' (\bar{x}-x),\, y^{n} \rangle\\
    &= \langle K' (x-x^{n+1}),\, \bar{y} \rangle + \langle K' (x^{n+1}-\bar{x}),\, y \rangle + \langle K' [(\bar{x} - x^{n+1})-(x-x^{n+1})],\, y^{n} \rangle\\
    &= \langle K' (x-x^{n+1}),\, \bar{y} \rangle + \langle K' (x^{n+1}-\bar{x}),\, y \rangle + \langle K' (\bar{x} - x^{n+1}),\, y^{n} \rangle - \langle K'(x-x^{n+1}),\, y^{n} \rangle\\
    &= \langle K' (x-x^{n+1}),\, \bar{y} -y^{n} \rangle + \langle K' (x^{n+1}-\bar{x}),\, y -y^{n} \rangle\\
    &= \langle K' (x^{n+1} - x),\, y^{n} - \bar{y}\rangle - \langle K' (x^{n+1}-\bar{x}),\,y^{n} -y \rangle
\end{aligned}
\]
Hence our inequality becomes:

    \begin{align*}
    \frac{\|y-y^{n-1}\|^2}{2\sigma}+\frac{\|x-x^n\|^2}{2\tau} \geq& \mathcal{L}(x^{n+1},y)- \mathcal{L}(x,y^n)\\ &+ \frac{\|y-y^{n}\|^2}{2\sigma}+\frac{\|x-x^{n+1}\|^2}{2\tau} + \frac{\|y^{n-1}-y^{n}\|^2}{2\sigma}+\frac{\|x^n-x^{n+1}\|^2}{2\tau}\\ &+ \langle K' (x^{n+1} - x),\, y^{n} - \bar{y}\rangle - \langle K' (x^{n+1}-\bar{x}),\,y^{n} -y \rangle.
    \end{align*}
    The last line of this inequality will play an important role in proving convergence of the algorithm. Taking now the values of $\bar{y}$ and $\bar{x}$ as used in the algorithm, $\bar{y}=y^n$ and $\bar{x}=2x^{n}-x^{n-1}$, this last line becomes
    \begin{align*}
        &\langle K' (x^{n+1} - x),\, y^{n} - \bar{y}\rangle - \langle K' (x^{n+1}-\bar{x}),\,y^{n} -y \rangle \\
        &=\left \langle K'(x^{n+1}-x), y^{n}-y^n\right \rangle - \left \langle K'(x^{n+1}-2x^n+x^{n-1}), y^{n}-y\right \rangle \\
        &= \left \langle K'((x^{n+1}-x^n)-(x^n-x^{n-1})), y-y^{n}\right \rangle \\
        &= \left \langle K'(x^{n+1}-x^n), y-y^{n}\right \rangle - \left \langle K'(x^n-x^{n-1}), y-y^{n-1}-y^{n}+y^{n-1}\right \rangle\\
        &= \left \langle K'(x^{n+1}-x^n), y-y^{n}\right \rangle - \left \langle K'(x^n-x^{n-1}), (y-y^{n-1})-(y^{n}-y^{n-1})\right \rangle\\
        &= \left \langle K'(x^{n+1}-x^n), y-y^{n}\right \rangle - \left \langle K'(x^n-x^{n-1}), y-y^{n-1}\right \rangle + \left \langle K'(x^n-x^{n-1}), (y^{n}-y^{n-1})\right \rangle\\
        &\ge \left \langle K'(x^{n+1}-x^n), y-y^{n}\right \rangle - \left \langle K'(x^n-x^{n-1}), y-y^{n-1}\right \rangle - \|K'\|\|x^n-x^{n-1}\|\|y^{n}-y^{n-1}\|\\
    \end{align*}

For any $\alpha >0$, recall the fact that $2ab \leq \alpha a^2+b^2/\alpha$ (since: \(\alpha a^2-2ab+b^2/ \alpha = (\sqrt{\alpha}a-b/\sqrt{\alpha})^2 \ge0\)) for any scalars $a,b$. Letting $\alpha=\sqrt{\sigma/\tau}$, $a=\|x^n-x^{n-1}\|$, and $b=\|y^{n}-y^{n-1}\|$, we get
\begin{align*}
    \|x^n-x^{n-1}\|\|y^{n}-y^{n-1}\| &\leq \frac{\sqrt{\sigma/\tau}}{2} \|x^n-x^{n-1}\|^2+\frac{1}{2\sqrt{\sigma/\tau}}\|y^{n}-y^{n-1}\|^2 \\
    \|K'\|\|x^n-x^{n-1}\|\|y^{n}-y^{n-1}\| &\leq \frac{\|K'\|\sqrt{\sigma\tau}}{2\tau}\|x^n-x^{n-1}\|^2+\frac{\|K'\|\sqrt{\sigma\tau}}{2\sigma}\|y^{n}-y^{n-1}\|^2.
\end{align*}
Combining this with the previous two inequalities reveals that for any $x \in \mathcal{X}$ and $y\in \mathcal{Y}$,
\begin{equation}
\label{N term before summation}
\begin{aligned}
    \frac{\|y-y^{n-1}\|^2}{2\sigma}+\frac{\|x-x^n\|^2}{2\tau} \geq& \mathcal{L}(x^{n+1},y)- \mathcal{L}(x,y^n) + \frac{\|y-y^{n}\|^2}{2\sigma}+\frac{\|x-x^{n+1}\|^2}{2\tau}\\
    &+ \left \langle K'(x^{n+1}-x^n), y-y^{n}\right \rangle - \left \langle K'(x^n-x^{n-1}), y-y^{n-1}\right \rangle\\
    &+ \frac{\|x^n-x^{n+1}\|^2}{2\tau}-\|K'\|\sqrt{\sigma\tau}\frac{\|x^n-x^{n-1}\|^2}{2\tau} +(1-\|K'\|\sqrt{\sigma\tau})\frac{\|y^{n}-y^{n-1}\|^2}{2\sigma}.
\end{aligned}
\end{equation}

Let 

\[
\label{An to Dn}
    \begin{aligned}
A_n &= \frac{\|y-y^{n-1}\|^2}{2\sigma} + \frac{\|x-x^n\|^2}{2\tau} \\
B_n &= (1-\|K'\|\sqrt{\sigma\tau}) \frac{\|y^n - y^{n-1}\|^2}{2\sigma} \\
C_n &= \frac{\|x^n - x^{n-1}\|^2}{2\tau} \\
D_n &= \left\langle K'(x^n - x^{n-1}), y - y^{n-1} \right\rangle
\end{aligned}
\]

Our inequality becomes:
\[
\begin{aligned}
    A_n \geq& \mathcal{L}(x^{n+1},y)- \mathcal{L}(x,y^n)\\ &+ A_{n+1} +B_n +C_{n+1}-\|K'\|\sqrt{\tau\sigma}C_n +D_{n+1}-D_n.
\end{aligned}
\]

\[
\begin{aligned}
    A_n -A_{n+1}  \geq& \mathcal{L}(x^{n+1},y)- \mathcal{L}(x,y^n)\\ & +B_n +C_{n+1}-\|K'\|\sqrt{\tau\sigma}C_n + D_{n+1}-D_n.
\end{aligned}
\]

Now for some integer $N>0$, sum each side of this inequality from $n=0$ to $N-1$, and notice the telescoping cancellation of many of the terms

\noindent\textbf{Derivation of the Fundamental Inequality:}
\begin{align*}
&\sum_{n=0}^{N-1}(A_n - A_{n+1}) \ge 
\sum_{n=0}^{N-1} \left[ \mathcal{L}(x^{n+1},y)- \mathcal{L}(x,y^n) \right] \\
&\quad + \sum_{n=0}^{N-1} B_n + \sum_{n=0}^{N-1} C_{n+1} - \sum_{n=0}^{N-1} \|K'\|\sqrt{\tau\sigma}  C_n + \sum_{n=0}^{N-1} (D_{n+1} - D_n)
\end{align*}

\noindent\textit{Applying telescoping sums and simplifying:}
\begin{align*}
\Rightarrow\quad
A_0 - A_N &\ge 
\sum_{n=0}^{N-1} \left[ \mathcal{L}(x^{n+1},y)- \mathcal{L}(x,y^n) \right] \\
&\quad + \sum_{n=0}^{N-1} B_n + C_N + \sum_{n=0}^{N-2} C_{n+1} - \sum_{n=1}^{N-1} \|K'\|\sqrt{\tau\sigma}  C_n \\
&\quad - \|K'\|\sqrt{\tau\sigma}  C_0 + D_N - D_0
\end{align*}

\noindent\textit{Combining similar terms:}
\begin{align*}
\Rightarrow\quad
&A_0 - A_N \ge 
\sum_{n=0}^{N-1} \left[ \mathcal{L}(x^{n+1},y)- \mathcal{L}(x,y^n) \right] \\
&\quad + \sum_{n=0}^{N-1} B_n + C_N - \|K'\|\sqrt{\tau\sigma}  C_0 \\
&\quad + \sum_{n=0}^{N-2} (1 - \|K'\|\sqrt{\tau\sigma})  C_{n+1} + D_N - D_0
\end{align*}

\noindent\textbf{Final Inequality with Initial Conditions ($x^{-1} = x^0$, $y^{-1} = y^0$):}
\begin{align*}
\Rightarrow\quad
A_0 &\ge A_N +
\sum_{n=0}^{N-1} \left[ \mathcal{L}(x^{n+1},y)- \mathcal{L}(x,y^n) \right] \\
&\quad + (1 - \|K'\|\sqrt{\sigma\tau}) \sum_{n=0}^{N-1} \frac{\|y^n - y^{n-1}\|^2}{2\sigma} + \frac{\|x^N - x^{N-1}\|^2}{2\tau} \\
&\quad + (1 - \|K'\|\sqrt{\tau\sigma}) \sum_{n=0}^{N-2} \frac{\|x^{n+1} - x^n\|^2}{2\tau} \\
&\quad + \langle K'(x^N - x^{N-1}), y - y^{N-1} \rangle
\end{align*}

\noindent\textbf{Bound on the Initial Solution:}
\begin{align*}
&\frac{\|y - y^{0}\|^2}{2\sigma} + \frac{\|x - x^0\|^2}{2\tau} \ge \frac{\|y - y^{N-1}\|^2}{2\sigma} + \frac{\|x - x^N\|^2}{2\tau} \\
&\quad + \sum_{n=0}^{N-1} \left[ \mathcal{L}(x^{n+1},y)- \mathcal{L}(x,y^n) \right] \\
&\quad + (1 - \|K'\|\sqrt{\sigma\tau}) \sum_{n=0}^{N-1} \frac{\|y^n - y^{n-1}\|^2}{2\sigma} + \frac{\|x^N - x^{N-1}\|^2}{2\tau} \\
&\quad + (1 - \|K'\|\sqrt{\tau\sigma}) \sum_{n=0}^{N-2} \frac{\|x^{n+1} - x^n\|^2}{2\tau} \\
&\quad + \langle K'(x^N - x^{N-1}), y - y^{N-1} \rangle
\end{align*}

\noindent\textbf{Simplified Convergence Result:}
\begin{align*}
&\sum_{n=0}^{N-1} \left[ \mathcal{L}(x^{n+1},y)- \mathcal{L}(x,y^n) \right] \\
&\quad + (1 - \|K'\|\sqrt{\sigma\tau}) \sum_{n=0}^{N-1} \frac{\|y^n - y^{n-1}\|^2}{2\sigma} \\
&\quad + (1 - \|K'\|\sqrt{\tau\sigma}) \sum_{n=0}^{N-2} \frac{\|x^{n+1} - x^n\|^2}{2\tau} \\
&\quad + \frac{\|x^N - x^{N-1}\|^2}{2\tau} + \frac{\|y - y^{N-1}\|^2}{2\sigma} + \frac{\|x - x^N\|^2}{2\tau}\\
\leq{}& \frac{\|y - y^{0}\|^2}{2\sigma} + \frac{\|x - x^0\|^2}{2\tau} + \langle K'(x^N - x^{N-1}), y^{N-1} - y \rangle
\end{align*}

Now as before, let \(\alpha=\frac{1}{\tau\|K'\|}\)
\[
    \begin{aligned}
    \langle K'(x^N - x^{N-1}), y^{N-1} - y \rangle &\le \|K'\|\|x^N - x^{N-1}\|\|y^{N-1} - y\|\\
    &\le \frac{\alpha\|K'\|\|x^N - x^{N-1}\|^2}{2} +\frac{\|K'\|\|y^{N-1} - y\|^2}{2\alpha}\\
    &= \frac{\|K'\|\|x^N - x^{N-1}\|^2}{2\|K'\|\tau} +\frac{\|K'\|\tau\|K'\|\|y^{N-1} - y\|^2}{2}\\
    &= \frac{\|x^N - x^{N-1}\|^2}{2\tau} +\frac{\tau\sigma\|K'\|^2\|y^{N-1} - y\|^2}{2\sigma}\\
\end{aligned}
\]

and it follows that
\[
\begin{aligned}
\Rightarrow\quad
&\sum_{n=0}^{N-1} \left[ \mathcal{L}(x^{n+1},y)- \mathcal{L}(x,y^n) \right] \\
&\quad + (1 - \|K'\|\sqrt{\sigma\tau}) \sum_{n=0}^{N-1} \frac{\|y^n - y^{n-1}\|^2}{2\sigma} \\
&\quad + (1 - \|K'\|\sqrt{\tau\sigma}) \sum_{n=0}^{N-2} \frac{\|x^{n+1} - x^n\|^2}{2\tau} \\
&\quad + \frac{\|x^N - x^{N-1}\|^2}{2\tau} + \frac{\|y - y^{N-1}\|^2}{2\sigma} + \frac{\|x - x^N\|^2}{2\tau} \\
&\leq \frac{\|y - y^{0}\|^2}{2\sigma} + \frac{\|x - x^0\|^2}{2\tau} + \frac{\|x^N - x^{N-1}\|^2}{2\tau} + \frac{\tau\sigma\|K'\|^2\|y^{N-1} - y\|^2}{2\sigma}
\end{aligned}
\]
\begin{equation}
\label{eq: boundedness}
\begin{aligned}
\Rightarrow\quad
&\sum_{n=0}^{N-1} \left[ \mathcal{L}(x^{n+1},y)- \mathcal{L}(x,y^n) \right] \\
&\quad + (1 - \|K'\|\sqrt{\sigma\tau}) \sum_{n=0}^{N-1} \frac{\|y^n - y^{n-1}\|^2}{2\sigma} \\
&\quad + (1 - \|K'\|\sqrt{\tau\sigma}) \sum_{n=0}^{N-2} \frac{\|x^{n+1} - x^n\|^2}{2\tau} \\
&\quad + (1 - \tau\sigma\|K'\|^2) \frac{\|y - y^{N-1}\|^2}{2\sigma} + \frac{\|x - x^N\|^2}{2\tau} \\
&\leq \frac{\|y - y^{0}\|^2}{2\sigma} + \frac{\|x - x^0\|^2}{2\tau}
\end{aligned}
\end{equation}

First we choose \((x,y) = (\hat{x},\hat{y})\) a saddle point in \eqref{eq: boundedness}, then it follows from \eqref{eq:saddle-point} that the first summation \(\sum_{n=0}^{N-1} \left[ \mathcal{L}(x^{n+1},\hat{y})- \mathcal{L}(\hat{x},y^n) \right]\) is non-negative, and using our assumption \(0 <1-\tau\sigma\|K'\|^2\), it follows that:
\[
\begin{aligned}
\Rightarrow\quad
&\quad  (1 - \tau\sigma\|K'\|^2) \frac{\|\hat{y} - y^{N-1}\|^2}{2\sigma} + \frac{\|\hat{x} - x^N\|^2}{2\tau} \\
&\leq \frac{\|\hat{y} - y^{0}\|^2}{2\sigma} + \frac{\|\hat{x} - x^0\|^2}{2\tau}
\end{aligned}
\]

where \(\frac{\|\hat{y} - y^{0}\|^2}{2\sigma} + \frac{\|\hat{x} - x^0\|^2}{2\tau}\) is a constant, it follows that the sequence \((x^N,y^N)\) is bounded in \(X \oplus Y\).

Let \(x_N = \sum_{n=1}^{N}{x^n}/N, y_N = \sum_{n=0}^{N-1}{y^n}/N\). Then from inequality \eqref{eq: boundedness}. we obtain:
\[
\begin{aligned}
    \frac{\|y - y^{0}\|^2}{2\sigma} + \frac{\|x - x^0\|^2}{2\tau} \ge &\sum_{n=0}^{N-1} \left[\mathcal{L}(x^{n+1},y)- \mathcal{L}(x,y^n) \right]\\
    =&\sum_{n=0}^{N-1} \left[ \left( c^\top x^{n+1} - \langle K'x^{n+1}, y \rangle + q^\top y \right) 
- \left( c^\top x - \langle K'x, y^n \rangle + q^\top y^n \right) \right]\\
=& [Nc^\top x_N - \langle NK'x_N, y \rangle + Nq^\top y] - [N c^\top x - \langle K'x, Ny_N \rangle + Nq^\top y_N ]
\end{aligned}
\]
\begin{equation}
\label{eq:bounded_xNyN}
\implies \frac{1}{N}(\frac{\|y - y^{0}\|^2}{2\sigma} + \frac{\|x - x^0\|^2}{2\tau}) \ge  [c^\top x_N - \langle K'x_N, y \rangle + q^\top y] - [c^\top x - \langle K'x, y_N \rangle + q^\top y_N ]
\end{equation}
for any \((x,y) \in X \oplus Y\). Which yields:

\[
\begin{aligned}
\mathcal{G}_{B_1 \oplus B_2}(x_N,y_N) &= \max_{y \in B_2}\left\{ \langle y,Kx_N \rangle - F^*(y)+G(x_N)\right\} - \min_{x \in B_1}\left\{ \langle y_N,Kx \rangle - F^*(y_N)+G(x)\right\} \\
&= \max_{y \in B_2}\left\{ c^\top x_N - \langle K'x_N, y \rangle + q^\top y\right\} - \min_{x \in B_1}\left\{ c^\top x - \langle K'x, y_N \rangle + q^\top y_N\right\}\\
&=\max_{y \in B_2}\left\{ c^\top x_N - \langle K'x_N, y \rangle + q^\top y\right\} + \max_{x \in B_1}\left\{-[ c^\top x - \langle K'x, y_N \rangle + q^\top y_N\right]\}\\
&=\max_{(x,y) \in B_1\oplus B_2}\left\{ [c^\top x_N - \langle K'x_N, y \rangle + q^\top y] -[ c^\top x - \langle K'x, y_N \rangle + q^\top y_N\right]\}\\
&\le \frac{1}{N}[\max_{(x,y) \in B_1\oplus B_2}\{\frac{\|y - y^{0}\|^2}{2\sigma} + \frac{\|x - x^0\|^2}{2\tau}\}]
\end{aligned}
\]

Since \(X \oplus Y\) is a finite dimensional Hibert space, and \(\{(x_n,y_n\}_{n=1}^{\infty} \in X \oplus Y\) is a bounded sequence by the sequence \((x^n,y^n)\) is bounded in \(X \oplus Y\). So there exist a point \((x^*,y^*) \in X \oplus Y\), and a subsequence \(\{(x_{N_k},y_{N_k})\}_{k=1}^{\infty}\), such that: 
\[
(x_{N_k},y_{N_k}) \to (x^*,y^*) \text{ as }k\to \infty 
\]
\((x^*,y^*)\) also called \(\textit{weak cluster point}\).

And it follows from \eqref{eq:bounded_xNyN} that: for any \(k \in \mathbb{N}\)
\[
\frac{1}{{N_k}}(\frac{\|y - y^{0}\|^2}{2\sigma} + \frac{\|x - x^0\|^2}{2\tau}) \ge  [c^\top x_{N_k} - \langle K'x_{N_k}, y \rangle + q^\top y] - [c^\top x - \langle K'x, y_{N_k} \rangle + q^\top y_{N_k} ]
\]

Since \((x_{N_k},y_{N_k}) \to (x^*,y^*) \text{ as }k\to \infty \), and \(c^\top (\cdot), q^\top (\cdot),  \langle K'\cdot, \cdot \rangle \) are continuous functions. we get:
\[
\lim_{k\to \infty}\left\{ \frac{1}{{N_k}}(\frac{\|y - y^{0}\|^2}{2\sigma} + \frac{\|x - x^0\|^2}{2\tau}) \right \} \ge  \lim_{k\to \infty}\left\{ [c^\top x_{N_k} - \langle K'x_{N_k}, y \rangle + q^\top y] - [c^\top x - \langle K'x, y_{N_k} \rangle + q^\top y_{N_k} ] \right \}
\]
\[
\begin{aligned}
    \implies 0\ge& [c^\top x^* - \langle K'x^*, y \rangle + q^\top y] - [c^\top x - \langle K'x, y^* \rangle + q^\top y^* ]\\
    &= \mathcal{L}(x^*,y) - \mathcal{L}(x,y^*)
\end{aligned}
\]

This show that \((x^*,y^*)\) satisfy \eqref{eq:saddle-point} and therefore is a saddle point. We have show the weak cluster points of \((x_N,y_N)\) are saddle points.

It remains to prove the convergence to a saddle-point of the whole sequence \((x^n,y^n)\).  \((x^n,y^n)\) is a bounded sequence, so there are some subsequences \((x^{n_k},y^{n_k})\) converges to some limit \((x^*,y^*)\). Observe that \eqref{eq: boundedness} implies that: \(\lim_{n \to \infty} (x^n-x^{n-1}) = \lim_{n \to \infty} (y^n-y^{n-1}) = 0\)

In particular also \(x^{n_k-1}\) and \(y^{n_k-1}\) converge respectively to \(x^*\) and \(y^*\). It follows that the limit \((x^*,y^*)\) is a fixed point of our algorithm, hence a saddle point of our problem. We can take \((x,y) =(x^*,y^*)\) in \eqref{N term before summation}, which we sum from \(n=n_k\) to \(N-1, n>n_k\). Using the same notation in \eqref{An to Dn}, We obtain:
\begin{align*}
&\sum_{n=n_k}^{N-1}(A_n - A_{n+1}) \ge 
\sum_{n=n_k}^{N-1} \left[ \mathcal{L}(x^{n+1},y^*)- \mathcal{L}(x^*,y^n) \right] \\
&\quad + \sum_{n=n_k}^{N-1} B_n + \sum_{n=n_k}^{N-1} C_{n+1} - \sum_{n=n_k}^{N-1} \|K'\|\sqrt{\tau\sigma}  C_n + \sum_{n=n_k}^{N-1} (D_{n+1} - D_n)
\end{align*}

Using the same idea as before, we get:
\begin{align*}
\Rightarrow\quad
A_{n_k} - A_N &\ge 
\sum_{n=n_k}^{N-1} \left[ \mathcal{L}(x^{n+1},y^*)- \mathcal{L}(x^*,y^n) \right] \\
&\quad + \sum_{n=n_k}^{N-1} B_n + C_N - \|K'\|\sqrt{\tau\sigma}  C_{n_k} \\
&\quad + \sum_{n=n_k}^{N-2} (1 - \|K'\|\sqrt{\tau\sigma})  C_{n+1} + D_N - D_{n_k}
\end{align*}

Since \((x^*,y^*)\) is a saddle point, 
\[\left[ \mathcal{L}(x^{n+1},y^*)- \mathcal{L}(x^*,y^n) \right] \ge 0\]

\[
\begin{aligned}
\Rightarrow\quad
&A_{n_k} \ge 
A_N \\
&\quad + \sum_{n=n_k}^{N-1} B_n + C_N - \|K'\|\sqrt{\tau\sigma}  C_{n_k} \\
&\quad + \sum_{n=n_k}^{N-2} (1 - \|K'\|\sqrt{\tau\sigma})  C_{n+1} + D_N - D_{n_k}
\end{aligned}
\]

\[
\begin{aligned}
\Rightarrow\quad
&\frac{\|y^*-y^{{n_k}-1}\|^2}{2\sigma} + \frac{\|x^*-x^{n_k}\|^2}{2\tau} \ge 
\frac{\|y^*-y^{N-1}\|^2}{2\sigma} + \frac{\|x^*-x^N\|^2}{2\tau} \\
&\quad + \frac{\|x^N- x^{N-1}\|^2}{2\tau} - \|K'\|\sqrt{\tau\sigma}  \frac{\|x^{n_k} - x^{{n_k}-1}\|^2}{2\tau} \\
&\quad +  \left\langle K'(x^N - x^{N-1}), y^* - y^{N-1} \right\rangle -  \left\langle K'(x^{n_k} - x^{{n_k}-1}), y^* - y^{{n_k}-1} \right\rangle
\end{aligned}
\]

Let \(k \to \infty\), since \(N>n_k\), and \(x^{n_k} \to x^*,y^{n_k} \to y^*, x^N-x^{N-1}\to0,y^N-y^{N-1}\to0\)
\[
\begin{aligned}
\Rightarrow\quad
&0 \ge 
\lim_{N \to \infty} \frac{\|y^*-y^{N-1}\|^2}{2\sigma} + \frac{\|x^*-x^N\|^2}{2\tau}
\end{aligned}
\]
from which we easily deduce that \(x^N \to x^*, y^N \to y^*\) as \(N\to \infty\).
\end{proof}

\end{document}